\documentclass[11pt,reqno,twoside]{amsart}
\usepackage{graphicx}
\usepackage{amssymb}
\usepackage{epstopdf}
\usepackage[asymmetric,top=2.5cm,bottom=2.8cm,left=2.38cm,right=2.38cm]{geometry}
\usepackage{hyperref}
\usepackage{booktabs} 
\usepackage{array}
\usepackage{paralist} 
\usepackage{verbatim} 
\usepackage{subfig}
\usepackage{tabularx}
\usepackage{amsmath,amsfonts,amsthm,mathrsfs,amssymb,cite}
\usepackage{algorithm}
\usepackage{algpseudocode}
\usepackage[usenames]{color}
\usepackage{bm}
\usepackage{color}

\usepackage{xcolor}
\usepackage{subcaption}
\graphicspath{ {./picture/} } 
\hypersetup{
	colorlinks,
	linkcolor={blue!80!black},
	urlcolor={blue!80!black},
	citecolor={blue!80!black}
}

\newtheorem{thm}{Theorem}[section]

\newtheorem{lem}{Lemma}[section]

\theoremstyle{definition}
\newtheorem{defn}{Definition}[section]
\newtheorem{example}{Example}[section]
\theoremstyle{remark}

\newtheorem{rem}{Remark}[section]
\numberwithin{equation}{section}

\newcommand{\bmf}[1]{{\mathbf{#1}}}

\usepackage{xstring} 

\newcommand{\AuthorInfo}[4]{ \textsc{#1} \IfStrEq{#4}{true}{$^{*}$}{}\\ #2\\ \textit{E-mail:} \texttt{#3} \IfStrEq{#4}{true}{\\\textit{$^{*}$Corresponding author.}}{} \par\vspace{0.5em} }

\usepackage{authblk}

\title[Monotonicity spectral sampling methods]{Robust shape reconstruction of elastic impenetrable scatterers via monotonicity spectral sampling methods}

\date{} 
\begin{document}
	\maketitle
	
	\begin{center}
		\bigskip
		\footnotesize
		
		\AuthorInfo{Mengjiao Bai}{School of Mathematics, Jilin University,
			Changchun, Jilin 130012, China.}{{baimj24@mails.jlu.edu.cn}}{false}
		
		\AuthorInfo{Huaian Diao}{School of Mathematics and Key Laboratory of Symbolic Computation and Knowledge Engineering of Ministry of Education, Jilin University, Changchun, Jilin, China.}{diao@jlu.edu.cn, hadiao@gmail.com}{true}
		
		\AuthorInfo{Weisheng Zhou}{School of Mathematics, Jilin University,
			Changchun, Jilin 130012, China.}{zhouws24@mails.jlu.edu.cn}{false}
	\end{center}

	\normalsize

\begin{abstract}	
Reconstructing the location and shape of an unknown impenetrable scatterer from far-field measurements is a fundamental inverse problem in elastic scattering. In this paper, we propose monotonicity-based shape characterization theorems and develop corresponding algorithms for rigid and traction-free impenetrable scatterers. By establishing the factorization of the elastic far-field operator and constructing localized wave functions, we derive a sharp monotonicity-based characterization criterion for determining the shape and position of the impenetrable scatterer. This criterion is based on the spectral properties of the \emph{monotonicity operator}, defined as a specific linear combination of the far-field and Herglotz probing operators.
Building on this theoretical foundation, we first present a counting-based monotonicity sampling method that evaluates the number of negative eigenvalues of the monotonicity operator. To address the inherent sensitivity of eigenvalue-counting to measurement noise, we further develop two novel monotonicity spectral sampling algorithms that exploit the magnitudes, rather than merely the signs, of the negative eigenvalues. The single-frequency monotonicity spectral sampling method provides robust stability against data perturbations, while the multi-frequency monotonicity spectral sampling method extension aggregates frequency information into a multiscale indicator that balances noise robustness with high-resolution geometric fidelity. Numerical experiments across various scatterer geometries and noise levels demonstrate sharp boundary localization and accurate reconstruction of complex concave features, confirming the effectiveness of the single-frequency and multi-frequency monotonicity spectral sampling methods.

    

	\end{abstract}

\section{Introduction}
	
\subsection{Mathematical setup}

Determining the location and shape of an unknown scatterer from measurements of the scattered wave field far away from the region of interest is a fundamental inverse problem in wave propagation. This challenge is central to numerous scientific and engineering fields, ranging from \emph{seismic imaging} for geophysical exploration and crustal monitoring to the \emph{high-precision nondestructive evaluation} (NDE) of aerospace components and civil infrastructures \cite{OB04,ABGKLW15}. In elastic scattering, this problem is particularly challenging due to the coexistence of compressional (P-) and shear (S-) wave modes, which travel at different velocities and interact with the scatterer's geometry in distinct ways.  The synergistic use of these complementary modes is crucial for resolving fine-scale geometric features that might be invisible to a single wave type, yet it requires a robust mathematical framework to decouple and interpret the multimodality of the elastic far-field data. 

This paper addresses the inverse problem of reconstructing a two-dimensional impenetrable elastic scatterer $D$ from far-field measurements under two boundary conditions. To this end, we consider a bounded Lipschitz domain $D \subset \mathbb{R}^2$ with a connected exterior, embedded in a homogeneous isotropic elastic background characterized by the Lam\'{e} constants $\lambda$ and $\mu$ satisfying the \emph{strong ellipticity conditions}:
\begin{equation}\notag
	\mu > 0 \quad 	\text{and} \quad \lambda + \mu > 0.
\end{equation} 
Without loss of generality, we assume the background homogeneous  medium has a normalized unit density. The impenetrable scatterer $D$ is illuminated by an incident time-harmonic elastic plane wave $\mathbf{u}^i$, defined as a superposition of a compressional (P-) wave and a shear (S-) wave:
\begin{equation}\label{eq:ui}
	\mathbf{u}^i(\mathbf{x}) = a_p \mathbf{d} e^{\mathrm{i} k_p \mathbf{d} \cdot \mathbf{x}} + a_s \mathbf{d}^\perp e^{\mathrm{i} k_s \mathbf{d} \cdot \mathbf{x}},
\end{equation}
where $\mathbf{d}\in\mathbb{S}$ denotes the propagation direction, $\mathbf{d}^\perp$ its orthogonal direction, $a_p,a_s\in\mathbb{C}$ are complex amplitudes satisfying $|a_p|+|a_s|>0$, and the wavenumbers are defined by
\begin{equation}\notag
	k_p = \frac{\omega}{\sqrt{\lambda + 2\mu}}, \quad k_s = \frac{\omega}{\sqrt{\mu}},
\end{equation}
with $\omega>0$ the angular frequency. The incident field $\mathbf{u}^i$ satisfies the Navier equation in $\mathbb{R}^2$:
\begin{equation}\notag
	\Delta^*\mathbf{u}^i + \omega^2 \mathbf{u}^i = \mathbf{0}.
\end{equation}
where \(\Delta^*\mathbf{u}^i:= \mu \Delta \mathbf{u}^i + (\lambda + \mu) \nabla( \nabla \cdot \mathbf{u}^i)\). The interaction of the incident wave $\mathbf{u}^i$ with the impenetrable scatterer $D$ gives rise to a scattered field $\mathbf{u}$, which satisfies the Navier equation in the exterior domain
\begin{equation}\label{eq:syst1}
	\Delta^* \mathbf{u} + \omega^2 \mathbf{u} = \mathbf{0}
	\quad \text{in } \mathbb{R}^2 \setminus \overline{D}.
\end{equation}

The total field is defined by
$$
\mathbf{u}^{\mathrm{tot}} := \mathbf{u}^i + \mathbf{u},
$$
and is required to satisfy the boundary condition
\begin{equation}\label{eq:BC}
	\mathcal{B} \mathbf{u}^{\mathrm{tot}} = \mathbf{0}
	\quad \text{on } \partial D,
\end{equation}
where $\mathcal{B}$ denotes the boundary operator. In this paper, we consider the following two types of boundary conditions imposed on the impenetrable scatterer:

\begin{itemize}
	\item[(1)] Dirichlet boundary condition
	\begin{equation}\label{eq:DirichletBC}
		\mathbf{u}^{\mathrm{tot}} = \mathbf{0}
		\quad \text{on } \partial D.
	\end{equation}
	In this case, $D$ is referred to as a \emph{rigid impenetrable scatterer or obstacle}. Physically, this condition characterizes a hard inclusion or a second-phase particle embedded within an elastic matrix, where the high stiffness of the inclusion relative to the background leads to vanishing total displacement at the interface.
	
	\item[(2)] Neumann boundary condition
	\begin{equation}\label{eq:NeumannBC}
		\mathbf{T}_{\mathbf{\nu}} \mathbf{u}^{\mathrm{tot}} = \mathbf{0}
		\quad \text{on } \partial D.
	\end{equation}
	Here, the traction operator $\mathbf{T}_{\mathbf{\nu}}$ is defined by
	$$
	\mathbf{T}_{\mathbf{\nu}} \mathbf{u}^{\mathrm{tot}}
	:=
	2\mu \partial_{\mathbf{\nu}} \mathbf{u}^{\mathrm{tot}}
	+ \lambda \mathbf{\nu} \nabla \cdot \mathbf{u}^{\mathrm{tot}}
	- \mu \mathbf{\tau} \, \mathrm{curl}\,\mathbf{u}^{\mathrm{tot}},
	$$
	where
	$$
	\mathrm{curl}\,\mathbf{u}^{\mathrm{tot}}
	=
	\mathrm{curl}\,(u_1^{\mathrm{tot}},u_2^{\mathrm{tot}})^\top
	:=
	\partial_1 u_2^{\mathrm{tot}} - \partial_2 u_1^{\mathrm{tot}},
	$$
    and $\bm{\nu}=(\nu_1,\nu_2)^{\top}$ denote the unit outward normal vector to $\partial D$. 
	From a physical standpoint, this boundary condition models the scattering of elastic waves by structural defects such as cracks, voids, or internal delaminations. In these scenarios, the vanishing of surface traction causes the boundary to act as a perfect reflector of elastic wave energy; hence, $D$ is referred to as a \emph{traction-free impenetrable scatterer or cavity}.
\end{itemize}

The scattered field $\mathbf{u}$ can be decomposed into its compressional and shear components via the Helmholtz decomposition:
\begin{equation}\notag
    \mathbf{u} = \mathbf{u}_p + \mathbf{u}_s,
\end{equation}
where the compressional (longitudinal) part $\mathbf{u}_p$ and the shear (transverse) part $\mathbf{u}_s$ are defined respectively by
\begin{equation}\notag 
    \mathbf{u}_p = -\frac{1}{k_p^2} 
\nabla ( 
\nabla \cdot \mathbf{u} ), \quad \mathbf{u}_s = \frac{1}{k_s^2} \mathbf{curl} ( \mathrm{curl} \, \mathbf{u} ).
\end{equation}
Here, the two-dimensional scalar curl of a vector field $\mathbf{w} = (w_1, w_2)^\top$ is defined as $\mathrm{curl} \, \mathbf{w} := \partial_1 w_2 - \partial_2 w_1$, while the vector curl of a scalar function $w$ is defined as $\mathbf{curl} \, w := (\partial_2 w, -\partial_1 w)^\top$.

To guarantee the uniqueness of the solution to \eqref{eq:syst1} and to ensure that the scattered wave propagates strictly outward to infinity, the scattered field $\mathbf{u}$ is required to be outgoing, satisfying the Kupradze radiation condition:
\begin{equation}\label{eq:RadiationCondition}
	\lim_{r \to \infty} r^{1/2} \left( \frac{\partial \mathbf{u}_p}{\partial r} - {i} k_p \mathbf{u}_p \right) = 0,
	\quad
	\lim_{r \to \infty} r^{1/2} \left( \frac{\partial \mathbf{u}_s}{\partial r} - {i} k_s \mathbf{u}_s \right) = 0,
\end{equation}
where $i=\sqrt{-1}$ and $r:=|\mathbf{x}|$.

Given an incident field $\mathbf{u}^i$ and an impenetrable scatterer $D$, the direct (forward) scattering problem is to find the scattered field $\mathbf{u}$ satisfying the governing Navier equation \eqref{eq:syst1}, the boundary condition \eqref{eq:BC}, and the Kupradze radiation condition \eqref{eq:RadiationCondition}. Under these conditions, the direct scattering problem is well-posed; indeed, it admits a unique weak solution $\mathbf{u} \in [H^1_{\mathrm{loc}}(\mathbb{R}^2 \setminus \overline{D})]^2$ \cite{LWWZ16, KGBB79}. Furthermore, the scattered field $\mathbf{u}$ exhibits the following asymptotic behavior as $r=|\mathbf x|\rightarrow \infty$:
\begin{equation}\label{eq:asymptotic}
    \mathbf{u}(\mathbf{x})
    =
    \frac{e^{\mathrm{i}k_p r}}{\sqrt{r}} \mathbf{u}_p^\infty(\hat{\mathbf{x}})
    +
    \frac{e^{\mathrm{i}k_s r}}{\sqrt{r}} \mathbf{u}_s^\infty(\hat{\mathbf{x}})
    +
    O\left( r^{-3/2} \right),
    \quad r := |\mathbf{x}| \to \infty,
\end{equation}
which holds uniformly in all observation directions $\hat{\mathbf{x}} := \mathbf{x}/|\mathbf{x}| \in \mathbb{S}^1$. Here, the unit circle in $\mathbb{R}^2$ is denoted by
\begin{equation*}
    \mathbb{S} := \left\{ \hat{\mathbf{x}} \in \mathbb{R}^2 : |\hat{\mathbf{x}}| = 1 \right\}.
\end{equation*}
The fields $\mathbf{u}_p^\infty$ and $\mathbf{u}_s^\infty$ are analytic vector-valued functions on $\mathbb{S}$, representing the compressional and shear  far-field patterns of $\mathbf{u}$, respectively. Due to the orthogonal nature of the Helmholtz decomposition, $\mathbf{u}_p^\infty(\hat{\mathbf{x}})$ is collinear with $\hat{\mathbf{x}}$, whereas $\mathbf{u}_s^\infty(\hat{\mathbf{x}})$ is orthogonal to $\hat{\mathbf{x}}$. Consequently, we can represent the combined far-field pattern as a vector-valued function $(\mathbf{u}_p^\infty, \mathbf{u}_s^\infty) \in [L^2(\mathbb{S})]^2$. In the sequel, we denote the total far-field pattern associated with the scattered field induced by the incident plane wave \eqref{eq:ui} as $\mathbf{u}^\infty(\cdot, \mathbf{d}, a_p, a_s):=(\mathbf{u}_p^\infty, \mathbf{u}_s^\infty)$.

The inverse problem under consideration is to reconstruct the boundary $\partial D$ (characterizing both the location and shape of the scatterer $D$) from the measured far-field pattern $\mathbf{u}^\infty(\cdot, \mathbf{d}, a_p, a_s)$ for all incident directions $\mathbf{d} \in \mathbb{S}$. Depending on the available data, this reconstruction is performed either at a fixed angular frequency $\omega > 0$ or across a finite range of frequencies $\omega \in [\omega_{\min}, \omega_{\max}]$ with $0 < \omega_{\min} < \omega_{\max}$. Assuming that the physical nature of the boundary is known \emph{a priori}, we investigate two canonical configurations:
\begin{itemize}
    \item \textbf{Inverse Problem 1 (IP1)}: Reconstruct $\partial D$ assuming $D$ is a rigid impenetrable scatterer satisfying the Dirichlet condition \eqref{eq:DirichletBC}.
    \item \textbf{Inverse Problem 2 (IP2)}: Reconstruct $\partial D$ assuming $D$ is a traction-free scatterer satisfying the Neumann condition \eqref{eq:NeumannBC}.
\end{itemize}
The severe non-linearity and ill-posedness inherent in inverse obstacle scattering present significant computational challenges, as small, unavoidable measurement perturbations in the far-field patterns can trigger massive instabilities in the reconstructed boundaries. This structural sensitivity underscores the necessity of developing robust and noise-tolerant numerical algorithms. To this end, this work develops monotonicity spectral sampling methods capable of exploiting both single-frequency and multi-frequency data, effectively stabilizing the reconstruction process while achieving high-fidelity boundary resolution under high noise levels.


\subsection{Connections to existing results and main contributions}

The uniqueness of inverse scattering problems has been extensively investigated over the past few decades; see \cite{CK19, DFL26, DGLY23, DGT25, DLM25, KS14, KS15, GS12, HKS12} and the references therein. In the context of elastic scattering, it is well-established that if the physical boundary type of the impenetrable scatterer (i.e., rigid or traction-free) is known \emph{a priori}, its shape is uniquely determined by the far-field patterns at a single frequency for all incident directions \cite{KS14, KS15, GS12, HKS12}. Remarkably, this unique determination remains valid even when the measurements are restricted to either the compressional (P-) or shear (S-) component of the far-field data (see \cite{GS12,KS15}).

For numerical reconstruction methods applied to inverse scattering problems, extensive numerical algorithms have been proposed to tackle the nonlinearity and ill-posedness inherent in the corresponding inverse problem (cf. \cite{CK19}). For inverse elastic impenetrable scattering, optimization-based iterative methods for determining the shape of the scatterer have been developed (see \cite{LWWZ16,DLL19,LY19}); these approaches typically incur significant computational cost due to the need to solve the associated direct scattering problem. Sampling-type methods (see \cite{Aren01,JLX18,LLS19,KG08,AK02,HKS12})  aim to image the boundary of the impenetrable scatterer by evaluating suitable imaging functionals at discrete sampling points. Furthermore, reverse time migration has been studied for inverse elastic impenetrable scattering (cf. \cite{CH15}). Multi-frequency based reconstruction methods have been shown to improve both the stability and resolution of recovering the scatterer; see \cite{BLLT15,BYZ19} for further discussion.

Unlike classical pointwise sampling-type methods \cite{CK19,KG08}, the monotonicity method operates as a domain-sampling approach. This paradigm was originally formulated for electrical impedance tomography (EIT) based on the monotonicity of the resistance matrix \cite{TR02}, and was subsequently investigated  for EIT by exploiting the monotonicity of the Neumann-to-Dirichlet (NtD) operator with respect to conductivity \cite{BU13}. In recent years, the monotonicity paradigm has been successfully extended to reconstruct elastic inclusions within bounded domains \cite{EH21, EH23, EP24}. These formulations establish rigorous monotonicity relations for the NtD operator with respect to the Lam\'{e} parameters and density, thereby enabling standard and linearized schemes to recover the support of internal elastic anomalies from boundary measurements.

Recently, monotonicity-based sampling methods have been developed for inverse acoustic and electromagnetic scattering in unbounded domains \cite{GB18,AG20,AG23}. However, time-harmonic elastic scattering presents substantial additional complexity due to the coexistence of P- and S-wave modes. These modes are governed by distinct wavenumbers and radiation conditions, and each contributes to the factorization of the far-field operator in a structurally unique way. This inherent coupling between wave modes introduces mathematical and numerical challenges. Although the monotonicity-based shape characterization for rigid impenetrable elastic scatterers was initiated in our previous work \cite{BDZ25}, and a recent framework for penetrable elastic media has been established in \cite{HX26}, a unified monotonicity-based sampling method that encompasses both rigid (Dirichlet) obstacles and traction-free (Neumann) cavities using far-field measurements has yet to be established. 

Furthermore, the classical monotonicity-based sampling method relies on counting the negative eigenvalues of the linear combination of specific operators, such as the far-field operator and the probing operator \cite{AG20, AG23}, hereafter referred to as the \emph{monotonicity operator}.  However, this approach  compresses the entire negative spectrum information into a single integer, discarding all magnitude information in favor of binary sign data. Meanwhile, since the underlying monotonicity operator is compact and self-adjoint, its eigenvalues accumulate at zero. Even arbitrary small data perturbations can flip the signs of numerous eigenvalues near the origin, making the counting map intrinsically discontinuous. This structural fragility renders the classical counting-based monotonicity sampling method highly sensitive to measurement noise, severely limiting its practical utility in realistic, noise-contaminated scattering scenarios. The numerical instability of classical counting-based sampling method represents a fundamental bottleneck that limits their applicability in practical inverse problems. To overcome this bottleneck, in this paper we propose novel spectral sampling methods for reconstructing rigid (Dirichlet) and traction-free (Neumann) impenetrable scatterers. By exploiting the continuous magnitude information of the eigenvalues of the monotonicity operator, rather than relying on a discontinuous, binary sign-counting of its spectrum, the proposed algorithms achieve exceptional robustness against measurement noise.

The main contributions of this work are twofold:
\begin{itemize}
    \item[(1)] 
  We establish sharp, monotonicity-based criteria for determining the shape and location of a Lipschitz impenetrable scatterer $D$ under either rigid (Dirichlet) or traction-free (Neumann) boundary conditions, as rigorously formulated in Theorems~\ref{Characterofimpenetrable scatterer} and~\ref{Characterofimpenetrable scatterer2}. Specifically, the family of monotonicity operators $\mathcal{A}_B$ associated with a probing domain $B$ serves as a rigorous mathematical test to verify the spatial inclusion $B \subseteq D$ (see Remark~\ref{rem:41} for a detailed discussion on this characterization).  To numerically realize these criteria, we first present a counting-based monotonicity sampling method, detailed in Algorithm~\ref{alg:discrete monotonicity}, which reconstructs the shape of $D$ utilizing single-frequency far-field data across all incident directions.

    \item[(2)] 
As discussed earlier, to resolve the critical computational bottleneck of the counting-based monotonicity sampling method, namely, the extreme sensitivity of the monotonicity operator's eigenvalue signs to arbitrarily small, unavoidable measurement perturbations which destabilizes the resulting indicator function, we introduce two novel monotonicity spectral sampling algorithms that  exploit the underlying eigenvalue \emph{magnitudes}. The single-frequency monotonicity spectral sampling method (Algorithm~\ref{alg:spectrum monotonicity}) employs a spectral indicator function that sums the negative eigenvalues of the discretized  the monotonicity operator. The multi-frequency monotonicity spectral sampling method (Algorithm~\ref{alg:multi-frequency spectrum monotonicity})  is obtained by taking a weighted combination of the single-frequency indicator functions across a frequency range, yielding a multiscale reconstruction that balances low-frequency stability with high-frequency resolution. Due to the continuous dependence of eigenvalues on perturbations, these indicator functions using the spetrum of the discretized  monotonicity operator are Lipschitz-continuous with respect to discretized  far-field data. These methods offer superior robustness and geometric fidelity compared to the counting-based monotonicity sampling method, particularly for non-convex scatterers, as demonstrated by extensive numerical experiments.
\end{itemize}

The remainder of this paper is organized as follows. Section~\ref{sec:far-field Operator decomposition} introduces the factorization of the far-field operator with some properties for rigid and traction-free impenetrable scatterers, introduces the relevant boundary integral operators, the Herglotz wave operator, and the notation used throughout the paper.
Section~\ref{sec:LocalizedWaveFunction} constructs the localized wave functions central to the monotonicity argument; more precisely, we build sequences of functions that blow up in a prescribed region while vanishing elsewhere in both the rigid and traction-free cases.
Section~\ref{sec:Characterization of the impenetrable scatterer} presents the main theoretical results (the shape characterization Theorems~\ref{Characterofimpenetrable scatterer} and~\ref{Characterofimpenetrable scatterer2}) and translates them into the counting-based monotonicity sampling method (Algorithm~\ref{alg:discrete monotonicity}) via numerical approximations of the far-field and probing operators.
Section~\ref{sec:Numerical algorithm} develops, within the monotonicity framework of Section~\ref{sec:Characterization of the impenetrable scatterer}, the single-frequency and multi-frequency monotonicity spectral sampling methods (Algorithms~\ref{alg:spectrum monotonicity} and~\ref{alg:multi-frequency spectrum monotonicity}). Finally, Section~\ref{sec:Numerical example} provides an extensive numerical validation of all three algorithms, demonstrating their respective noise robustness and geometric resolution across a range of scatterer geometries.


\section{Factorizations of far-field operators with properties}\label{sec:far-field Operator decomposition}
This section establishes the necessary operator-theoretic foundation for the monotonicity-based shape characterization to be developed in Section~\ref{sec:Characterization of the impenetrable scatterer}. We begin by reviewing the basic properties of the single-layer potential operator $\mathbf{S}$ and the boundary integral operator $\mathbf{N}$ (Lemmas~\ref{SProperty} and~\ref{NProperty}); these properties are essential  for establishing the control of the data-to-pattern operators $\mathbf{G}_D^d$ and $\mathbf{G}_D^n$. We then present the factorization of the far-field operators $\mathbf{F}^d$ and $\mathbf{F}^n$ for both boundary conditions (Lemma~\ref{Fdecomposition}), which expresses them in terms of $\mathbf{S}$, $\mathbf{N}$, the data-to-pattern operators $\mathbf{G}_D^d$ and $\mathbf{G}_D^n$. Finally, we establish quantitative estimates for the data-to-pattern operators $\mathbf{G}_D^d$ and $\mathbf{G}_D^n$ (Lemma~\ref{DatatoPatternProperty}) that allow the compact perturbations arising in the factorization to be controlled on finite-codimension subspaces, which is a technique that is central to the monotonicity argument in 
Section~\ref{sec:Characterization of the impenetrable scatterer}.

Let $D$ be an open bounded domain with Lipschitz boundary, we denote by $\left[L^2\left(\partial D\right) \right] ^2$ and $\left[H^{1/2} \left(\partial D\right) \right] ^2$ the standard vector-valued Sobolev spaces defined on $\partial D$. For any two vector functions $\mathbf{g}=\left( g_p,g_s\right)$, $\mathbf{h}=\left( h_p,h_s\right)$ belonging to $ \left[L^2\left(\partial D\right) \right] ^2$, the inner product on this Hilbert space is defined as

	\begin{align*}
		\left\langle \mathbf{g},\mathbf{h}\right\rangle := \frac{\omega}{k_p}\int_{\partial D} g_p(\mathbf d) \overline{h_p(\mathbf d)}ds(\mathbf d)+\frac{\omega}{k_s}\int_{\partial D} g_s(\mathbf d) \overline{h_s(\mathbf d)}ds(\mathbf d).
	\end{align*}
We usually denote the dual space of  $\left[H^{1/2}\left(\partial D\right) \right] ^2$ as  $\left[H^{-1/2}\left(\partial D\right) \right] ^2$ with respect to inner product in $\left[L^2\left(\partial D\right) \right] ^2$ and their dual pairing is denoted as $\left( \cdot,\cdot\right) $.

For a given $\mathbf{g}=\left( g_p, g_s\right)$, the \emph{Herglotz wave operator} associated with density $\mathbf{g}$ is defined by $\mathbf H: \left[L^2\left(\mathbb S\right) \right] ^2 \to \left[H^{1/2}\left(\partial D\right) \right] ^2$ with the explicit form
	 	\begin{align}\label{eq:HerglotzOperatorD}
	 		\mathbf{Hg}\left( \mathbf x\right) :=e^{-i\pi /4} \int_{\mathbb S}  \left\lbrace \sqrt{\frac{k_p}{\omega}}\mathbf{d} e^{ik_p\mathbf{d}\cdot\mathbf{x}}g_p(\mathbf{d})+\sqrt{\frac{k_s}{\omega}}\mathbf{d}^{\perp} e^{ik_s\mathbf{d}\cdot\mathbf{x}}g_s(\mathbf{d})\right\rbrace ds(\mathbf{d}).
	 	\end{align}
	 	It is straightforward to obtain the adjoint operator $\mathbf H^*$ have the following form
	 	\begin{align*}\label{eq:HerglotzOperatorDA}
	 		\mathbf{H^* \phi }\left( \mathbf d\right) :=e^{i\pi /4} \int_{\partial \mathbb D}  \left\lbrace \sqrt{\frac{\omega}{k_p}}\mathbf{d} \cdot e^{-ik_p\mathbf{d}\cdot\mathbf{x}}\mathbf{\phi}\left( \mathbf{x}\right)  , \sqrt{\frac{k_s}{\omega}}\mathbf{d}^{\perp} \cdot e^{-ik_s\mathbf{d}\cdot\mathbf{x}}\mathbf{\phi}\left( \mathbf{x}\right)\right\rbrace ds(\mathbf{x}).
	 	\end{align*}
The elastic Herglotz wave function with density $\mathbf{g}$ is defined as $\mathbf{v_g}(\mathbf{x}):= \mathbf{Hg}\left( \mathbf x\right) $, it is evident that $\mathbf{v_g}(\mathbf{x})$ can be regarded as a superposition of plane waves of the form $\eqref{eq:ui}$. With this definition, we are now in a position to introduce the far-field operator $\mathbf{F}$, which maps the incident Herglotz density $\mathbf{g}$ to the far-field pattern of the corresponding scattered field.
		
According to \cite{Aren01}, for $\hat{\mathbf x}\in \mathbb{S}$, the \emph{far-field operator} $\mathbf{F}$: $\left[L^2\left(\mathbb S\right) \right] ^2 \to \left[L^2\left(\mathbb S\right) \right] ^2$ is defined as
	\begin{equation}\begin{aligned}\label{eq:FarfieldOperator}
		\mathbf{Fg}(\hat{\mathbf x})
		 & :=e^{-i\pi /4} \int_{\mathbb S} \mathbf{u}^{\infty} \left( \hat{\mathbf x}, \mathbf d, \sqrt{\frac{k_p}{\omega}}g_p(\mathbf d),\sqrt{\frac{k_s}{\omega}}g_s(\mathbf d) \right) ds(\mathbf d)\\
		& =e^{-i\pi /4} \int_{\mathbb S}\left\lbrace \sqrt{\frac{k_p}{\omega}}\mathbf{u}^\infty (\hat{\mathbf x},\mathbf d,1,0)g_p(\mathbf d)+\sqrt{\frac{k_s}{\omega}}\mathbf{u}^\infty (\hat{\mathbf x},\mathbf d,0,1)g_s(\mathbf d)\right\rbrace ds(\mathbf d),
	\end{aligned}\end{equation}
here, 
    \begin{equation}\label{eq:far-field p s}
            \mathbf{u}^\infty(\hat{\mathbf{x}},\mathbf{d},1,0) \mbox{ and } \mathbf{u}^\infty(\hat{\mathbf{x}},\mathbf{d},0,1)
    \end{equation}
represent the far-field patterns generated by purely compressional  and purely shear incident waves, respectively. In our discussion, to distinguish between the far-field operators corresponding to the two types of boundary conditions, we denote by $\mathbf{F}^d$ the far-field operator associated with the rigid impenetrable scatterer, and by $\mathbf{F}^n$ the far-field operator associated with the traction-free impenetrable scatterer.

The factorization of the far-field operators $\mathbf{F}^d$ and $\mathbf{F}^n$ involves the single-layer potential operator $\mathbf{S}$ and the boundary integral operator $\mathbf{N}$, which are introduced below.

We define the single-layer potential operator $\mathbf{S}:\left[H^{-1/2}\left(\partial D\right) \right] ^2 \to \left[H^{1/2}\left(\partial D\right) \right] ^2$ and the boundary integral operator $\mathbf{N}:\left[H^{1/2}\left(\partial D\right) \right] ^2 \to \left[H^{-1/2}\left(\partial D\right) \right] ^2$ 
on $\partial D$ as follows
	\begin{align*}
		\mathbf S \phi \left( \mathbf x\right) &:= \int_{\partial D} \mathbb G \left( \mathbf x, \mathbf y\right) \phi\left(\mathbf y\right) ds\left( \mathbf y\right) , \qquad \mathbf x\in \partial D, \\
		\mathbf N \phi \left( \mathbf x\right) &:=\mathbf T_{\bm{\nu}(\mathbf x)}\int_{\partial D} \left[\mathbf T_{\bm{\nu}(\mathbf y)}\mathbb G \left( \mathbf x, \mathbf y\right)\right]^{\top} \phi\left(\mathbf y\right) ds\left( \mathbf y\right) , \qquad \mathbf x\in \partial D,
	\end{align*}
where $\mathbb G \left( \mathbf x, \mathbf y\right)$ is the Green's tensor of the Navier equation, namely
$$
\mathbb G \left( \mathbf x, \mathbf y\right):=\frac{i}{4\mu} H_0^{\left( 1\right) }\left( k_s|\mathbf x -\mathbf y|\right) \mathbf I_2 +\frac{i}{4\omega^2} \nabla_x^{\top} \nabla_x \left( H_0^{\left( 1\right) }\left( k_s|\mathbf x -\mathbf y|\right)-H_0^{\left( 1\right) }\left( k_p|\mathbf x -\mathbf y|\right)\right) .
$$
Here, $\mathbf I_2$ is the $2 \times 2$ identity matrix, $H_0^{\left( 1\right) }$ is the Hankel function of the first kind and of order 0.

The following two lemmas summarize some properties of the single-layer potential operator $\mathbf{S}$ and the boundary integral operator $\mathbf{N}$ established in \cite{Aren01,LLW25}. Here, $\mathbf{S}_i$  and $\mathbf{N}_i$ denote the single-layer potential operator and the boundary integral operator, respectively, corresponding to the angular frequency $\omega = i$.
	
	 \begin{lem}\label{SProperty}
	 	Assume $\omega^2$ is not a Dirichlet eigenvalue of $-\Delta^*$ in $D$, then there hold that
	 	\begin{itemize}
	 		\item [(1).] $\mathbf S$ is an isomorphism from the Sobolev space $\left[H^{-1/2}\left(\partial D\right) \right] ^2 $ onto $\left[H^{1/2}\left(\partial D\right) \right] ^2$.
	 		\item [(2).] $Im\left( \phi,\mathbf S\phi\right) =0$ for some $\phi \in \left[H^{-1/2}\left(\partial D\right) \right] ^2$ implies $\phi=0$.
	 		\item [(3).] The operator $\mathbf S_i$ is compact, self adjoint, and positive defined  in $\left[L^2\left(\partial D\right) \right] ^2 $. Moreover, $\mathbf S_i$ is coercive as an operator from $\left[H^{-1/2}\left(\partial D\right) \right] ^2 $ onto $\left[H^{1/2}\left(\partial D\right) \right] ^2$, more precisely, there exists $c_0>0$ such that
	 		$$
	 		\left( \phi,\mathbf S_i \phi\right) \geq c_0\Vert \phi\Vert^2_{\left[H^{-1/2}\left(\partial D\right) \right] ^2}, \qquad \phi \in \left[H^{-1/2}\left(\partial D\right) \right] ^2.
	 		$$
	 		Furthermore, there exists a self adjoint and positive definite square root $\mathbf S_{i}^{1/2}$ of $\mathbf S_i$ and $\mathbf S_{i}^{1/2}$ is an isomorphism from $\left[H^{-1/2}\left(\partial D\right) \right] ^2 $ onto $\left[L^2\left(\partial D\right) \right] ^2 $ and from $\left[L^2\left(\partial D\right) \right] ^2 $ onto $\left[H^{1/2}\left(\partial D\right) \right] ^2 $.
	 		\item [(4).] The difference $\mathbf S - \mathbf S_i$ is compact from $\left[H^{-1/2}\left(\partial D\right) \right] ^2 $ to $\left[H^{1/2}\left(\partial D\right) \right] ^2$.
	 	\end{itemize}
	 \end{lem}
	 
	 \begin{lem}\label{NProperty}
	 Assume $\omega^2$ is not a Neumann eigenvalue of $-\Delta^*$ in $D$, then there hold that
	 	\begin{itemize}
	 		\item [(1).] $\mathbf{N}$ is an isomorphism from $[H^{1/2}(\partial D)]^2$ onto $[H^{-1/2}(\partial D)]^2$.
	 		\item [(2).] $Im(\mathbf{N}\varphi, \varphi) > 0$ for all $\varphi \in [H^{1/2}(\partial D)]^2$ with $\varphi \neq 0$.
	 		\item [(3).] The operator $\mathbf{N}_i$ is self-adjoint in $[L^2(\partial D)]^2$. Moreover, $-\mathbf{N}_i$ is coercive as an operator $\left[H^{-1/2}\left(\partial D\right) \right] ^2 $ onto $\left[H^{1/2}\left(\partial D\right) \right] ^2$, more precisely there exists $\tilde{c}_0>0$ such that
			\[
	 		-(\mathbf{N}_i\varphi, \varphi) \geq \tilde{c_0} \|\varphi\|_{[H^{1/2}(\partial D)]^2}, \qquad \varphi \in [H^{1/2}(\partial D)]^2.
	 		\]
			\item [(4).] The difference $\mathbf{N} - \mathbf{N}_i$ is compact from $[H^{1/2}(\partial D)]^2$ into $[H^{-1/2}(\partial D)]^2$.
		\end{itemize}
	 \end{lem}

Finally, we introduce the \emph{data-to-pattern operator} $\mathbf{G}_D^d$ and $\mathbf{G}_D^n$. By the existence and uniqueness of radiating solutions to the exterior Dirichlet and Neumann boundary value problems for the Navier equation \eqref{eq:syst1}~\cite{HH93}, the mappings from boundary values to far-field patterns are well-defined and injective. We define these operators as

\begin{align}\label{eq:DatatoPatternOperatorD}
	\mathbf{G}_D^d : \left[H^{1/2}\left(\partial D\right)\right]^2 \to \left[L^2\left(\mathbb{S}^1\right)\right]^2, \quad \mathbf{G}_D^d \mathbf{f} := \mathbf{u}_d^{\infty},
\end{align}

\noindent and

\begin{align}\label{eq:DatatoPatternOperatorN}
	\mathbf{G}_D^n : \left[H^{-1/2}\left(\partial D\right)\right]^2 \to \left[L^2\left(\mathbb{S}^1\right)\right]^2, \quad \mathbf{G}_D^n \mathbf{g} := \mathbf{u}_n^{\infty},
\end{align}

\noindent where $\mathbf{u}_d^{\infty}$ and $\mathbf{u}_n^{\infty}$ are the far-field patterns of the solutions to the exterior Dirichlet and Neumann boundary value problems with boundary values $\mathbf{f}$ and $\mathbf{g}$, respectively.
 
In Lemma \ref{Fdecomposition}, we  recall from \cite{Aren01,LLW25} the factorization of the far-field operators $\mathbf{F}^d$ and $\mathbf{F}^n$, as well as the relationships among the operators $\mathbf{S}$, $\mathbf{H}$, and $\mathbf{G}_D^d$. 
    
\begin{lem}\label{Fdecomposition}
For the far-field operator $\mathbf F^d$ and $\mathbf F^n$, Herglotz wave operator $\mathbf H$, data-to-pattern operator $\mathbf G_D^d$ and $\mathbf G_D^n$, we have the following three results
\begin{itemize}
\item [(1).] The far-field operator $\mathbf F^d$ can be represented as
\begin{align}\label{eq:FarfieldOperatorDecomposition}
 \mathbf F^d=-\sqrt{8\pi\omega}\mathbf{G}_D^d \bmf{S}^*\bmf{G}_D^{d*},
\end{align}
where $\bmf{G}_D^{d*}:\left[L^2\left(\mathbb S\right) \right] ^2 \to \left[H^{-1/2}\left(\partial D\right) \right] ^2$ and $\mathbf{S^*}:\left[H^{-1/2}\left(\partial D\right) \right] ^2 \to \left[H^{1/2}\left(\partial D\right) \right] ^2$ denote the adjoints of $\mathbf G_D^d$ and $\mathbf S$, respectively. 
\item [(2).] The far-field operator $\mathbf F^n$ can be represented as
\begin{align*}
 \mathbf F^n=-\sqrt{8\pi\omega}\mathbf{G}_D^n \bmf{N}^*\bmf{G}_D^{n*},
\end{align*}
where $\bmf{G}_D^{n*}:\left[L^2\left(\mathbb S\right) \right] ^2 \to \left[H^{1/2}\left(\partial D\right) \right] ^2$ and $\mathbf{N^*}:\left[H^{1/2}\left(\partial D\right) \right] ^2 \to \left[H^{-1/2}\left(\partial D\right) \right] ^2$ denote the adjoints of $\mathbf G_D^n$ and $\mathbf N$, respectively. 
\item [(3).] The operators $\mathbf S$, $\mathbf H$ and $\mathbf G_D^d$ and their adjoints satisfy the following two equations
\begin{align}\label{eq:HGSoperator}
	\mathbf{H^*}=\sqrt{8\pi\omega}\mathbf G_D^d \mathbf S\quad, \quad \mathbf{H}=\sqrt{8\pi\omega}\mathbf{S^*} \mathbf G_D^{d*}.
\end{align}
\end{itemize}
\end{lem}
\begin{rem}
The factorization \eqref{eq:FarfieldOperatorDecomposition} for the far-field operator $\mathbf{F}^d$ in the case where $D$ is a rigid impenetrable scatterer was proved in \cite{Aren01}, under the assumption that $D$ is a $C^2$-smooth domain. In fact, this result can be extended to the case where $D$ is merely a bounded Lipschitz domain, with only minor modifications.
\end{rem}
In the following lemma, we establish estimates for the inner products involving the operators $\mathbf{G}^d$ and $\mathbf{G}^n$ and arbitrary compact operators $\mathbf{K}_1$ and $\mathbf{K}_2$, acting on $\phi \in V_1^\perp$ and $\phi \in V_2^\perp$, respectively. These estimates play a crucial role in the derivation of the shape characterization by means of the monotonicity method.
    
	\begin{lem}\label{DatatoPatternProperty}
		\mbox{}
	\begin{itemize}
	  \item [(1).] For the rigid impenetrable scatterer, let $\mathbf K_1:\left[H^{-1/2}\left(\partial D\right) \right] ^2 \to \left[H^{1/2}\left(\partial D\right) \right] ^2$ is a compact and self and adjoint operator, then for any constant $c_1>0$, there exists a finite-dimensional subspace $V_1 \subseteq \left[L^2\left(\mathbb S\right) \right] ^2$, such that
	\begin{align*}
		\left|\left( \mathbf G_D^{d*} \mathbf {\phi} ,\mathbf {K}_1\mathbf{G}_D^{d*} \mathbf {\phi}\right) \right| \leq c_1\Vert \mathbf G_D^{d*} \mathbf \phi\Vert _{\left[H^{-1/2}\left(\partial D\right) \right] ^2} ^2, \qquad \forall \mathbf {\phi} \in V_1^\perp,
	\end{align*}
	here $\left( \cdot,\cdot\right) $denotes the duality pairing in $\left( \left[H^{-1/2}\left(\partial D\right) \right] ^2 ,\left[H^{1/2}\left(\partial D\right) \right] ^2\right) $.
	\item [(2).] For the traction-free impenetrable scatterer, let $\mathbf K_2:\left[H^{1/2}\left(\partial D\right) \right] ^2 \to \left[H^{-1/2}\left(\partial D\right) \right] ^2$ is a compact and self and adjoint operator, then for any constant $c_2>0$, there exists a finite-dimensional subspace $V_2 \subseteq \left[L^2\left(\mathbb S\right) \right] ^2$, such that
	\begin{align*}
		\left|\left( \mathbf G_D^{n*} \mathbf {\phi} ,\mathbf {K}_2\mathbf {G}_D^{n*} \mathbf {\phi}\right) \right| \leq c_2\Vert \mathbf G_D^{n*} \mathbf \phi\Vert _{\left[H^{1/2}\left(\partial D\right) \right] ^2} ^2, \qquad \forall \mathbf {\phi} \in V_2^\perp,
	\end{align*}
	here $\left( \cdot,\cdot\right) $denotes the duality pairing in $\left( \left[H^{1/2}\left(\partial D\right) \right] ^2 ,\left[H^{-1/2}\left(\partial D\right) \right] ^2\right) $.
	\end{itemize}
	\end{lem}

	\begin{proof}

    We first prove (1).
    From Lemma \ref{SProperty}, we can deduce that $\mathbf S_{i}^{1/2}$ is an isomorphism from $\left[H^{-1/2}\left(\partial D\right) \right] ^2 \to \left[L^2\left(\partial D\right) \right] ^2 ,$  thus it inverse operator $\mathbf S_{i}^{-1/2}:\left[L^2\left(\partial D\right) \right] ^2 \to \left[H^{-1/2}\left(\partial D\right) \right] ^2$  exists, and then $\mathbf S_{i}^{-1/2} \mathbf S_{i}^{1/2}$ is the identity operator from $\left[H^{-1/2}\left(\partial D\right) \right] ^2$ to $\left[H^{-1/2}\left(\partial D\right) \right] ^2$. Since $\mathbf S_{i}^{1/2}$ is a self adjoint operator, we can derived that $\mathbf S_{i}^{1/2} \mathbf S_{i}^{-1/2}$ is the identity operator from $\left[H^{1/2}\left(\partial D\right) \right] ^2$ to $\left[H^{1/2}\left(\partial D\right) \right] ^2$. Then we have
\begin{equation}\begin{aligned} \label{eq:inequality1}
	\left|\left( \mathbf G_D^{d*} \mathbf {\phi} , \mathbf {K_1}\mathbf{G}_D^{d*} \mathbf {\phi}\right) \right|
	 & =\left|\left( \mathbf G_D^{d*} \mathbf {\phi}, \mathbf S_{i}^{1/2} \mathbf S_{i}^{-1/2} \mathbf K_1 \mathbf S_{i}^{-1/2} \mathbf S_{i}^{1/2} \mathbf {G}_D^{d*} \mathbf {\phi}\right)\right|\\
		& =\left|\left( \mathbf S_{i}^{1/2} \mathbf {G}_D^{d*} \mathbf {\phi}, \mathbf {\tilde K_1}\mathbf S_{i}^{1/2} \mathbf {G}_D^{d*} \mathbf {\phi} \right)\right|,
	\end{aligned}\end{equation}
	where $ \mathbf {\tilde K}_1:=\mathbf S_{i}^{-1/2} \mathbf K_1 \mathbf S_{i}^{-1/2}$ is a compact and self adjoint operator from $\left[L^2\left(\partial D\right) \right] ^2$ to $\left[L^2\left(\partial D\right) \right] ^2$. From the spectral theory of compact and self adjoint operators, $\mathbf {\tilde K}_1$ admits a countable number of real eigenvalues with no non-zero accumulation points. The associated eigenvectors constitute a complete orthogonal basis of $\left[L^2\left(\partial D\right) \right] ^2$. We denote the space $\tilde V_1$ which are formed by the eigenvectors corresponding to the eigenvalues of $\mathbf {\tilde K} _1$ that are greater than $\tilde c_1$, here $\tilde c_1:=c_1/{||\mathbf S^{1/2}_i||^2_{\left[H^{-1/2}\left(\partial D\right) \right] ^2 \to \left[L^2\left(\partial D\right) \right] ^2}}$, thus $\tilde V_1$ is finite-dimensional, and for any $\tilde v_1 \in {\tilde V_1}^{\perp} $, we have the following inequality
	 \begin{align}\label{eq:inequality2}
	 	\left|\left( \tilde v_1,\mathbf {\tilde K}_1 \tilde v_1 \right) \right| \leq \tilde c_1 \Vert \tilde v_1\Vert^2_{\left[L^2\left(\partial D\right) \right] ^2}.
	 \end{align}
	Let $\phi \in \left[L^2\left(\partial D\right) \right] ^2$, then from the definition of the orthogonal complement space, we know that $\mathbf S^{1/2}_i \mathbf G_D^{d*} \phi \in {\tilde V_1}^{\perp}$ if and only if the following equality holds
	$$
	0=\left( \mathbf S^{1/2}_i \mathbf G_D^{d*} \phi,\tilde v_1 \right)=\left( \phi, \mathbf G_D^d \mathbf S^{1/2}_i \tilde v_1 \right) ,\qquad \tilde v_1 \in \tilde V_1.
	$$
	Therefore, $\mathbf S^{1/2}_i \mathbf G_D^{d*} \phi \in {\tilde V_1}^{\perp}$ if and only if $\phi \in {\mathbf G_D^d \mathbf S^{1/2}_i \tilde V_1}^{\perp}$, we can define $V_1:=\mathbf G_D^d \mathbf S^{1/2}_i \tilde V_1 \subseteq \left[L^2\left(\mathbb S\right) \right] ^2$, then we have
	$$
	\rm {dim}\left( V_1 \right) =\rm {dim}\left( \mathbf G_D^d \mathbf S^{1/2}_i \tilde V_1\right) \leq \rm {dim}\left( \tilde V_1\right) < \infty.
	$$
	From  \eqref{eq:inequality1} and \eqref{eq:inequality2} , for any $\phi \in V_1^{\perp}$ we have
\begin{equation}\notag\begin{aligned}
	\left|\left( \mathbf G_D^{d*} \mathbf {\phi} , \mathbf {K}_1\mathbf G_D^{d*} \mathbf {\phi}\right)\right| 
	&\leq \tilde c_1 \left\|\mathbf S^{1/2}_i \mathbf G_D^{d*} \phi\right\|^2_{\left[L^2\left(\partial D\right) \right] ^2}\\
	&\leq c_1\left\|\mathbf G_D^{d*} \phi\right\|^2_{\left[H^{-1/2}\left(\partial D\right) \right] ^2}.
	\end{aligned}\end{equation}

The proof of (2) is analogous and is therefore omitted.

The proof of this lemma is complete.
\end{proof}

\section{The existences of localized wave functions}\label{sec:LocalizedWaveFunction}

In this section, we address the existence of localized wave functions stated in Theorems \ref{LocialWaveFounc} and \ref{LocialWaveFounc2}. Specifically, let $B \subset \mathbb{R}^2$ be a bounded, simply connected Lipschitz domain. If $D$ is a rigid impenetrable scatterer and $B \not\subseteq D$, then we prove the existence of a sequence of functions such that the Herglotz wave operator $\mathbf{H}_B$ acting on this sequence tends to infinity in the $L^2$-norm, whereas the data-to-pattern operator $\mathbf{G}_D^d$ acting on the same sequence tends to zero in the $L^2$-norm. A similar result holds in the case where $D$ is a traction-free impenetrable scatterer. Before proceeding to the main results of this section, we first introduce the operators $\mathbf{H}_B$, $\mathbf{R}_{\tau}$, and $\mathbf{H}_{\tau}$, which will be useful in the subsequent discussion.

Given a bounded simply connected Lipschitz domain \(B \subset \mathbb R^2\), the Herglotz wave operator  $\mathbf{H}_B$ associated with \(B\) is defined similar to \eqref{eq:HerglotzOperatorD}:
\begin{equation}\label{eq:HerglotzOperatorB}
    \mathbf{H}_B \mathbf{g}(\mathbf{x})
    :=
    e^{-\mathrm{i}\pi/4}
    \int_{\mathbb{S}}
    \left\{
        \sqrt{\frac{k_p}{\omega}} \mathbf{d} \, e^{\mathrm{i}k_p \mathbf{d}\cdot\mathbf{x}} g_p(\mathbf{d})
        +
        \sqrt{\frac{k_s}{\omega}} \mathbf{d}^{\perp} \, e^{\mathrm{i}k_s \mathbf{d}\cdot\mathbf{x}} g_s(\mathbf{d})
    \right\}
    ds(\mathbf{d}),
    \qquad \mathbf{x} \in \partial B.
\end{equation}
It is easy to see that
$$
    \mathbf{H}_B : [L^2(\mathbb{S})]^2 \to [H^{1/2}(\partial B)]^2.
$$

Let $\tau \subseteq \partial B$ be a relatively open subset.  We firstly define the \emph{restriction operator}
$$
    \mathbf{R}_{\tau} : [H^{1/2}(\partial B)]^2 \to [H^{1/2}(\tau)]^2
$$
by
\begin{equation}\label{eq:RestrictionOperator}
    \mathbf{R}_{\tau}\mathbf{f} := \mathbf{f}|_{\tau}.
\end{equation}
In order to define the adjoint of $\mathbf{R}_{\tau}$, we first introduce the Sobolev space
\begin{equation}\notag
    [H^{-1/2}_{\mathrm{supp}}(\tau)]^2
    :=
    \left\{
        \mathbf{f} \in [H^{-1/2}(\partial B)]^2 \;:\; \operatorname{supp}\mathbf{f} \subseteq \overline{\tau}
    \right\}.
\end{equation}
Then the adjoint operator
$$
    \mathbf{R}_{\tau}^* : [H^{-1/2}_{\mathrm{supp}}(\tau)]^2 \to [H^{-1/2}(\partial B)]^2
$$
is given by the extension by zero, namely,
\begin{equation}\label{eq:RestrictionOperatorA}
    \mathbf{R}_{\tau}^* \mathbf{f}
    :=
    \begin{cases}
        \mathbf{f}, & \mbox{on } \tau,\\
        \mathbf{0}, & \mbox{on } \partial B \setminus \tau.
    \end{cases}
\end{equation}

We define the operator $\mathbf{H}_{\tau}$ by composing the Herglotz wave operator $\mathbf{H}_B$ defined by \eqref{eq:HerglotzOperatorB} with the restriction operator $\mathbf{R}_{\tau}$ defined by \eqref{eq:RestrictionOperator}, namely,
$$
    \mathbf{H}_{\tau} := \mathbf{R}_{\tau}\mathbf{H}_B.
$$
According to \eqref{eq:HGSoperator} in Lemma~\ref{Fdecomposition}, we immediately obtain the following decompostion of $\mathbf{H}_{\tau}^*$:
\begin{equation}\label{eq:RestriH}
    \mathbf{H}_{\tau}^*
    =
    \mathbf{H}_B^* \mathbf{R}_{\tau}^*
    =
    \sqrt{8\pi\omega}\,\mathbf{G}_B^d \mathbf{S}_B \mathbf{R}_{\tau}^*.
\end{equation}

We now prove in the following theorem that the range of $\mathbf{H}_{\tau}^*$ and the ranges of both data-to-pattern operators $\mathbf{G}_D^d$ and $\mathbf{G}_D^n$ (defined in \eqref{eq:DatatoPatternOperatorD} and \eqref{eq:DatatoPatternOperatorN}) have only the zero element in their intersection. This properties is the analytical mechanism that permits the construction of sequences of Herglotz densities with the divergence properties required in Theorems~\ref{LocialWaveFounc} and~\ref{LocialWaveFounc2}.
    
	\begin{thm}\label{RangeHG}
		Let $B, D \subset \mathbb R^2$ are open and bounded domain with Lipschitz  boundary, if $B \nsubseteq D$, let $\tau \subseteq \partial B \setminus \overline{D} $ is relatively open and $\mathbb R^2 \setminus \left( \overline{\tau \cup D }\right) $ is connected, then we have
			\begin{itemize}
				\item [(1).] When $D$ is a rigid impenetrable scatterer, it holds that
				\begin{align}\notag
					R\left( \mathbf H^*_{\tau} \right) \cap R\left( \mathbf G_D^d\right) =\left\lbrace 0\right\rbrace.
				\end{align}
				\item[(2).] When $D$ is a traction-free impenetrable scatterer, we obtain the following result
				\begin{align}\notag
					R\left( \mathbf H^*_{\tau} \right) \cap R\left( \mathbf G_D^n\right) =\left\lbrace 0\right\rbrace.
				\end{align}
			\end{itemize}
			
	\end{thm}
	
	\begin{proof}
	We provide a detailed proof for the case where $D$ is a rigid impenetrable scatterer. Since the proof for a traction-free impenetrable scatterer is analogous, we omit the proof of part~(2).

	Let $\mathbf h \in 	R\left( \mathbf H^*_{\tau} \right) \cap R\left( \mathbf G_D^d\right)$, since $\mathbf h$ lies in the ranges of $\mathbf H^*_{\tau}$ and $\mathbf G_D^d$, we can deduce that there exists $\mathbf{\phi} _{\tau} \in \left[H^{-1/2} \left(\tau \right) \right] ^2$, $\mathbf f \in \left[H^{1/2}\left(\partial D\right) \right] ^2$, such that
	$$
	\mathbf h= \mathbf H^*_{\tau} \phi _{\tau}= \mathbf G_D^d \mathbf f.
	$$
    Let $\mathbf v _1:= \sqrt{8\pi\omega} \mathbf S_B \mathbf R^*_{\tau} \phi _{\tau} \in \left[ H^1_{loc} \left( \mathbb R^2 \backslash \overline{\tau} \right) \right] ^2$, then $\mathbf v _1$ is the radiating solution to the following equation
	\begin{equation}\notag
			\Delta ^* \mathbf v _1 + \omega ^2 \mathbf v _1 =0,\qquad \mbox{in}\quad \mathbb R^2 \backslash \overline{\tau}.
	\end{equation}
	From \eqref{eq:HGSoperator}, we can infer that $\mathbf H^*_{\tau} \phi _{\tau}$ is $\sqrt{8\pi\omega}$ times the far-field pattern of the single-layer potential $\mathbf S_B$ with density $\mathbf R^*_{\tau} \phi _{\tau}$. Then, we can derive that $\mathbf H^*_{\tau} \phi _{\tau}= \mathbf v^{\infty} _1$. 
	
	On the other hand, if we let $\mathbf v _2 \in \left[ H^1_{loc} \left( \mathbb R^2 \backslash \overline{D} \right) \right] ^2$ be the radiating solution to the Navier equation satisfying $ \mathbf v _2= \mathbf f$ on $\partial D $, then we can obtain $ \mathbf G_D^d \mathbf f= \mathbf v^{\infty} _2$. Clearly, $\mathbf v _2 $ is the radiating solution to the following equation
   \begin{equation}\notag
	\Delta ^* \mathbf v _2 + \omega ^2 \mathbf v _2 =0,\qquad \mbox{in}\quad \mathbb R^2 \backslash \overline{D}.
   \end{equation}
   
	Since $\mathbf v^{\infty} _1=\mathbf H^*_{\tau} \phi _{\tau}= \mathbf h= \mathbf G_D^d \mathbf f= \mathbf v^{\infty} _2$, according to the Rellich lemma for elastic waves, we can derive that
	$$
	\mathbf v _1=\mathbf v _2, \qquad \mbox{in}\quad \mathbb R^2 \backslash \overline{\tau \cup D}.
	$$
	Then we can define $\mathbf v \in \left[ H^1_{loc} \left( \mathbb R^2 \right) \right] ^2$
	\begin{equation}\notag
		\mathbf v:=
		\begin{cases}
			\mathbf v _1=\mathbf v _2,& \mbox{in}\quad \mathbb R^2 \backslash \overline{\tau \cup D},\\
			\mathbf v _1, &  \mbox{in}\quad D ,	\\
			\mathbf v _2, &  \mbox{on}\quad \tau .
		\end{cases}
	\end{equation}
	Then $\mathbf v$ is an entire solution to the Navier equation, so we have
	$$
	\mathbf h= \mathbf v^{\infty} _1= \mathbf v^{\infty} _2 =0.
	$$
	
	The proof is complete.	
	\end{proof}

To construct the sequences of localized wave functions in Theorems \ref{LocialWaveFounc} and \ref{LocialWaveFounc2}, we require two auxiliary results from functional analysis: one relating operator norms to range inclusions (Lemma~\ref{Inquallity}), and another providing a dimension-counting argument for subspaces with prescribed intersection properties (Lemma~\ref{Dim}). These lemmas, adapted from ~\cite{HPS19b}, form the final technical components of the existence proof.

\begin{lem}\label{Inquallity}
	Assume $X, Y, Z$ be Hilbert spaces, let $A_1:X \to Y$ and $A_2: X \to Z$ are linear operators, then the following two statements are equivalent
	\begin{itemize}
		\item [(1).] There exists a constant $C>0$ such that $||A_1 x||_Y \leq C ||A_2 x||_Z, \quad \forall x \in X$.
		\item [(2).] $R \left( A^*_1 \right) \subseteq R \left( A^*_2 \right) $.
	\end{itemize}
\end{lem}

\begin{lem}\label{Dim}
	Assume $V, Z_1, Z_2 $ be subspaces of a vector space $Z$, if 
	$$
	Z_1 \cap Z_2= \left\lbrace 0 \right\rbrace \quad and \quad Z_1 \subseteq Z_2 +V,
	$$
	then $\rm {dim}\left(  Z_1 \right) \leqslant \rm {dim}\left(  V \right) $.
\end{lem}

Following the above introduction, we will demonstrate the existence of localized wave functions for the rigid impenetrable scatter and traction-free impenetrable scatter separately through the following two theorems.

\begin{thm}\label{LocialWaveFounc}
 Let $B, D  \subseteq \mathbb R^2$ be open and bounded domain with  Lipschitz boundary such that $\mathbb R^2 \backslash \overline {D}$ is connected. Assume $B \nsubseteq D$, then for any finite-dimensional subspace $V_1 \subseteq \left[L^{2}\left(\mathbb S\right) \right] ^2 $ , there exists a sequence $\left\lbrace  \mathbf f_n \right\rbrace  \subseteq  V_1^{\perp} $ such that
 $$
 \left\|\mathbf H_B \mathbf f_n\right\|_{\left[H^{1/2}\left(\partial B\right) \right] ^2 } \to \infty \quad and \quad \left\| \mathbf G_D^{d*} \mathbf f_n\right\|_{\left[H^{-1/2}\left(\partial D\right) \right] ^2 } \to 0, \qquad n\to \infty.
 $$
\end{thm}

\begin{proof}
Let $B, D  \subseteq \mathbb R^2$ is open and bounded domain with Lipschitz boundary such that $\mathbb R^2 \backslash \overline {D}$ is connected. Assume $B \nsubseteq D$, let $V_1 \subseteq \left[L^{2}\left(\mathbb S\right) \right] ^2 $ is a finite-dimensional subspace, therefore, the orthogonal projection onto $V_1$ is well-defined, and we denote it as $\mathbf P_1 :\left[L^{2}\left(\mathbb S\right) \right] ^2 \to \left[L^{2}\left(\mathbb S\right) \right] ^2$. 

Since $B \nsubseteq D $ and $\mathbb R^2 \backslash \overline {D}$ is connected, there exists a relatively open $\tau \subseteq \partial B\backslash \overline{D}$ such that $\mathbb R^2 \backslash \overline {\left( \tau \cup D\right) }$ is connected, from Theorem \ref{RangeHG} we have
\begin{align}\label{eq:RangeHtG}
R\left( \mathbf H^*_{\tau} \right) \cap R\left( \mathbf G_D^d\right) =\left\lbrace 0\right\rbrace .
\end{align}
We now turn our attention to \eqref{eq:RestriH}. Assume that $\omega^2$ is not a Dirichlet eigenvalue of the Navier operator $\Delta^*$ in $B$. Then both the single-layer potential operator $\mathbf{S}_B$ and the data-to-pattern operator $\mathbf{G}_B^d$ associated with $B$ are injective.

Moreover, the range of the extension operator $\mathbf R^*_{\tau} $ is infinite-dimensional, which implies that $ R \left( \mathbf H^*_{\tau} \right) = R \left( \sqrt{8\pi\omega}\mathbf{G}^d_B \mathbf S_B  \mathbf R^*_{\tau}   \right) $ is infinite-dimensional. Therefore, according to Lemma \ref{Dim} and \eqref{eq:RangeHtG} , we can obtain
$$
R \left( \mathbf H^*_{\tau} \right) \nsubseteq R\left( \mathbf G_D^d\right) +V_1=  R \left( \left[ \mathbf G_D^d \quad \mathbf P_1 \right] \right).
$$
Hence, utilizing Lemma \ref{Inquallity}, we can deduce that there does not exist a constant $C_1 > 0$ such that
\begin{equation}\begin{aligned}\label{eq:inver}
|| \mathbf H_{\tau} \mathbf g ||^2_{\left[H^{1/2}\left(\tau \right) \right] ^2 }
&\leq C_1^2 \left\| 
\begin{bmatrix}
	\mathbf G_D^{d*} \\ \mathbf P_1
\end{bmatrix}
\mathbf g \right\| ^2_{\left[H^{-1/2}\left(\partial D \right) \right] ^2  \times \left[L^{2}\left(\mathbb S \right) \right] ^2}\\
&=  C_1^2 \left( \left\| \mathbf G_D^{d*} \mathbf g \right\|^2_{\left[H^{-1/2}\left(\partial D \right) \right] ^2}+\left\| \mathbf P_1 \mathbf g \right\|^2_{\left[L^{2}\left(\mathbb S \right) \right] ^2} \right) .
\end{aligned}\end{equation}
Since $\mathbf P_1$ is an orthogonal projection operator, it follows that $\mathbf P_1$ is a self-adjoint operator, that is, $\mathbf P_1 = \mathbf P_1 ^*$. Therefore, from \eqref{eq:inver} for any $n \in \mathbb N$, there exists a vector valued function $\mathbf g_n \in \left[L^{2}\left(\mathbb S \right) \right] ^2  $ such that
$$
\left\| \mathbf H_{\tau} \mathbf g_n \right\|^2_{\left[H^{1/2}\left(\tau \right) \right] ^2 } > n^2 \left( \left\| \mathbf G_D^{d*} \mathbf g_n \right\|^2_{\left[H^{-1/2}\left(\partial D \right) \right] ^2}+\left\| \mathbf P_1 \mathbf g_n \right\|^2_{\left[L^{2}\left(\mathbb S \right) \right] ^2} \right).
$$
For any $n \in \mathbb N$, we denote $M_n:=\left\| \mathbf G_D^{d*} \mathbf g_n \right\|^2_{\left[H^{-1/2}\left(\partial D \right) \right] ^2}+\left\| \mathbf P_1 \mathbf g_n \right\|^2_{\left[L^{2}\left(\mathbb S \right) \right] ^2}$,  $\tilde{\mathbf g}_n:=\mathbf g_n /\left(\sqrt{n M_n} \right) $. On one hand, we have
\begin{equation}\notag 
\begin{aligned}
\left\| \mathbf H_{\tau} \tilde{\mathbf g}_n\right\|^2_{\left[H^{1/2}\left(\tau \right) \right] ^2 }
&= \frac{1}{n M_n} \left\| \mathbf H_{\tau} \mathbf g_n \right\|^2_{\left[H^{1/2}\left(\tau \right) \right] ^2 } \\
&>\frac{n}{M_n} \left( \left\|  \mathbf G_D^{d*} \mathbf g_n \right\|^2_{\left[H^{-1/2}\left(\partial D \right) \right] ^2}+\left\|  \mathbf P_1 \mathbf g_n \right\|^2_{\left[L^{2}\left(\mathbb S \right) \right] ^2} \right) \\
&= n,
\end{aligned}\end{equation}
which implies that 
$$
\left\| \mathbf H_{\tau} \tilde{\mathbf g}_n \right\|^2_{\left[H^{1/2}\left(\tau \right) \right] ^2 } > n \to \infty,\quad   n \to \infty .
$$

On the other hand, 
\begin{equation}\notag
\begin{aligned}
		\left\| \mathbf G_D^{d*} \tilde{\mathbf g}_n \right\|^2_{\left[H^{-1/2}\left(\partial D \right) \right] ^2}+\left\| \mathbf P_1 \tilde{\mathbf g}_n \right\|^2_{\left[L^{2}\left(\mathbb S \right) \right] ^2}
		&= \frac{1}{n M_n} \left(\left\| \mathbf G_D^{d*} \mathbf g_n \right\|^2_{\left[H^{-1/2}\left(\partial D \right) \right] ^2}+\left\| \mathbf P_1 \mathbf g_n \right\|^2_{\left[L^{2}\left(\mathbb S \right) \right] ^2} \right) \\
		&= \frac{1}{n}.
\end{aligned}\end{equation}
Thus, it yields that 
$$
\left\| \mathbf G_D^{d*} \tilde{\mathbf g}_n \right\|^2_{\left[H^{-1/2}\left(\partial D \right) \right] ^2}+\left\| \mathbf P_1 \tilde{\mathbf g}_n \right\|^2_{\left[L^{2}\left(\mathbb S \right) \right] ^2}= \frac{1}{n} \to 0, \quad n \to \infty .
$$
In summary, when $n \to \infty $, it can be deduced that
\begin{equation}\begin{aligned}\label{eq:HGPasymptotic}
  \left\| \mathbf H_{\tau} \tilde{\mathbf g}_n \right\|^2_{\left[H^{1/2}\left(\tau \right) \right] ^2 } \to \infty ,\quad \left\| \mathbf G_D^{d*} \tilde{\mathbf g}_n \right\|^2_{\left[H^{-1/2}\left(\partial D \right) \right] ^2} \to 0,\quad \left\| \mathbf P_1 \tilde{\mathbf g}_n \right\|^2_{\left[L^{2}\left(\mathbb S \right) \right] ^2} \to 0.
 \end{aligned}\end{equation}

For any $n \in \mathbb{N}$, we define $\mathbf{f}_n := \tilde{\mathbf{g}}_n - \mathbf{P}_1 \tilde{\mathbf{g}}_n$. By applying the triangle inequality, we obtain
\begin{equation}\begin{aligned}\label{eq:HInequality}
	    \left\| \mathbf H_{\tau} \mathbf f_n \right\|_{\left[H^{1/2}\left(\tau \right) \right]^2 }
		&\geq \left\| \mathbf H_{\tau} \tilde{\mathbf g}_n \right\|_{\left[H^{1/2}\left(\tau \right) \right]^2 } - \left\| \mathbf H_{\tau} \mathbf P_1 \tilde{\mathbf g}_n \right\|_{\left[H^{1/2}\left(\tau \right) \right]^2 }\\
		&\geq \left\| \mathbf H_{\tau} \tilde{\mathbf g}_n \right\|_{\left[H^{1/2}\left(\tau \right) \right]^2 } - \left\| \mathbf H_{\tau} \right\|_{ \left[L^{2}\left(\mathbb S \right) \right]^2 \to \left[H^{1/2} \left(\tau \right) \right] ^2 } \left\|\mathbf P_1 \tilde{\mathbf g}_n \right\|_{\left[L^{2} \left(\mathbb S \right) \right] ^2}.
\end{aligned}\end{equation}
and
\begin{equation}\begin{aligned}\label{eq:GInequality}
		\left\| \mathbf G_D^{d*} \mathbf f_n \right\|_{\left[H^{-1/2}\left(\partial D \right) \right]^2 }
		&\leq \left\| \mathbf G_D^{d*} \tilde{\mathbf g }_n\right\|_{\left[H^{-1/2}\left(\partial D \right) \right]^2 }+ \left\| \mathbf G_D^{d*} \mathbf P_1 \tilde{\mathbf g }_n\right\|_{\left[H^{-1/2}\left(\partial D \right) \right]^2 } \\
		&\leq \left\| \mathbf G_D^{d*} \tilde{\mathbf g }_n\right\|_{\left[H^{-1/2}\left(\partial D \right) \right]^2 }+ \left\| \mathbf G_D^{d*}\right\|_{\left[L^{2} \left(\mathbb S \right) \right] ^2 \to \left[H^{-1/2}\left(\partial D \right) \right]^2} \left\|\mathbf P_1 \tilde{\mathbf g }_n\right\|_{\left[L^{2} \left(\mathbb S \right) \right] ^2}.
\end{aligned}\end{equation}

It follows from \eqref{eq:HGPasymptotic}, \eqref{eq:HInequality}, and \eqref{eq:GInequality} that 
$\|\mathbf{H}_{\tau} \mathbf{f}_n\|_{[H^{1/2}(\tau)]^2} \to \infty$ and 
$\|(\mathbf{G}_D^d)^* \mathbf{f}_n\|_{[H^{-1/2}(\partial D)]^2} \to 0$ as $n \to \infty$. 
Recalling the definition $\mathbf{H}_{\tau} = \mathbf{R}_{\tau} \mathbf{H}_B$, we have
\begin{equation}\notag
\begin{aligned}
\|\mathbf{H}_{\tau} \mathbf{f}_n\|_{[H^{1/2}(\tau)]^2} 
&= \|\mathbf{R}_{\tau} \mathbf{H}_B \mathbf{f}_n\|_{[H^{1/2}(\partial B)]^2} \\
& \leq \left\|\mathbf R_{\tau}\right\|_{\left[H^{1/2}\left(\partial B \right) \right]^2 \to \left[H^{1/2}\left(\tau \right) \right]^2 } \left\| \mathbf H_{B} \mathbf f_n \right\|_{\left[H^{1/2}\left(\partial B \right) \right]^2 },
\end{aligned}
\end{equation}
which implies
\begin{equation}\notag 
\|\mathbf{H}_B \mathbf{f}_n\|_{[H^{1/2}(\partial B)]^2} 
\geq \frac{\|\mathbf{H}_{\tau} \mathbf{f}_n\|_{[H^{1/2}(\tau)]^2}}{\left\|\mathbf R_{\tau}\right\|_{\left[H^{1/2}\left(\partial B \right) \right]^2 \to \left[H^{1/2}\left(\tau \right) \right]^2 }}.
\end{equation}
Since $\|\mathbf{H}_{\tau} \mathbf{f}_n\|_{[H^{1/2}(\tau)]^2} \to \infty$ as $n \to \infty$ and the norm of the restriction operator $\mathbf{R}_{\tau}$ is bounded, it follows that
\begin{equation}\notag
\|\mathbf{H}_B \mathbf{f}_n\|_{[H^{1/2}(\partial B)]^2} \to \infty \quad \text{as } n \to \infty.
\end{equation}
This, together with \eqref{eq:GInequality}, completes the proof.

\end{proof}

The proof of Theorem~\ref{LocialWaveFounc2} is analogous to that of Theorem~\ref{LocialWaveFounc} and follows by repeating the similar arguments. For the sake of brevity, the detailed proof is omitted here.
\begin{thm}\label{LocialWaveFounc2}
	Let $B, D  \subseteq \mathbb R^2$ be open and bounded domain with Lipschitz boundary such that $\mathbb R^2 \backslash \overline {D}$ is connected. Assume $B \nsubseteq D$, then for any finite-dimensional subspace $V_2 \subseteq \left[L^{2}\left(\mathbb S\right) \right] ^2 $ , there exists a sequence $\left\lbrace  \mathbf h_n \right\rbrace  \subseteq  V_2^{\perp} $ such that
	$$
	\left\|\mathbf H_B \mathbf h_n\right\|_{\left[H^{1/2}\left(\partial B\right) \right] ^2 } \to \infty \quad and \quad \left\| \mathbf G_D^{n*} \mathbf h_n\right\|_{\left[H^{1/2}\left(\partial D\right) \right] ^2 } \to 0, \qquad n\to \infty.
	$$
\end{thm}

Theorems~\ref{LocialWaveFounc} and~\ref{LocialWaveFounc2} demonstrate that: whenever $B \not\subseteq D$, Herglotz wave fields exist that concentrate on $\partial B$ while becoming asymptotically invisible to $\partial D$. Together with the factorization of Section~\ref{sec:far-field Operator decomposition} and the boundary-space estimates of Lemma~\ref{DatatoPatternProperty}, this allows the geometric condition \(B \subseteq D\) to be reformulated as a finite-dimensional spectral condition on the monotonicity operator, which serves as the cornerstone for the shape characterization theorems derived in the next section.

\section{Characterizations of the impenetrable scatterer with reconstruction algorithms}\label{sec:Characterization of the impenetrable scatterer}

In this section, we present the shape characterization theorems for impenetrable scatterers and their numerical implementation, which are organized into three separate subsections.
In the first subsection, we establish the main theoretical results in Theorems~\ref{Characterofimpenetrable scatterer} and~\ref{Characterofimpenetrable scatterer2}. Specifically, when $D$ is a rigid impenetrable scatterer, we employ a probing domain $B$ to detect $D$ by examining the number of negative eigenvalues of the far-field operator $\mathbf{F}^d$ and of the \emph{probing operator} $\mathbf{H}_B^* \mathbf{H}_B$ in Theorem~\ref{Characterofimpenetrable scatterer}. The operators $\mathbf{F}^d$ and $\mathbf{H}_B$ are defined by \eqref{eq:FarfieldOperator} and \eqref{eq:HerglotzOperatorB}, respectively. A parallel characterization is derived for a traction-free impenetrable scatterer in Theorem~\ref{Characterofimpenetrable scatterer2}.
In the second subsection, we focus on the numerical discretization of the far-field operator $\mathbf{F}$ and the probing operator $\mathbf{H}_B^* \mathbf{H}_B$. And in 
the third subsection, we devise a counting-based indicator function $I_{{count}}$ based on the shape characterization criteria established by Theorems~\ref{Characterofimpenetrable scatterer} and~\ref{Characterofimpenetrable scatterer2}. Finally, the location and shape of the scatterer can be quantitatively reconstructed by evaluating the introduced indicator function $I_{{count}}$ at the sampling points. The detailed description of the proposed reconstruction method is displayed in Algorithm \ref{alg:discrete monotonicity}.

\subsection{Monotonicity shape characterization theorem}
First, we introduce a boundary value mapping operator $\mathbf{M}_{B \to D}$ and establish its compactness. Let $B \subset \mathbb{R}^2$ be a simply connected and bounded Lipschitz domain such that $\overline{B} \subset D$. We define the \emph{boundary value mapping operator} $\mathbf{M}_{B \to D} : [H^{1/2}(\partial B)]^2 \to [H^{1/2}(\partial D)]^2$ by
\begin{equation*}
    \mathbf{M}_{B \to D} \mathbf{f} := \mathbf{u}|_{\partial D},
\end{equation*}
where $\mathbf{u} \in [H^1_{\mathrm{loc}}(\mathbb{R}^2 \setminus \overline{B})]^2$ is the unique solution to the following exterior Dirichlet boundary value problem:
\begin{equation}\label{eq:ueq}
    \begin{cases}
        \Delta^* \mathbf{u} + \omega^2 \mathbf{u} = \mathbf{0}, & \text{in } \mathbb{R}^2 \setminus \overline{B}, \\
        \mathbf{u} = \mathbf{f}, & \text{on } \partial B, \\
        \lim\limits_{\rho \to \infty} \rho^{1/2} \left( \frac{\partial \mathbf{u}_\beta}{\partial \rho} - i k_\beta \mathbf{u}_\beta \right) = 0, & \rho = |\mathbf{x}|, \, \beta = p, s.
    \end{cases}
\end{equation}
\begin{rem}
By the definition of the data-to-pattern operator \(\mathbf{G}_D^d\) associated with \(D\), we have
\begin{equation}\notag
    \mathbf{G}_D^d \mathbf{M}_{B \to  D} \mathbf{f} = \mathbf{w}^{\infty},
\end{equation}
where $\mathbf{w} \in [H^1_{\mathrm{loc}}(\mathbb{R}^2 \setminus \overline{D})]^2$ is the unique solution to the following exterior Dirichlet boundary value problem:
\begin{equation}\notag
    \begin{cases}
        \Delta^* \mathbf{w} + \omega^2 \mathbf{w} = \mathbf{0}, & 	\text{in } \mathbb{R}^2 \setminus \overline{D}, \\
        \mathbf{w} = \mathbf{M}_{B 	\to D} \mathbf{f}, & 	\text{on } \partial D, \\
        \lim\limits_{\rho 	\to \infty} \rho^{1/2} \left( \frac{\partial \mathbf{w}_\beta}{\partial \rho} - i k_\beta \mathbf{w}_\beta \right) = 0, & \rho = |\mathbf{x}|, \, \beta = p, s.
    \end{cases}
\end{equation}
Since $\overline{B} \subset D$, the uniqueness of the solution to the exterior Dirichlet problem implies that $\mathbf{w} = \mathbf{u}$ in $\mathbb{R}^2 \setminus \overline{D}$, where $\mathbf{u}$ is defined in \eqref{eq:ueq}. Consequently, their far-field patterns coincide, namely $\mathbf{w}^{\infty} = \mathbf{u}^{\infty}$. Recall that $\mathbf{G}^d_B \mathbf{f} = \mathbf{u}^{\infty}$, where $\mathbf{G}^d_B$ denotes the data-to-pattern operator associated with \(B\). We then obtain
\begin{equation}\notag
    \mathbf{G}_D^d \mathbf{M}_{B 	\to D} \mathbf{f} = \mathbf{w}^{\infty} = \mathbf{u}^{\infty} = \mathbf{G}_B^d \mathbf{f},
\end{equation}
which leads to the operator identity 
\begin{equation}\label{eq:GM relation}
\mathbf{G}_B^d = \mathbf{G}_D^d \mathbf{M}_{B \to D}.
\end{equation}
\end{rem}

In Lemma~\ref{BoundaryVCompac}, we establish that $\mathbf{M}_{B \to D}$ is a compact operator. This result follows from the interior regularity of the Navier operator, the trace theorem, and the compact embedding properties of the associated Sobolev spaces.

\begin{lem}\label{BoundaryVCompac}
	Let $B \subset \mathbb{R}^2$ be a simply connected and bounded Lipschitz domain such that $\overline{B} \subset D$, the boundary value mapping operator $ \mathbf M_{B \to D}$ is a compact operator from $ \left[H^{1/2}\left(\partial B \right) \right]^2 $ to $\left[H^{1/2}\left(\partial D \right) \right]^2$.
\end{lem}
\begin{proof}
 From \eqref{eq:ueq}, it is clear that $\mathbf{u}$ satisfies the Navier equation in $\mathbb{R}^2 \setminus \overline{B}$. Since the Navier operator is strictly elliptic, for any open set $D'$ such that $\overline{D} \subset D'$, we have that $\mathbf{u}$ satisfies the Navier equation in $D' \setminus \overline{B}$. Given that $\overline{B} \subset D$, there exists an open neighborhood $\Omega'$ of $\partial D$ such that $\overline{\Omega'} \subset D' \setminus \overline{B}$. By the interior regularity of the Navier operator \cite{GT01}, it follows that $\mathbf{u} \in [H^2(\Omega')]^2$.

 Furthermore, by the trace theorem, it follows that $\mathbf u |_{\partial D} \in \left[H^{3/2}\left(\partial D \right) \right]^2$. Therefore, the boundary value mapping operator \(\mathbf M_{B \to D}\) can be rewritten as
 $$
 \mathbf M_{B \to D} = \mathbf E \mathbf M^{'}_{B \to D}.
 $$
Here, $\mathbf E$ is the compact embedding operator from $\left[H^{3/2}\left(\partial D \right) \right]^2$ to $\left[H^{1/2}\left(\partial D \right) \right]^2$, and $\mathbf M^{'}_{B \to D} := \left[H^{1/2}\left(\partial B \right) \right]^2 \to \left[H^{3/2}\left(\partial D \right) \right]^2$ is defined as
$$
\mathbf M^{'}_{B \to D} \mathbf f:= \mathbf u |_{\partial D}.
$$
Since $\mathbf{E}$ is compact, we conclude that $\mathbf{M}_{B 	\to D}$ is a compact operator from $[H^{1/2}(\partial B)]^2$ to $[H^{1/2}(\partial D)]^2$. 

The proof is complete.
\end{proof}

To prepare for the proof of Theorem~\ref{Characterofimpenetrable scatterer2},  in which the traction-free case is connected to the rigid case via the identity $\mathbf{G}_D^d = \mathbf{G}_D^n \bm{\mathcal{T}}$ (see~\eqref{eq:DNrelation}), we introduce the Dirichlet-to-Neumann (DtN) operator on $\partial D$.
\begin{defn} \label{de:DtN} 
		Let $B \subset \mathbb{R}^2$ be a simply connected and bounded Lipschitz domain such that $\overline{B} \subset D$, for any $\mathbf f\in \left[H^{1/2}\left(\partial D \right) \right]^2$, the DtN operator $ \bm{\mathcal{T}}: \left[H^{1/2}\left(\partial D \right) \right]^2 \to \left[H^{-1/2}\left(\partial D \right) \right]^2 $ for elastic waves is defined as follows
		$$ \bm{\mathcal{T}} \mathbf f =T_{\bm\nu}\mathbf u\vert_{\partial D},$$
		where $\mathbf u$ is the solution to the exterior Dirichlet boundary value problem \eqref{eq:ueq} with the impenetrable scatterer $B$ replaced by $D$.
\end{defn}
\begin{rem}
Since the solution to the exterior Dirichlet boundary value problem exists and is unique, the DtN operator $\bm{\mathcal{T}}$ is well-defined. Furthermore, according to \cite{BHSY18}, $\bm{\mathcal{T}}$ is a bounded operator.
\end{rem}

Finally, for two compact self-adjoint operators, we define a notation regarding the number of negative eigenvalues of their difference.
\begin{defn} 
	Let $X$ be a Hilbert space, and $A_1,A_2: X \to X $ be compact self-adjoint linear operators, if $A_2-A_1$ has at most $r$ negative eigenvalues (where $r\in \mathbb N$), we define this as
	$$
	A_1 \leq_r A_2.
	$$
\end{defn}

In particular, when the number of negative eigenvalues of $A_2-A_1$ is finite (that is, we can find $r \in \mathbb N$ such that $A_2-A_1$ satisfies $A_1 \leq_r A_2$), we denote this as $A_1 \leq_{fin} A_2$. Meanwhile, we recall from \cite{HPS19b} the equivalent condition under which the difference of these two operators has finitely many negative eigenvalues.

\begin{lem}\label{EquDef}
	Let $X$ be a Hilbert space with the inner product defined as $\left\langle  \cdot , \cdot \right\rangle _X$, and let $r \in \mathbb N$. Then the following two statements are equivalent:
	\begin{itemize}
		\item [(1).] $ A_1 \leq_r A_2 $.
		\item [(2).] There exists a finite-dimensional subspace $V \subseteq X$ and $dim \left( V\right) \leq r$, such that:
		$$
		\left\langle  \left( A_2- A_1\right) v, v \right\rangle _X \geq 0, \quad \forall v \in V^{\perp}.
		$$
	\end{itemize}
\end{lem}

Next, we consider the case of the rigid impenetrable scatterer. We develop a criterion to recover the shape of the scatterer \( D \) by analyzing the eigenvalue distribution of the operator \( -\mathbf H^*_B\mathbf H_B-\Re\left(\mathbf F^d \right) \). This criterion is formulated rigorously in the following theorem.
\begin{thm}\label{Characterofimpenetrable scatterer}
	Let $B,D\subseteq \mathbb{R}^2$ be open and  Lipschitz bounded domains such that $\mathbb{R}^2\setminus D$ is connected, where $D$ is a rigid impenetrable scatterer,
	\begin{itemize}
		\item [(1).] If $\overline{B}\subseteq D$, then we have $\Re\left( \mathbf F^d\right) \leq_{fin} -\mathbf H^*_B\mathbf H_B$;
		\item [(2).] If $B\nsubseteq D$, then we have $\Re\left( \mathbf F^d\right) \nleq_{fin} -\mathbf H^*_B\mathbf H_B$.
	\end{itemize} 
\end{thm}
\begin{proof}
First, we now prove the first part of the theorem.

In  \eqref{eq:HGSoperator}, replacing \( D \) with \( B \), we derive the relation $\mathbf H^*_B=\sqrt{8\pi\omega}\mathbf G^d_B \mathbf S_B$. Furthermore, from \eqref{eq:GM relation}, it follows that 
\begin{equation}\label{eq:Decomposition H}
    \mathbf H_B= \sqrt{8\pi\omega}\mathbf S^*_B \left( \mathbf M_{B \to D} \right) ^* \mathbf G_D^{d*}.
\end{equation}
Consequently, based on the factorization of the far-field operator $\mathbf F^d$ in Lemma \ref{Fdecomposition}, we can obtain
$$
\Re\left( \mathbf F^d\right)+\mathbf H^*_B\mathbf H_B = -\sqrt{8\pi\omega}\mathbf G_D^d \left[ \frac{1}{2}\left( \mathbf S^*+\mathbf S\right) \right] \mathbf G_D^{d*}+ 8\pi\omega\mathbf G_D^d \left[\mathbf M_{B \to D}\mathbf S_B\mathbf S^*_B\left( \mathbf M_{B \to D}\right) ^*\right] \mathbf G_D^{d*},
$$
where
\begin{equation}\notag\begin{aligned}
\frac{1}{2}\left( \mathbf S^*+\mathbf S\right) 
= \mathbf S_{i} + \frac{1}{2}\left( \left( \mathbf S^*-\mathbf S_{i}\right) +\left( \mathbf S-\mathbf S_{i}\right) \right) .
\end{aligned}\end{equation}
From Lemma \ref{SProperty}, $\mathbf S-\mathbf S_i$ is a compact operator. Letting $\mathbf K:= \mathbf S-\mathbf S_{i}$, we thus have that $\mathbf K$ is compact. Additionally, since the adjoint of a compact operator remains compact, it follows that $\mathbf S^*-\mathbf S_{i}=\left(\mathbf S-\mathbf S_{i}\right) ^*=\mathbf K^*$ is also compact. Define $\mathbf K_1:=\mathbf K +\mathbf K^*$, then we obtain $\frac{1}{2}\left( \mathbf S^*+\mathbf S\right)=\mathbf S_{i} +\mathbf K_1$, which implies that $\frac{1}{2}\left( \mathbf S^*+\mathbf S\right)$ is a compact perturbation of the self-adjoint and coercive operator $\mathbf S_{i}$. Thus, we can derive the following equation
$$
\Re\left( \mathbf F^d\right)+\mathbf H^*_B\mathbf H_B = -\mathbf G_D^d \left[\sqrt{8\pi\omega}\left( \mathbf S_{i} +\mathbf K_1 \right)-8\pi\omega \mathbf M_{B \to D}\mathbf S_B\mathbf S^*_B\left( \mathbf M_{B \to D}\right) ^*\right]\mathbf G_D^{d*}.
$$
We denote $\mathbf K_2:=8\pi\omega \mathbf M_{B \to D}\mathbf S_B\mathbf S^*_B\left( \mathbf M_{B \to D}\right) ^*$. By the compactness of the boundary value mapping operator $\mathbf M_{B \to D}$ demonstrated in Lemma \ref{BoundaryVCompac}, it follows that $\mathbf K_2$ is a compact operator from $ \left[H^{-1/2}\left(\partial D \right) \right]^2 $ to $\left[H^{1/2}\left(\partial D \right) \right]^2$. Meanwhile, we define $\mathbf{\tilde{K}}=\sqrt{8\pi\omega} \mathbf K_1+\mathbf K_2$, from which we derive 
\begin{equation*}
\Re\left( \mathbf F^d\right)+\mathbf H^*_B\mathbf H_B = -\mathbf G_D^d \left( \sqrt{8\pi\omega} \mathbf S_{i}+\mathbf{\tilde{K}} \right) \mathbf G_D^{d*},
\end{equation*}
where $\mathbf{\tilde{K}}:\left[H^{-1/2}\left(\partial D \right) \right]^2  \to \left[H^{1/2}\left(\partial D \right) \right]^2$ is a compact self-adjoint operator. Consequently, for any $\mathbf f \in \left[L^{2}\left(\mathbb S \right) \right]^2$, we have
\begin{equation*}\begin{aligned}
		\left\langle \left( \Re\left( \mathbf F^d\right)+\mathbf H^*_B\mathbf H_B\right) \mathbf f, \mathbf f\right\rangle 
		&= \left\langle \left(-\mathbf G_D^d \left( \sqrt{8\pi\omega} \mathbf S_{i}+\mathbf{\tilde{K}} \right) \mathbf G_D^{d*} \right) \mathbf f, \mathbf f\right\rangle  \\
		&\leq -c_1 || \mathbf G_D^{d*} \mathbf f ||^2_{\left[H^{-1/2}\left(\partial D \right) \right]^2} - \left(  \mathbf G_D^{d*} \mathbf f, \mathbf{\tilde{K}}\mathbf G_D^{d*} \mathbf f \right) ,
\end{aligned}\end{equation*}
	where the last inequality is derived from the coercivity of $\mathbf S_{i}$ from Lemma \ref{SProperty}. For the second term on the right-hand side of the above equation, Lemma \ref{DatatoPatternProperty} implies that there exists a finite-dimensional subspace \( V_1 \subseteq \left[L^{2}\left(\mathbb S \right) \right]^2 \) such that
	$$
	 - \left(  \mathbf G_D^{d*} \mathbf f, \mathbf{\tilde{K}}\mathbf G_D^{d*} \mathbf f \right) \leq c_1 ||\mathbf G_D^{d*} \mathbf f||_{\left[H^{-1/2}\left(\partial D\right) \right] ^2} ^2, \qquad \forall \, \mathbf f \in V_1^\perp,
	$$
	thus we can obtain 
	$$
	\left\langle \left( \Re\left( \mathbf F^d\right)+\mathbf H^*_B\mathbf H_B\right) \mathbf f, \mathbf f\right\rangle  \leq -c_1 || \mathbf G_D^{d*} \mathbf f ||^2_{\left[H^{-1/2}\left(\partial D \right) \right]^2} + c_1 ||\mathbf G_D^{d*} \mathbf f||_{\left[H^{-1/2}\left(\partial D\right) \right] ^2} ^2=0,\qquad \forall \,\mathbf f \in V_1^\perp.
	$$
	Therefore, according to Lemma \ref{EquDef}, it can be concluded that the first part of this theorem has been proved.
	
	To proceed with the proof of the second part of this theorem, we adopt a proof by contradiction.
	
	If $B\nsubseteq D$, suppose there exists a finite-dimensional subspace $\tilde{V}_1\in \left[L^{2}\left(\mathbb S \right) \right]^2$ such that
	$$
		\left\langle \left( \Re\left( \mathbf F^d\right)+\mathbf H^*_B\mathbf H_B\right) \mathbf f, \mathbf f\right\rangle \leq 0, \qquad \forall \,\mathbf f \in \tilde{V}_1^\perp.
	$$
	From the definitions of the inner product and norm in Hilbert space, we can deduce that
	\begin{equation}\label{eq:InequalityA}
		\left\langle \mathbf H^*_B\mathbf H_B \mathbf f ,\mathbf f\right\rangle =||\mathbf H_B \mathbf f||^2_{\left[H^{1/2}\left(\partial B \right) \right]^2}.
	\end{equation}
Meanwhile, since $\Re\left( \mathbf F^d\right) = -\sqrt{8\pi\omega}\mathbf G_D^d \left[ \frac{1}{2}\left( \mathbf S^* + \mathbf S\right) \right] \mathbf G_D^{d*}$, we can infer that
\begin{equation}\begin{aligned}\label{eq:InequalityB}
-\left\langle \Re\left( \mathbf F^d\right) \mathbf f, \mathbf f\right\rangle
&\leq |\left\langle \Re\left( \mathbf F^d\right) \mathbf f, \mathbf f\right\rangle| \\
&= \left| \left( -\frac{\sqrt{8\pi\omega}}{2} \left( \mathbf S^*+\mathbf S\right) \mathbf G_D^{d*} \mathbf f, \mathbf G_D^{d*} \mathbf f \right) \right| \\
&\leq c_1^{'}||\mathbf G_D^{d*} \mathbf f||_{\left[H^{-1/2}\left(\partial D \right) \right]^2}.
\end{aligned}\end{equation}
It is straightforward to observe that the above $c_1^{'} > 0 $. Combining \eqref{eq:InequalityA} and \eqref{eq:InequalityB}, we can deduce that
	$$
	\left\langle \left( \Re\left( \mathbf F^d\right)+\mathbf H^*_B\mathbf H_B\right) \mathbf f, \mathbf f\right\rangle \geq - c_1^{'}||\mathbf G_D^{d*} \mathbf f||^2_{\left[H^{-1/2}\left(\partial D \right) \right]^2} + ||\mathbf H_B \mathbf g||^2_{\left[H^{1/2}\left(\partial B \right) \right]^2},
	$$
	For the right-hand side of the above inequality, by Theorem \ref{LocialWaveFounc}, there exists a sequence $\left\lbrace \mathbf f_n \right\rbrace \subseteq V_1^\perp$ such that $||\mathbf G_D^{d*} \mathbf f_n||_{\left[H^{-1/2}\left(\partial D \right) \right]^2} \to 0$ and $||\mathbf H_B \mathbf f_n||^2_{\left[H^{1/2}\left(\partial B \right) \right]^2} \to \infty$ as $n\to \infty$. It follows that there exists $\mathbf f \in \tilde{V_1}^\perp $ such that $ \left\langle \left( \Re\left( \mathbf F^d\right)+\mathbf H^*_B\mathbf H_B\right) \mathbf f, \mathbf f\right\rangle > 0$, which contradicts our assumption. This completes the proof of the second part of the theorem.
\end{proof}

Next, we turn our attention to the traction-free impenetrable scatterer. We establish the corresponding shape characterization criterion, which relies on the eigenvalue properties of the operator \( \Re\left(\mathbf F^n \right)-\mathbf H^*_B\mathbf H_B \). The associated theorem is presented as follows.

\begin{thm}\label{Characterofimpenetrable scatterer2}
	Let $B,D\subseteq \mathbb R^{2}$ be open and bounded domains with Lipschitz boundary such that $\mathbb R^2\backslash D$ is connected, where $D$ is a traction-free impenetrable scatterer,
	\begin{itemize}
		\item [(1).] If $\overline{B}\subseteq D$, then we have $\mathbf H^*_B\mathbf H_B \leq_{fin} \Re\left( \mathbf F^n\right) $;
		\item [(2).] If $B\nsubseteq D$, then we have $\mathbf H^*_B\mathbf H_B \nleq_{fin} \Re\left( \mathbf F^n\right)$.
	\end{itemize} 
\end{thm}
\begin{proof}
        We first prove the first part of this theorem. From Definition \ref{de:DtN}, it follows that
		\begin{equation}\label{eq:DNrelation}
		    \mathbf G_D^d=\mathbf G_D^n \bm{\mathcal{T}}.
		\end{equation}
	   Therefore, based on the factorization of the far-field operator $\mathbf F^n$ in Lemma \ref{Fdecomposition}  and \eqref{eq:Decomposition H}, we can obtain
		\begin{equation*}\begin{aligned}
		\Re\left( \mathbf F^n\right)-\mathbf H^*_B\mathbf H_B 
		&= -\sqrt{8\pi\omega}\mathbf G_D^n \left[ \frac{1}{2}\left( \mathbf N^*+\mathbf N\right) \right] \mathbf G_D^{n*}- 8\pi\omega\mathbf G_D^d \left[\mathbf M_{B \to D}\mathbf S_B\mathbf S^*_B\left( \mathbf M_{B \to D}\right) ^*\right] \mathbf G_D^{d*}\\
		&=-\mathbf G_D^n \left[ \frac{\sqrt{8\pi\omega}}{2}\left( \mathbf N^*+\mathbf N\right) - 8\pi\omega \bm{\mathcal{T}} \mathbf M_{B \to D}\mathbf S_B\mathbf S^*_B\left( \mathbf M_{B \to D}\right) ^* \bm{\mathcal{T}}^* \right] \mathbf G_D^{n*},
		\end{aligned}\end{equation*}
		where $\bm{\mathcal{T}}^*$ denotes the adjoint operator of $\bm{\mathcal{T}}$.
	
It follows from Lemma \ref{NProperty} that $\mathbf N-\mathbf N_{i}$ is a compact operator. Accordingly, we have $-\frac{1}{2}\left( \mathbf N^*+\mathbf N\right)=-\mathbf N_{i} +\mathbf Q_1$, where $\mathbf Q_1$ is a compact operator. Therefore,
	$$
	\Re\left( \mathbf F^n\right)-\mathbf H^*_B\mathbf H_B =\mathbf G_D^n \left[\sqrt{8\pi\omega} \left( -\mathbf N_i+\mathbf Q_1\right) + 8\pi\omega \bm{\mathcal{T}} \mathbf M_{B \to D}\mathbf S_B\mathbf S^*_B\left( \mathbf M_{B \to D}\right) ^* \bm{\mathcal{T}}^* \right] \mathbf G_D^{n*},
	$$
	We denote $\mathbf Q_2:=8\pi\omega \mathbf M_{B \to D}\mathbf S_B\mathbf S^*_B\left( \mathbf M_{B \to D}\right) ^*$, according to the compactness of the boundary value mapping operator $\mathbf M_{B \to D}$ and the boundedness of the DtN operator $\bm{\mathcal{T}}$, it follows that $\mathbf Q_2$ is a compact operator from $ \left[H^{1/2}\left(\partial D \right) \right]^2 $ to $\left[H^{-1/2}\left(\partial D \right) \right]^2$. Meanwhile, we write $\mathbf{\tilde{Q}}=\sqrt{8\pi\omega} \mathbf Q_1+\mathbf Q_2$, from which we can obtain $\Re\left( \mathbf F^n\right)-\mathbf H^*_B\mathbf H_B = \mathbf G_D^n \left( -\sqrt{8\pi\omega} \mathbf N_{i}+\mathbf{\tilde{Q}} \right) \mathbf G_D^{n*}$, where $\mathbf{\tilde{Q}}:\left[H^{1/2}\left(\partial D \right) \right]^2  \to \left[H^{-1/2}\left(\partial D \right) \right]^2$ is a self-adjoint compact operator, therefore for any $\mathbf h \in \left[L^{2}\left(\mathbb S \right) \right]^2$, we have
		\begin{equation*}\begin{aligned}
				\left\langle \left( \Re\left( \mathbf F^n\right)-\mathbf H^*_B\mathbf H_B\right) \mathbf h, \mathbf h\right\rangle 
				&= \left\langle \left(\mathbf G_D^n \left( -\sqrt{8\pi\omega} \mathbf N_{i}+\mathbf{\tilde{Q}} \right) \mathbf G_D^{n*} \right) \mathbf h, \mathbf h\right\rangle  \\
				&\geq c_2 || \mathbf G_D^{n*} \mathbf h ||^2_{\left[H^{1/2}\left(\partial D \right) \right]^2} + \left(  \mathbf G_D^{n*} \mathbf h, \mathbf{\tilde{Q}}\mathbf G_D^{n*} \mathbf h \right) ,
		\end{aligned}\end{equation*}
		where the last inequality is derived from the coerciveness of $-\mathbf N_{i}$ from Lemma \ref{NProperty}. For the second term on the right-hand side of the above equation, according to Lemma \ref{DatatoPatternProperty}, we know that there exists a finite-dimensional subspace $V_2 \subseteq \left[L^{2}\left(\mathbb S \right) \right]^2$ such that
		$$
		 -\left(  \mathbf G_D^{n*} \mathbf h, \mathbf{\tilde{Q}}\mathbf G_D^{n*} \mathbf h \right) \leq c_2 ||\mathbf G_D^{n*} \mathbf h||_{\left[H^{1/2}\left(\partial D\right) \right] ^2} ^2, \qquad \forall \,\mathbf h \in V_2^\perp,
		$$
		Thus we can obtain 
		$$
		\left\langle \left( \Re\left( \mathbf F^n\right)-\mathbf H^*_B\mathbf H_B\right) \mathbf h, \mathbf h\right\rangle  \geq c_2 || \mathbf G_D^{n*} \mathbf h ||^2_{\left[H^{1/2}\left(\partial D \right) \right]^2} - c_2 ||\mathbf G_D^{n*} \mathbf h||_{\left[H^{1/2}\left(\partial D\right) \right] ^2} ^2=0,\qquad \forall \,\mathbf h\in V_2^\perp.
		$$
		Therefore, according to Lemma \ref{EquDef}, it can be concluded that the first part of this theorem has been proved.

        Similarly, for the proof of the second part of this theorem, we also employ a proof by contradiction.
		
		If $B\nsubseteq D$, suppose there exists a finite-dimensional subspace $\tilde{V}_2\in \left[L^{2}\left(\mathbb S \right) \right]^2$ such that
		$$
		\left\langle \left( \Re\left( \mathbf F^n\right)-\mathbf H^*_B\mathbf H_B\right) \mathbf h, \mathbf h\right\rangle \geq 0, \qquad \forall \,\mathbf h \in \tilde{V}_2^\perp.
		$$
		Since $\Re\left( \mathbf F^n\right) = -\sqrt{8\pi\omega}\mathbf G_D^n \left[ \frac{1}{2}\left( \mathbf N^* + \mathbf N\right) \right] \mathbf G_D^{n*}$, we can infer that
		\begin{equation}\begin{aligned}\label{eq:InequalityB2}
				\left\langle \Re\left( \mathbf F^n\right) \mathbf h, \mathbf h\right\rangle
				&\leq |\left\langle \Re\left( \mathbf F^n\right) \mathbf h, \mathbf h\right\rangle| \\
				&= \left| \left( -\frac{\sqrt{8\pi\omega}}{2} \left( \mathbf N^*+\mathbf N\right) \mathbf G_D^{n*} \mathbf h, \mathbf G_D^{n*} \mathbf h \right) \right| \\
				&\leq \frac{\sqrt{8\pi\omega}}{2} ||\mathbf N^*+\mathbf N ||_{\left[H^{1/2}\left(\partial D\right) \right] ^2 \to \left[H^{-1/2}\left(\partial D\right) \right] ^2} ||\mathbf G_D^{n*} \mathbf h||^2_{\left[H^{1/2}\left(\partial D \right) \right]^2}\\
				&\leq c_2^{'}||\mathbf G_D^{n*} \mathbf h||_{\left[H^{1/2}\left(\partial D \right) \right]^2}.
		\end{aligned}\end{equation}
		It is straightforward to observe that the above $c_2^{'} > 0 $. Combining \eqref{eq:InequalityA} and \eqref{eq:InequalityB2}, we can deduce that
		$$
		\left\langle \left( \Re\left( \mathbf F^n\right)-\mathbf H^*_B\mathbf H_B\right) \mathbf h, \mathbf h\right\rangle \leq c_2^{'}||\mathbf G_D^{n*} \mathbf h||^2_{\left[H^{1/2}\left(\partial D \right) \right]^2} + ||\mathbf H_B \mathbf h||^2_{\left[H^{1/2}\left(\partial B \right) \right]^2},
		$$
		For the right-hand side of the above inequality, according to Theorem \ref{LocialWaveFounc2}, we know that there exists a sequence $\left\lbrace \mathbf h_n \right\rbrace \subseteq V_2^\perp$ such that $||\mathbf G_D^{n*} \mathbf h_n||_{\left[H^{-1/2}\left(\partial D \right) \right]^2} \to 0$ and $||\mathbf H_B \mathbf h_n||^2_{\left[H^{1/2}\left(\partial B \right) \right]^2} \to \infty$ as $n\to \infty$, hence we can obtain that there exists $\mathbf h \in \tilde{V}_2^\perp $ such that $ \left\langle \left( \Re\left( \mathbf F^n\right)-\mathbf H^*_B\mathbf H_B\right) \mathbf h, \mathbf h\right\rangle > 0$, this contradicts our assumption, thus completing the proof of the second part of the theorem.
\end{proof}

\begin{rem}\label{rem:41}
Based on the shape characterization established in Theorems~\ref{Characterofimpenetrable scatterer} and~\ref{Characterofimpenetrable scatterer2}, the inclusion relationship between a probing domain $B$ and the unknown scatterer $D$ is encoded in the spectral properties of a specific linear combination of the far-field operator \(\mathbf F\) and \emph{probing operator} \(\mathbf{H}_B^* \mathbf{H}_B\). To provide a numerical framework, we define the \emph{monotonicity operator} $\mathcal{A}_{B}$ as:
\begin{equation}\label{operator:A_B}
\mathcal{A}_{B} :=
\begin{cases}
-\Re(\mathbf{F}^d) - \mathbf{H}_B^* \mathbf{H}_B, & 	\text{if $D$ is rigid}, \\
\Re(\mathbf{F}^n) - \mathbf{H}_B^* \mathbf{H}_B, & 	\text{if $D$ is traction-free}.
\end{cases}
\end{equation}
The theoretical criteria for shape reconstruction can be summarized as follows: $\overline B \subseteq D$ if and only if $\mathcal{A}_{B}$ possesses at most finitely many negative eigenvalues. Conversely, if $B \not\subseteq D$, the compact self-adjoint operator $\mathcal{A}_{B}$ admits infinitely many negative eigenvalues.

Building upon the shape characterizations established in Theorems~\ref{Characterofimpenetrable scatterer} and~\ref{Characterofimpenetrable scatterer2}, a natural counting-based indicator for recovering the boundary of $D$ is obtained by quantifying the number of negative eigenvalues of the monotonicity operator $\mathcal{A}_{B}$. Specifically, we define the \emph{counting-based indicator} as:
\begin{equation}\label{eq:indicatorContinuous}
    \mathcal{I}_{\mathrm{count}}(B) := \# \left\{ \lambda_j(\mathcal{A}_{B}) : \lambda_j(\mathcal{A}_{B}) < 0 \right\},
\end{equation}
where $\{\lambda_j(\mathcal{A}_{B})\}$ denotes the sequence of eigenvalues of the compact, self-adjoint operator $\mathcal{A}_{B}$. Let $\mathcal{N} \subset \mathbb{R}^2$ be a bounded search region (search domain) containing the unknown scatterer, i.e., $D \Subset \mathcal{N}$. For any probing domain $B \Subset \mathcal{N}$, it follows directly from Theorems~\ref{Characterofimpenetrable scatterer} and~\ref{Characterofimpenetrable scatterer2} that $\mathcal{I}_{\mathrm{count}}(B) < \infty$ if $\overline{B} \subseteq D$, whereas $\mathcal{I}_{\mathrm{count}}(B) = \infty$ if $B \not\subseteq D$ (or equivalently, $B \cap (\mathcal{N} \setminus D) \neq \emptyset$). 

Consequently, by sweeping the probing domain $B$ across the search region $\mathcal{N}$, the boundary $\partial D$ can be reconstructed. In the following subsection, we detail the discretized counterpart of $\mathcal{I}_{\mathrm{count}}(B)$. At the numerical level, due to discretization, truncation, and rounding effects, the indicator $\mathcal{I}_{\mathrm{count}}(B)$ does not truly diverge to infinity when $B$ intersects the exterior of $D$; nevertheless, it becomes larger compared to the relatively small values obtained when the containment $\overline{B} \subseteq D$ holds.

\end{rem}

\subsection{Numerical approximations of monotonicity operators}\label{sub:Numerical approximation}
In this subsection, we focus on the numerical implementation of Theorems \ref{Characterofimpenetrable scatterer} and \ref{Characterofimpenetrable scatterer2}. Specifically, we present the numerical approximation of the far-field operator $\mathbf F$ and the probing operator $\mathbf H_{B}^*\mathbf H_{B}$ involved in these results. According to Remark \ref{rem:41}, we shall design a counting-based indicator $I_{count}$ and further develop the \emph{counting-based monotonicity sampling method}.

Now, let us consider the numerical approximations of these two operators. We adopt the incident wave $\mathbf u^i$ given by \eqref{eq:ui}. Recall that $ \mathbf d $ denotes the incident direction of $\mathbf{u}^i$, while $ \hat{\mathbf x} $ denotes the observation direction in the far-field measurement. In practice, only limited data are available; therefore, we assume that there are $ N $ incident directions and $ N $ observation directions, namely,
\begin{align}
\label{eq:incidentd}
\mathbf d_n &= \left( \cos \alpha_n, \sin \alpha_n \right), \quad \alpha_n = \frac{(n-1)2\pi}{N}, \quad 1 \leq n \leq N,\\
\label{eq:observationx}
\hat{\mathbf x}_m &= \left( \cos \beta_m, \sin \beta_m \right), \quad \beta_m = \frac{(m-1)2\pi}{N}, \quad 1 \leq m \leq N.
\end{align}

First, we consider the numerical approximation of the far-field operator $\mathbf{F}$. From the definition of the far-field operator $\mathbf{F}$ in \eqref{eq:FarfieldOperator}, and using the trapezoidal rule, the integral can be discretized to obtain
\begin{align*}
(\mathbf{F} \mathbf g)_p(\hat{\mathbf x}) &\approx e^{-\frac{i\pi}{4}} \frac{2\pi}{N} \sum_{n=1}^{N} \left\{ \sqrt{\frac{k_p}{\omega}} u_p^\infty(\hat{\mathbf x}, \mathbf d_n, 1, 0) g_p(\mathbf d_n) + \sqrt{\frac{k_s}{\omega}} u_p^\infty(\hat{\mathbf x}, \mathbf d_n, 0, 1) g_s(\mathbf d_n) \right\},\\
(\mathbf{F} \mathbf g)_s(\hat{\mathbf x}) &\approx e^{-\frac{i\pi}{4}} \frac{2\pi}{N} \sum_{n=1}^{N} \left\{ \sqrt{\frac{k_p}{\omega}} u_s^\infty(\hat{\mathbf x}, \mathbf d_n, 1, 0) g_p(\mathbf d_n) + \sqrt{\frac{k_s}{\omega}} u_s^\infty(\hat{\mathbf x}, \mathbf d_n, 0, 1) g_s(\mathbf d_n) \right\},
\end{align*}
where $u_p^\infty(\hat{\mathbf{x}}, \mathbf{d}_n, 1, 0)$, $u_p^\infty(\hat{\mathbf{x}}, \mathbf{d}_n, 0, 1)$, $u_s^\infty(\hat{\mathbf{x}}, \mathbf{d}_n, 1, 0)$, and $u_s^\infty(\hat{\mathbf{x}}, \mathbf{d}_n, 0, 1)$ denote the compressional and shear parts of the far-field patterns $\mathbf{u}^\infty(\hat{\mathbf{x}}, \mathbf{d}_n, 1, 0)$ and $\mathbf{u}^\infty(\hat{\mathbf{x}}, \mathbf{d}_n, 0, 1)$ given by \eqref{eq:far-field p s}, respectively.
We take the values of the far-field operator at the \( N \) observation directions \( \hat{\mathbf x}_m \) as the approximation of the far-field operator. Accordingly, we employ the following \( 2 \times 2 \) block matrix \( \mathbb{F} \) as an approximation of the far-field operator
\begin{equation*}
\mathbb{F} = 
\begin{pmatrix}
	(\mathbb{F})_{pp} & (\mathbb{F})_{ps} \\
	(\mathbb{F})_{sp} & (\mathbb{F})_{ss}
\end{pmatrix}.
\end{equation*}
Here, \( (\mathbb{F})_{pp} \), \( (\mathbb{F})_{ps} \), \( (\mathbb{F})_{sp} \), and \( (\mathbb{F})_{ss} \) are all \( N \times N \) matrices. The entry at position \( (m,n) \) of each corresponding matrix is defined as follows:
\begin{align*}
    (\mathbb{F})_{pp}(m,n) &:=  e^{-\frac{i\pi}{4}} \frac{2\pi}{N} \sqrt{\frac{k_p}{\omega}} u_p^\infty(\hat{\mathbf x}_m, \mathbf d_n, 1, 0),\\ 
    (\mathbb{F})_{ps}(m,n) &:=  e^{-\frac{i\pi}{4}} \frac{2\pi}{N} \sqrt{\frac{k_s}{\omega}} u_p^\infty(\hat{\mathbf x}_m, \mathbf d_n, 0, 1),\\ 
    (\mathbb{F})_{sp}(m,n) &:=  e^{-\frac{i\pi}{4}} \frac{2\pi}{N} \sqrt{\frac{k_p}{\omega}} u_s^\infty(\hat{\mathbf x}_m, \mathbf d_n, 1, 0),\\ 
    (\mathbb{F})_{ss}(m,n) &:=  e^{-\frac{i\pi}{4}} \frac{2\pi}{N} \sqrt{\frac{k_s}{\omega}} u_s^\infty(\hat{\mathbf x}_m, \mathbf d_n, 0, 1). 
\end{align*}

Next, we first select an appropriate probing domain and focus on the approximation of the probing operator \( \mathbf H_B^* \mathbf H_B \). To begin with, we assume \( D \) is contained within the square \( [-R, R]^2 \). We then perform a grid partition on this domain of interest and denote its nodes as
\begin{equation*}
\mathcal{N}_h := \left\{ \mathbf z_{ij} = (ih, jh) \mid -M \leq i,j \leq M \right\} \subseteq [-R, R]^2,
\end{equation*}
where the grid size is given by \( h = {R}/{M} \). Based on this uniform partition of the domain \( [-R, R]^2 \), for each node \( \mathbf z_{ij} \in \mathcal{N}_h \), we take the probing domain \(B_{ij}\) at this position to be a disk centered at \(\mathbf z_{ij} \) with radius \( h \), namely,
\begin{equation*}
B_{ij} = B(\mathbf{z}_{ij}, h) = \left\{ \mathbf{x} \in \mathbb{R}^2 \mid |\mathbf{x} - \mathbf{z}_{ij}| < h \right\}.
\end{equation*}
Consequently, similar to the far-field operator $\mathbf{F}$, the operator \( \mathbf{H}_{B_{ij}}^* \mathbf{H}_{B_{ij}} \) can be approximated by the following \( 2 \times 2 \) block matrix
\begin{equation*}
\mathbb{H}_{ B_{ij}}^* \mathbb{H}_{ B_{ij}} = 
\begin{pmatrix}
	(\mathbb{H}_{ B_{ij}}^* \mathbb{H}_{ B_{ij}})_{pp} & (\mathbb{H}_{ B_{ij}}^* \mathbb{H}_{ B_{ij}})_{ps} \\
	(\mathbb{H}_{ B_{ij}}^* \mathbb{H}_{ B_{ij}})_{sp} & (\mathbb{H}_{ B_{ij}}^* \mathbb{H}_{ B_{ij}})_{ss}
\end{pmatrix}.
\end{equation*}
Here, \( (\mathbb{H}_{ B_{ij}}^* \mathbb{H}_{ B_{ij}})_{pp} \), \( (\mathbb{H}_{ B_{ij}}^* \mathbb{H}_{ B_{ij}})_{ps} \), \( (\mathbb{H}_{ B_{ij}}^* \mathbb{H}_{ B_{ij}})_{sp} \), and \( (\mathbb{H}_{ B_{ij}}^* \mathbb{H}_{ B_{ij}})_{ss} \) are all \( N \times N \) matrices. The entry at position \( (l,n) \) of each corresponding matrix is defined as follows
\begin{align*}
    \left(\mathbb{H}_{ B_{ij}}^* \mathbb{H}_{ B_{ij}} \right)_{pp}(l,n) &=\frac{4\pi^2 h}{N} e^{i k_p \mathbf z_{ij} \cdot (\bm\theta_n - \mathbf d_l)} J_0\left( k_p h |\bm\theta_n - \mathbf d_l| \right) \mathbf d_l \cdot \bm\theta_n;\\
    \left(\mathbb{H}_{ B_{ij}}^* \mathbb{H}_{ B_{ij}} \right)_{ps}(l,n) &=\frac{4\pi^2 h}{N} \sqrt{\frac{k_s}{k_p}} e^{i \mathbf z_{ij} \cdot (k_s \bm\theta_n - k_p \mathbf d_l)} J_0\left( h |k_s \bm\theta_n - k_p \mathbf d_l| \right) \mathbf d_l \cdot \bm\theta_n^\perp;\\
    \left(\mathbb{H}_{ B_{ij}}^* \mathbb{H}_{ B_{ij}} \right)_{sp}(l,n) &=\frac{4\pi^2 h}{N} \sqrt{\frac{k_p}{k_s}} e^{i \mathbf z_{ij} \cdot (k_p \bm\theta_n - k_s \mathbf d_l)} J_0\left( h |k_p \bm\theta_n - k_s \mathbf d_l| \right) \mathbf d_l^\perp \cdot \bm\theta_n;\\
    \left(\mathbb{H}_{ B_{ij}}^* \mathbb{H}_{ B_{ij}} \right)_{ss}(l,n) &=\frac{4\pi^2 h}{N} e^{i k_s \mathbf z_{ij} \cdot (\bm\theta_n - \mathbf d_l)} J_0\left( k_s h |\bm\theta_n - \mathbf d_l| \right) \mathbf d_l^\perp \cdot \bm\theta_n^\perp,
\end{align*}
where the definition of \(\mathbf d_l\) is given in \eqref{eq:incidentd} and
\[
\bm\theta_n = \left( \cos \sigma_n, \sin \sigma_n \right) \in \mathbb S, \quad \sigma_n = \frac{(n-1)2\pi}{N}, \quad 1 \leq n \leq N.
\]

In summary, the monotonicity operator $\mathcal A_{B_{ij}}$ associated with probing domain \(B_{ij}\) can be approximated by a $2N \times 2N$ \emph{discretized monotonicity matrix} $\mathbb A_{B_{ij}}$, where $\mathbb A_{B_{ij}}$ is defined as
\begin{equation}\label{discretized operator:A_B}
\mathbb A_{B_{ij}} :=
\begin{cases}
-\Re(\mathbb{F}^d) - \mathbb{H}_{B_{ij}}^* \mathbb{H}_{B_{ij}}, & 	\text{if $D$ is rigid}, \\
\Re(\mathbb{F}^n) - \mathbb{H}_{B_{ij}}^* \mathbb{H}_{B_{ij}}, & 	\text{if $D$ is traction-free}.
\end{cases}
\end{equation}

After obtaining $\mathbb A_{B_{ij}}$, we can determine whether a probing domain $ B_{ij} $ is contained in $ D $ by counting the number of negative eigenvalues of $\mathbb A_{B_{ij}}$, in accordance with the shape characterization results stated in Theorems \ref{Characterofimpenetrable scatterer} and \ref{Characterofimpenetrable scatterer2}. Given the discontinuous nature of this eigenvalue-counting procedure, we term this counting-based strategy the counting-based monotonicity sampling method, which is rigorously derived from the aforementioned monotonicity characterization theorems.

\subsection{Counting-based monotonicity sampling method}\label{subsec:discrete sampling}

Based on the numerical approximation $\mathbb A_{B_{ij}}$ defined in \eqref{discretized operator:A_B} of the monotonicity operator $\mathcal{A}_{B_{ij}}$ and the shape characterization results in
Theorems~\ref{Characterofimpenetrable scatterer} and~\ref{Characterofimpenetrable scatterer2}, we define the \emph{counting-based monotonicity indicator} $I_{count}:\mathcal{N}_h	\to \mathbb{N}$ by
\begin{equation}\label{eq:indicatorD}
I_{count}(\mathbf{z}_{ij}):=\#\left\{\lambda_{n,(ij)}\;\middle|\;\lambda_{n,(ij)} < -\sigma,\
1 \leq n \leq 2N \right\}, \quad -M \leq i,j \leq M,
\end{equation}
where $\lambda_{1,(ij)}, \lambda_{2,(ij)}, \dots, \lambda_{2N,(ij)}$ denote the eigenvalues of the discretized monotonicity matrix \(\mathbb A_{B_{ij}}\) and $\sigma > 0$ is a small predefined threshold introduced to enhance numerical stability.
We remark that $I_{count}(\mathbf{z}_{ij})$ is the discretized version of the continuous indicator ${\mathcal I}_{count}(B)$ given in \eqref{eq:indicatorContinuous}.

By Theorems~\ref{Characterofimpenetrable scatterer} and~\ref{Characterofimpenetrable scatterer2}, when the probing domain $B_{ij}$ centered at the sampling point $\mathbf{z}_{ij}$ satisfies $\overline{B_{ij}}\subseteq D$, the monotonicity operator $\mathcal{A}_{B_{ij}}$ possesses only finitely many negative eigenvalues, so that $I_{count}(\mathbf{z}_{ij})$ attains a relatively small value. Conversely, when $B_{ij}\not\subseteq D$, the number of negative eigenvalues is no longer guaranteed to be finite, and $I_{count}(\mathbf{z}_{ij})$ takes a relatively large value. Thus, the indicator $I_{count}$ distinguishes sampling points inside $D$ (small values) from those outside $D$ (large values), enabling shape reconstruction via scanning over all points in $\mathcal{N}_h$. This leads to Algorithm~\ref{alg:discrete monotonicity}, the counting-based monotonicity sampling method for shape imaging.

\begin{rem}
In the counting-based monotonicity indicator $I_{count}$ defined in \eqref{eq:indicatorD}, we count eigenvalues that are less than $-\sigma$ rather than directly counting all negative eigenvalues. This threshold is introduced to exclude eigenvalues near zero whose signs may be unreliable due to numerical discretization errors, thereby enhancing the stability of the imaging scheme.
\end{rem}

\begin{algorithm}[!htbp]
\caption{Counting-based monotonicity sampling method}
\label{alg:discrete monotonicity}
\begin{algorithmic}[1]
\State \textbf{Step 1:} Given $\omega, \lambda, \mu$, fix an incident direction $d$, and collect the corresponding far-field data generated by the scattering system.
\State \textbf{Step 2:} Choose a new incident direction $d \in \mathbb{S}$ and repeat Step 1 until all incident directions are covered.
\State \textbf{Step 3:} Using the data measured in Step 2, obtain the approximation  of the far-field operator $\mathbb{F}^d$ (or $\mathbb{F}^n$).
\State \textbf{Step 4:} Select a suitable grid partition $\mathcal{N}_h$ covering the scattering region $[-R, R]^2$. For each $\mathbf z_{ij} \in \mathcal{N}_h$, take the region $B_{ij}$ as the probing domain, and compute $\mathbb{H}_{B_{ij}}^* \mathbb{H}_{B_{ij}}$.
\State \textbf{Step 5:} Compute the eigenvalues of \(\mathbb{A}_{B_{ij}}\), and calculate the indicator \( I_{\mathrm{count}}(\mathbf z_{ij}) \).
\State \textbf{Step 6:} Select a new grid point $\mathbf z_{ij} \in \mathcal{N}_h$, repeat Steps 4–5 until all grid points are processed.
\State \textbf{Step 7:} Plot the indicator value at each sampling point.
\end{algorithmic}
\end{algorithm}

\begin{rem}\label{rem:al1}
During the validation of Algorithm~\ref{alg:discrete monotonicity} in Section~\ref{sec:Numerical example}, numerical experiments reveal that the counting-based monotonicity indicator $I_{count}$ is highly sensitive to noise. This sensitivity originates from the compactness of the monotonicity operator $\mathcal{A}_B$, compounded by the continuous dependence of eigenvalues on perturbations in the far-field data $\mathbf{u}^{\infty}$.

More precisely, denote by $\mathbf{u}^\infty_\delta$ and $\mathbb{F}_{\delta}$ the perturbed far-field pattern and the perturbed approximation matrix with noise level $\delta$, respectively. Let $\lambda$ and $\lambda_\delta$ be the eigenvalues of $\Re (\mathbb{F})$ and $\Re (\mathbb{F}_{\delta})$, respectively. Since $\Re (\mathbb{F})$ is a finite-dimensional Hermitian matrix, Weyl's monotonicity inequality~\cite{Kato13,SS90} ensures that these eigenvalues depend Lipschitz continuously on $\delta$. Consequently, the eigenvalues of $\mathbb{A}_{B}$ vary continuously with respect to perturbations in the far-field data.
However, since $\mathcal{A}_{B}$ is compact and self-adjoint,  its spectrum consists of real eigenvalues whose nonzero elements are isolated and can only accumulate at zero. Consequently, arbitrarily small perturbations in \(\mathbf{u}^{\infty}\) can induce sign changes in numerous eigenvalues clustered near zero, thereby undermining the reliability of counting-based imaging schemes. Nevertheless, eigenvalues near zero are of extremely small magnitudes, and their perturbations vary continuously with the noise level $\delta$. That is, the signs of the eigenvalues change discontinuously, while their magnitudes vary continuously with $\delta$. These observations motivate us to develop imaging functionals that are robust with respect to noise in Section~\ref{sec:Numerical algorithm}, which utilize the magnitudes of the eigenvalues rather than the sign information.

\end{rem}

\section{ summation-based monotonicity spectral sampling methods}\label{sec:Numerical algorithm}

In this section, we develop two novel monotonicity-based numerical algorithms for reconstructing the impenetrable scatterer. The algorithms are built upon
Theorems~\ref{Characterofimpenetrable scatterer} and~\ref{Characterofimpenetrable scatterer2}, and exploit the magnitudes of the negative eigenvalues of the monotonicity operator $\mathcal{A}_{B}$ defined in \eqref{operator:A_B} to design effective and robust indicator functionals.

In existing monotonicity-based methods (such as \cite{AG20,AG23}), including the indicator $\mathcal{I}_{count}(B)$ proposed in \eqref{eq:indicatorContinuous}, the shape is typically identified by counting the number of negative eigenvalues of $\mathcal{A}_{B}$. Such a criterion uses only the \emph{sign} information of the spectrum without accounting for the \emph{magnitudes} of the negative eigenvalues. Hence, the counting-based indicator may be regarded as a coarse measure that discards quantitative spectral information.

From the numerical results, we observe that when the probing domain $B$ approaches or intersects the true boundary $\partial D$, both the number and magnitudes of negative eigenvalues of $\mathcal{A}_B$ increase significantly, consistent with Theorems~\ref{Characterofimpenetrable scatterer} and~\ref{Characterofimpenetrable scatterer2}. Specifically, the dominant negative eigenvalues become substantially larger in magnitude when $B$ intersects $\partial D$ compared to $B \subseteq D$. This suggests that the boundary information is encoded in the negative spectral values of the monotonicity operator $\mathcal A_B$, rather than merely in the number of negative eigenvalues. Therefore, instead of considering only the counting indicator $\mathcal I_{count}(B)$ defined in \eqref{eq:indicatorContinuous}, it is natural to introduce an indicator defined via the sum of the negative part of the spectrum:
$$
\mathcal{I}_{sum}(B) := \sum_{\lambda_j(\mathcal{A}_B)<0} \lambda_j(\mathcal{A}_B),
$$
which quantifies the sum of the negative spectrum of \(\mathcal{A}_B\) and is more informative than the counting-based indicator \(\mathcal I_{count}\). In the noise-free case (Figure~\ref{fig:Dirichlet square}(a)), two probing domains $B_1 \subseteq D$ and $B_2 \subseteq \mathbb{R}^2 \setminus \overline{D}$ may yield the same number of negative eigenvalues. However, their negative eigenvalues can have substantially different magnitudes. In such a situation, the counting-based indicator \(\mathcal I_{count}\) fails to determine the boundary of $D$, whereas the summation-based indicator $\mathcal I_{sum}$ is still able to reconstruct $\partial D$, as illustrated in Figure~\ref{fig:Dirichlet square}(b).

The advantages of the summation-based indicator extend to the noisy regime. As discussed in Remark~\ref{rem:al1}, $\mathcal{I}_{sum}$ is more robust to noise than $\mathcal{I}_{count}$ because it exploits magnitude information, which varies continuously with perturbations. Numerical experiments confirm that $\mathcal{I}_{count}$ fails to reconstruct the scatterer shape under $10\%$ noise (Figure~\ref{fig:Dirichlet square}(d)), whereas $\mathcal{I}_{sum}$ stably reconstructs $\partial D$ at the same noise level (Figure~\ref{fig:Dirichlet square}(e)). Furthermore, since tiny eigenvalues near zero have signs that can flip under machine-precision perturbations, the counting-based indicator is sensitive to numerical rounding errors. In contrast, $\mathcal{I}_{sum}$ captures contributions from dominant negative spectral components and is therefore more resilient to round-off errors in practical implementations.

Therefore, instead of relying on the counting-based indicator $\mathcal I_{count}(B)$, we introduce a novel summation-based indicator $\mathcal I_{sum}(B)$ that quantifies the sum of the negative spectrum of $\mathcal{A}_{B}$, enabling more robust shape reconstruction of the scatterer \(D\). In the next subsection, we formally propose this new sampling method based on the negative spectral sum of the monotonicity operator \(\mathcal{A}_{B}\).

\subsection{Single-frequency monotonicity spectral sampling method}\label{subsec:soft monotonicity indicator}

In this subsection, we propose the \emph{single-frequency monotonicity spectral sampling method} described in Algorithm~\ref{alg:spectrum monotonicity}. Throughout this subsection, the incident wave is set at a single fixed angular frequency $\omega$, and the available data consist only of the far-field pattern $\mathbf{u}^\infty(\hat{x},d)$ measured at this fixed frequency $\omega$ for all observation directions $\hat{\mathbf x}\in\mathbb{S}$ and all incident
directions $\mathbf d\in\mathbb{S}$. 

We employ the grid partition \(\mathcal{N}_h\) and probing domains \(B_{ij}\) introduced in Subsection \ref{sub:Numerical approximation}. For each sampling point $\mathbf{z}_{ij}\in\mathcal{N}_h$, the discretized monotonicity matrix $\mathbb A_{B_{ij}}$ defined in \eqref{discretized operator:A_B} is assembled from the single-frequency far-field data via the discretized far-field operator $\mathbb{F}$ and the Herglotz wave operator $\mathbb{H}_{B_{ij}}$.

Let $\lambda_{1,(ij)}, \lambda_{2,(ij)}, \dots, \lambda_{2N,(ij)}$ be the eigenvalues of the discretized monotonicity matrix $\mathbb A_{B_{ij}}$. We define the \emph{single-frequency summation-based monotonicity indicator} $I_{sum}:\mathcal{N}_h	\to\mathbb{R}$ by
\begin{equation}\label{eq:soft-sec-indicator}
I_{sum}(\mathbf{z}_{ij}):= \operatorname{tr}\bigl((\mathbb A_{B_{ij}})_-\bigr)
=\sum_{\lambda_{k,(ij)} < 0}\lambda_{k,(ij)} ,
\end{equation}
where $(\mathbb A_{B_{ij}})_-$ denotes the negative part of $\mathbb A_{B_{ij}}$, defined spectrally by
$$
(\mathbb A_{B_{ij}})_-\,\phi_k
:=
\min(\lambda_{k,(ij)},\,0)\,\phi_k,
$$
with $\phi_k$ the eigenfunction corresponding to $\lambda_{k,(ij)}$.

The quantity $I_{sum}(\mathbf{z}_{ij})$ measures the sum of the negative spectrum of the discretized monotonicity matrix $\mathbb{A}_{B_{ij}}$ constructed from single-frequency far-field data. In contrast to the counting-based indicator $I_{count}$, which only detects the presence of negative eigenvalues, $I_{sum}$ exploits their magnitudes and is therefore expected to be more robust with respect to noise, a property that we verify numerically in the experiments below. Since both indicators require only single-frequency data, the method is computationally efficient. Based on the indicator $I_{sum}(\mathbf{z}_{ij})$ defined in \eqref{eq:soft-sec-indicator}, the detailed implementation steps of the newly developed algorithm are presented in Algorithm~\ref{alg:spectrum monotonicity}.

\begin{algorithm}[!htbp]
\caption{Single-frequency monotonicity  spectral sampling method}
\label{alg:spectrum monotonicity}
\begin{algorithmic}[1]
\State \textbf{Step 1:} Given $\omega, \lambda, \mu$, fix an incident direction $d$, and collect the corresponding far-field data generated by the scattering system.
\State \textbf{Step 2:} Choose a new incident direction $d \in \mathbb{S}$ and repeat Step 1 until all incident directions are covered.
\State \textbf{Step 3:} Using the data measured in Step 2, obtain the approximation  of the far-field operator $\mathbb{F}^d$ (or $\mathbb{F}^n$).
\State \textbf{Step 4:} Select a suitable grid partition $\mathcal{N}_h$ covering the scattering region $[-R, R]^2$. For each $\mathbf z_{ij} \in \mathcal{N}_h$, take the region $B_{ij}$ as the probing domain, and compute $\mathbb{H}_{B_{ij}}^* \mathbb{H}_{B_{ij}}$.
\State \textbf{Step 5:} Compute the eigenvalues of \(\mathbb {A}_{B_{ij}}\), and calculate the indicator \( I_{sum}(\mathbf z_{ij}) \).
\State \textbf{Step 6:} Select a new grid point $\mathbf z_{ij} \in \mathcal{N}_h$, repeat Steps 4–5 until all grid points are processed.
\State \textbf{Step 7:} Plot the indicator value at each sampling point.
\end{algorithmic}
\end{algorithm}

The numerical results in Figures~\ref{fig:Dirichlet cassini}(d) and~\ref{fig:Dirichlet cassini}(e) demonstrate that the single-frequency monotonicity spectral sampling method, based on the summation indicator $I_{sum}$ defined in \eqref{eq:soft-sec-indicator}, achieves superior numerical stability and more accurate boundary characterization compared to the counting-based method with indicator $I_{count}$ defined in \eqref{eq:indicatorD}. 
However, single-frequency far-field data carry limited geometric information because only one wavelength is used for probing. To exploit the multiscale geometric information available across different frequencies, we extend the framework and propose a \emph{multi-frequency monotonicity spectral sampling method} in next subsection.

\subsection{Multi-frequency monotonicity spectral sampling method}

The single-frequency indicator $I_{sum}(\mathbf{z}_{ij})$ defined in \eqref{eq:soft-sec-indicator} quantifies the sum of negative spectral of the discretized monotonicity matrix $\mathbb A_{B_{ij}}$ at a single prescribed frequency $\omega$. However, the spectral information encoded in $\mathbb A_{B_{ij}}$ is intrinsically scale-dependent: low frequencies are typically more stable and capture the coarse support of the scatterer $D$, while higher frequencies are more sensitive to fine boundary details and concave features. Restricting to a single frequency may lose useful information available at other scales. We therefore combine frequency-wise monotonicity spectral functionals into a multi-frequency indicator $I_{sum}^{\mathrm{MF}}$, which integrates multiscale geometric information for enhanced reconstruction.

We still employ the grid partition \(\mathcal{N}_h\) and probing domains \(B_{ij}\) introduced in Subsection \ref{sub:Numerical approximation}. For each sampling point $\mathbf{z}_{ij}\in\mathcal{N}_h$, we recall the definition of the discretized monotonicity matrix $\mathbb A_{B_{ij}}$ in \eqref{discretized operator:A_B}. For a fixed angular frequency $\omega > 0$, we introduce the frequency-dependent discretized monotonicity matrix $\mathbb A_{B_{ij}}(\omega)$ as follows:
\begin{equation*}
\mathbb A_{B_{ij}}(\omega)
:=
\begin{cases}
 -\Re\bigl(\mathbb{F}^d(\omega)\bigr) - \mathbb{H}_{B_{ij},\omega}^*\mathbb{H}_{B_{ij},\omega}, & \text{$D$ is rigid,} \\[6pt]
 \Re\bigl(\mathbb{F}^n(\omega)\bigr)  - \mathbb{H}_{B_{ij},\omega}^*\mathbb{H}_{B_{ij},\omega}, & \text{$D$ is traction-free,}
\end{cases}
\end{equation*}
where $\mathbb{F}^{d}(\omega)$ and $\mathbb{F}^{n}(\omega)$ denote the numerical approximations of the far-field operators at frequency $\omega$, and $\mathbb{H}_{B_{ij},\omega}$ denotes the numerical approximation of the Herglotz wave operator associated with the probing domain $B_{ij}$ at the same frequency. For each fixed $\omega$, the shape characterization results in Theorems~\ref{Characterofimpenetrable scatterer} and~\ref{Characterofimpenetrable scatterer2} assert that the negative spectral structure of $\mathbb A_{B_{ij}}(\omega)$ encodes information about the inclusion relation between $B$ and $D$.

Let $\omega_1, \dots, \omega_L > 0$ be a given set of frequencies. For each frequency $\omega_\ell \in \{ \omega_1, \dots, \omega_L\}$, let $\{\lambda_k(\mathbb{A}_{B_{ij}}(\omega_\ell))\}_{k=1,\dots,2N}$ denote the eigenvalues of the discretized monotonicity matrix $\mathbb{A}_{B_{ij}}(\omega_\ell)$. The frequency-wise monotonicity spectral functional is defined by
\begin{equation*}
S_{B_{ij}}(\omega_\ell) := \sum_{\lambda_k(\mathbb{A}_{B_{ij}}(\omega_\ell))<0} \lambda_k\bigl(\mathbb{A}_{B_{ij}}(\omega_\ell)\bigr), \qquad \ell = 1,\dots,L.
\end{equation*}
Note that $S_{B_{ij}}(\omega_\ell) \leq 0$, and when $L=1$, $S_{B_{ij}}(\omega_\ell)$ coincides with the single-frequency indicator $I_{sum}(\mathbf{z}_{ij})$ defined in \eqref{eq:soft-sec-indicator}.
We then define the \emph{multi-frequency summation-based monotonicity indicator} $I_{sum}^{\mathrm{MF}}:\mathcal{N}_h \to \mathbb{R}$ by
\begin{equation}\label{eq:mf-soft}
I_{sum}^{\mathrm{MF}}(\mathbf{z}_{ij}) := \sum_{\ell=1}^{L} a_\ell S_{B_{ij}}(\omega_\ell),
\end{equation}
where $a_\ell > 0$ are prescribed positive weights satisfying $\sum_{\ell=1}^{L} a_\ell = 1$. The indicator $I_{sum}^{\mathrm{MF}}(\mathbf{z}_{ij})$ is a weighted combination of the frequency-wise functionals $S_{B_{ij}}(\omega_\ell)$ with normalized weights, and satisfies $I_{sum}^{\mathrm{MF}}(\mathbf{z}_{ij}) \leq 0$.

The multi-frequency indicator $I_{sum}^{\mathrm{MF}}$ inherits the monotonicity-based interpretation from each single-frequency component: for every fixed $\omega_\ell$, the quantity $S_{B_{ij}}(\omega_\ell)$ measures the sum of the negative spectral of $\mathbb{A}_{B_{ij}}(\omega_\ell)$, and $I_{sum}^{\mathrm{MF}}(\mathbf{z}_{ij})$ aggregates these contributions across all prescribed frequencies in a balanced manner. 

The multi-frequency summation-based indicator $I_{sum}^{\mathrm{MF}}$ offers key advantages. By integrating spectral information across multiple frequencies, it captures both coarse geometric support from low-frequency data and fine boundary details from high-frequency data, achieving superior boundary localization compared to single-frequency methods. Moreover, $I_{sum}^{\mathrm{MF}}$ is more robust than $I_{sum}$ and $I_{count}$. First, it inherits the stability of the single-frequency indicator $I_{sum}$, which transforms the discontinuous counting process into a continuous summation process, thereby achieving continuity with respect to perturbations in the far-field data (as discussed in Remark~\ref{rem:al1}). Second, the multi-frequency aggregation further mitigates the influence of noise, since noise-induced perturbations at individual frequencies are unlikely to affect all frequencies simultaneously in a coherent manner. By combining these two mechanisms, $I_{sum}^{\mathrm{MF}}$ achieves enhanced robustness compared to both single-frequency and counting-based methods.

Based on the indicator $I_{sum}^{\mathrm{MF}}(\mathbf{z}_{ij})$ defined in \eqref{eq:mf-soft}, the detailed implementation steps of the \emph{multi-frequency monotonicity spectral sampling method} are presented in Algorithm~\ref{alg:multi-frequency spectrum monotonicity}.

\begin{algorithm}[!htbp]
\caption{Multi-frequency monotonicity spectral sampling method}
\label{alg:multi-frequency spectrum monotonicity}
\begin{algorithmic}[1]
\State \textbf{Step 1:} Given $\lambda, \mu$ and  \(L\) frequencies $\omega_1, \omega_2, \dots, \omega_L$. Fix an incident direction $d$ and a frequency $\omega_\ell$, and collect the corresponding far-field data generated by the scattering system.
\State \textbf{Step 2:} For the frequency $\omega_\ell$ chosen in Step 1, choose a new incident direction $d \in \mathbb{S}$ and repeat Step 1 until all incident directions are covered.
\State \textbf{Step 3:} Using the data measured in Step 2, obtain the approximation  of the far-field operator $\mathbb{F}^d(\omega_\ell)$ (or $\mathbb{F}^n(\omega_\ell)$).
\State \textbf{Step 4:} Select a suitable grid partition $\mathcal{N}_h$ covering the scattering region $[-R, R]^2$. For each $\mathbf z_{ij} \in \mathcal{N}_h$, take the region $B_{ij}$ as the probing domain, and compute $\mathbb{H}_{B_{ij},\omega_\ell}^* \mathbb{H}_{B_{ij,\omega_\ell}}$.
\State \textbf{Step 5:} Choose a new frequency $\omega_\ell$ and repeat Steps 1- 4 until all $L$ frequencies under consideration are exhausted.
\State \textbf{Step 5:} Calculate the indicator function \( I_{sum}^{\mathrm{MF}}(\mathbf z_{ij}) \).
\State \textbf{Step 6:} Select a new grid point $\mathbf z_{ij} \in \mathcal{N}_h$, repeat Steps 4–5 until all grid points are processed.
\State \textbf{Step 7:} Plot the indicator function value at each sampling point.
\end{algorithmic}
\end{algorithm}

At the end of this section, we present a remark to derive the equivalent formulation of the three sampling methods proposed in this work.

\begin{rem}
For rigid impenetrable scatterers, the indicators $I_{count}$, $I_{sum}$, and $I_{sum}^{\mathrm{MF}}$ are defined based on the negative eigenvalues of the discretized monotonicity matrix $-\Re(\mathbb{F}^d)-\mathbb{H}_B^*\mathbb{H}_B$, in accordance with Theorems~\ref{Characterofimpenetrable scatterer} and~\ref{Characterofimpenetrable scatterer2}. Equivalently, these indicators can be computed using the positive eigenvalues of $\Re(\mathbb{F}^d)+\mathbb{H}_B^*\mathbb{H}_B$. All numerical results for rigid impenetrable scatterers in Section~\ref{sec:Numerical example} were computed using this equivalent formulation.
\end{rem}

\section{ Numerical experiments}\label{sec:Numerical example}
In this section, we present extensive numerical experiments to validate the effectiveness and robustness of the three proposed algorithms: the counting-based monotonicity sampling method (Algorithm~\ref{alg:discrete monotonicity}), the single-frequency monotonicity spectral sampling method (Algorithm~\ref{alg:spectrum monotonicity}), and the multi-frequency monotonicity spectral sampling method (Algorithm~\ref{alg:multi-frequency spectrum monotonicity}).

For the forward elastic scattering problem, the synthetic far-field data are generated by solving the two-dimensional Navier equation subject to Dirichlet and Neumann boundary conditions via the boundary integral equation method. To simulate measurement errors, we add complex Gaussian noise to the synthetic far-field operator as
\begin{equation*}
\mathbf{u}^{\infty}_\delta := \mathbf{u}^\infty  + \delta\|\mathbf{u}^\infty\| \, \boldsymbol{\eta},
\label{eq:noisy_data}
\end{equation*}
where $\boldsymbol{\eta}$ is a complex-valued random vector and drawn from the standard complex normal distribution $\mathcal{CN}(0,1)$, and $\delta > 0$ denotes the relative noise level (e.g., $\delta = 0.01$ corresponds to $1\%$ noise). Selected experiments are conducted at multiple noise levels to systematically examine the noise robustness of each algorithm.

\subsection{Numerical examples for rigid impenetrable scatterers}
In this subsection, we treat $D$ as a rigid impenetrable scatterer subject to the Dirichlet boundary condition \eqref{eq:DirichletBC}. We first validate the counting-based monotonicity sampling method (Algorithm~\ref{alg:discrete monotonicity}) in Example~\ref{ex:Dirichlet discrete}, examining both its reconstruction performance at different frequencies and its noise sensitivity. In Example~\ref{ex:Dirichlet spectrum}, we then evaluate the counting-based monotonicity sampling method (Algorithm~\ref{alg:discrete monotonicity}), the single-frequency and multi-frequency monotonicity spectral sampling methods (Algorithms~\ref{alg:spectrum monotonicity} and \ref{alg:multi-frequency spectrum monotonicity}) for various impenetrable scatterers, and present a comprehensive comparison of the noise robustness and reconstruction quality across all three algorithms.

\begin{example}\label{ex:Dirichlet discrete}

We consider two rigid impenetrable scatterers: a Cassini-shaped and a kite-shaped scatterer, which are embedded in a homogeneous isotropic background medium with Lam\'{e} parameters $\lambda = 1$ and $\mu = 5$. A uniform rectangular sampling grid $\mathcal{N}_h$ with step size $h = 0.01$ over $[-2,2]^2$ yields $401 \times 401$ grid points. The far-field data \(\mathbf u^{\infty}\) that form matrix \(\mathbb{F}^d\) is calculated via the Nystr\"{o}m method, and we set the threshold parameter \(\sigma = 2 \times 10^{-10}\) for \(I_{count}\) (defined in \eqref{eq:indicatorD}).

We first investigate the reconstruction quality at two frequencies, $\omega = 10$ and $\omega = 50$, for the Cassini-shaped scatterer using $N = 32$ incident and observation directions. As shown in Figure~\ref{fig:cassini}, the reconstruction quality improves significantly at higher frequencies, a result consistent with the fact that shorter wavelengths resolve finer geometric features.

Figure~\ref{fig:kite discrete dirichlet}(a) displays the counting-based indicator $I_{count}(\mathbf{z}_{ij})$ for the kite-shaped scatterer at $\omega = 12$ with $N = 32$ directions in the noise-free case. The lowest level set provides a sharp approximation to the true boundary, accurately capturing both the overall shape and local geometric features. We next corrupt $\mathbb{F}^d$ with additive Gaussian noise at relative noise level $\delta = 0.1\%$. Figure~\ref{fig:kite discrete dirichlet}(b) shows that the counting-based indicator becomes compromised, with visible loss of boundary sharpness. This performance degradation reveals the inherent sensitivity of binary counting to small eigenvalue perturbations, motivating the use of summation-based indicators $I_{sum}$ and $I_{sum}^{\mathrm{MF}}$ in subsequent examples.

\end{example}

\begin{figure}[htbp]
    \centering
    \begin{minipage}{0.45\textwidth}
        \centering
        \includegraphics[width=\linewidth]{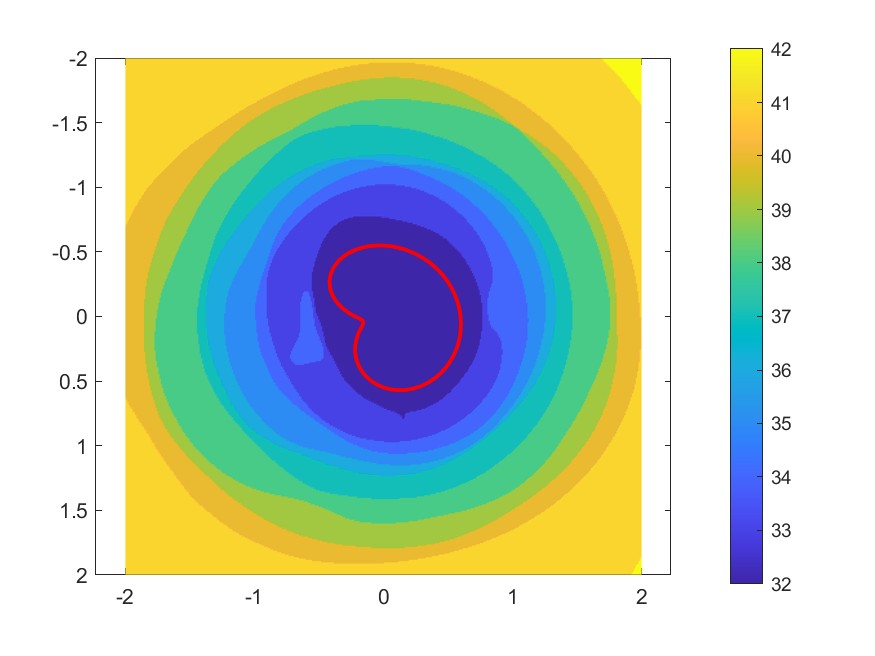}
        \vspace{-0.5em} 
        \centering (a)
    \end{minipage}
    \hfill
   \begin{minipage}{0.45\textwidth}
        \centering
        \includegraphics[width=\linewidth]{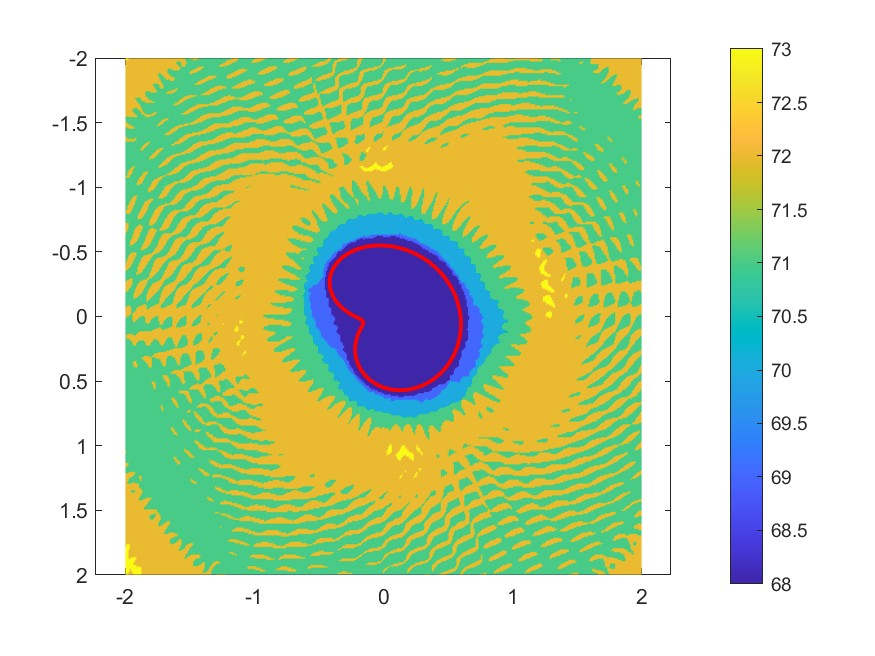}
        \vspace{-0.5em}
        \centering (b)
    \end{minipage}
   \caption{Reconstruction results for the Cassini-shaped rigid impenetrable scatterer obtained by the counting-based monotonicity sampling method (Algorithm~\ref{alg:discrete monotonicity}). (a): $\omega = 10$, $N = 32$. (b): $\omega = 50$, $N = 32$. The solid red curve denotes the true boundary $\partial D$.}
    \label{fig:cassini}
\end{figure}

\begin{figure}[htbp]
    \centering
    \begin{minipage}{0.45\textwidth}
        \centering
        \includegraphics[width=\textwidth]{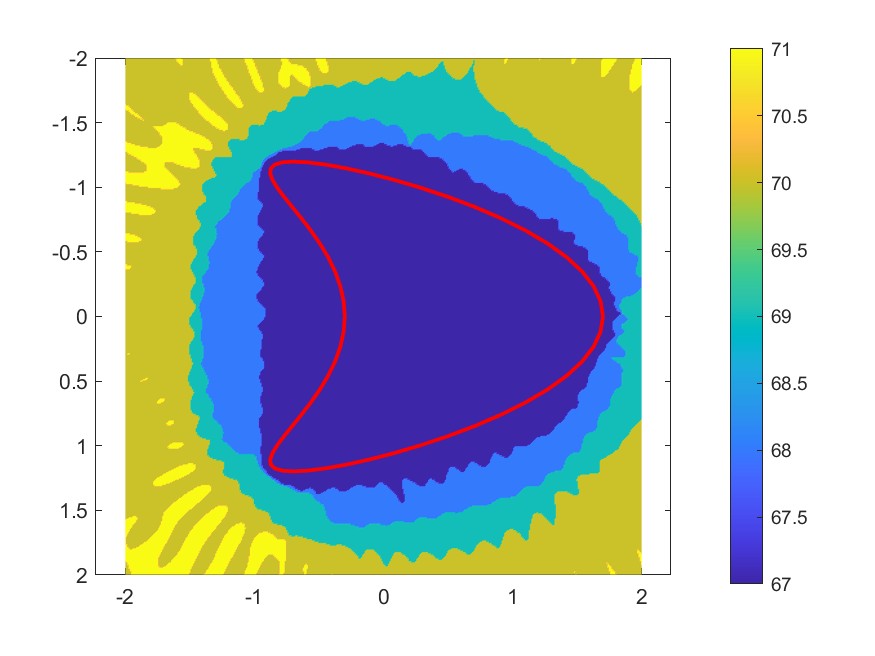}
        \vspace{-0.5em} 
        \centering (a)
    \end{minipage}
    \hfill
   \begin{minipage}{0.45\textwidth}
        \centering
        \includegraphics[width=\textwidth]{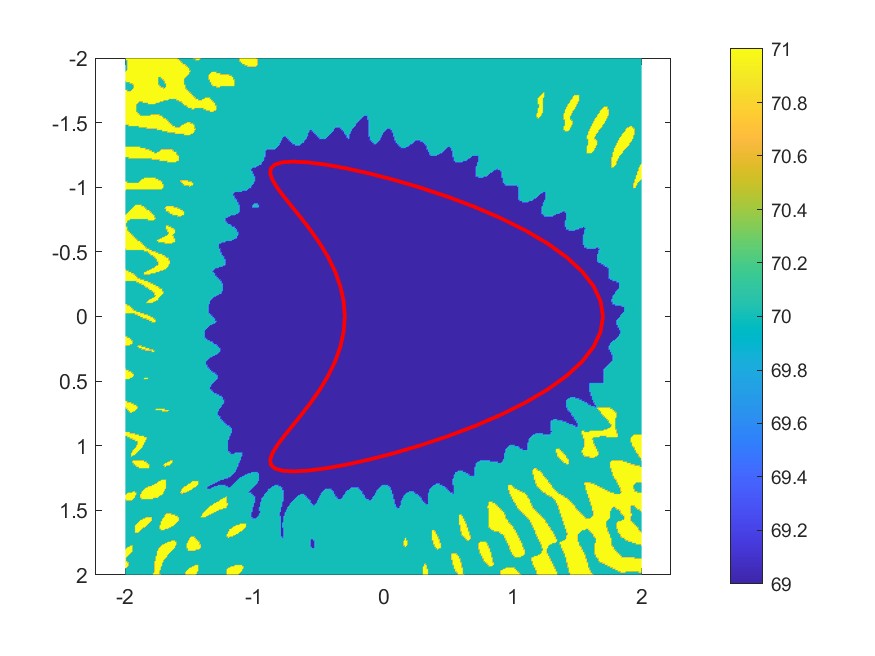}
        \vspace{-0.5em}
        \centering (b)
    \end{minipage}
   \caption{Reconstruction results for the kite-shaped rigid impenetrable scatterer obtained by the counting-based monotonicity sampling method (Algorithm~\ref{alg:discrete monotonicity}) at $\omega = 30$, $N = 32$. (a): Noise-free case ($\delta = 0$). (b): Noisy case with relative noise level $\delta = 0.1\%$. The solid red curve denotes the true boundary  $\partial D$.}
    \label{fig:kite discrete dirichlet}
\end{figure}

Example~\ref{ex:Dirichlet discrete} demonstrates that the counting-based monotonicity sampling method (Algorithm~\ref{alg:discrete monotonicity}) produces acceptable reconstructions at a relative noise level of $\delta = 0.1\%$ (see Figure~\ref{fig:kite discrete dirichlet}). However, the performance deteriorates significantly as the noise level increases. At $\delta = 10\%$, the indicator $I_{count}$ fails to produce meaningful reconstructions, making it difficult to determine the scatterer's location and shape (Figures~\ref{fig:Dirichlet square}(d), \ref{fig:Dirichlet cassini}(d), and \ref{fig:Dirichlet kite}(d)). This degradation is consistent with the theoretical observation that the binary counting mechanism is inherently discontinuous with respect to eigenvalue perturbations, rendering clustered eigenvalues near zero vulnerable to sign changes induced by noise.

In contrast, the newly proposed single-frequency monotonicity spectral sampling method (Algorithm~\ref{alg:spectrum monotonicity}) and multi-frequency monotonicity spectral sampling method (Algorithm~\ref{alg:multi-frequency spectrum monotonicity}) exploit the magnitudes of the negative eigenvalues of the monotonicity operator $\mathcal{A}_{ij}$ rather than relying on binary sign-counting. Consequently, these methods are expected to exhibit superior robustness with respect to noise, a property that will be demonstrated explicitly in the subsequent numerical experiments.

\begin{example}\label{ex:Dirichlet spectrum}
We consider three rigid impenetrable scatterers: a square, a Cassini-shaped , and a kite-shaped scatterer, which are embedded in a homogeneous isotropic background medium with Lam\'{e} parameters $\lambda = 0.2$ and $\mu = 1$. A uniform rectangular sampling grid $\mathcal{N}_h$ with step size $h = 0.02$ over $[-2,2]^2$ yields $201 \times 201$ grid points. We employ $N = 32$ incident and observation directions throughout.

For the counting-based monotonicity indicator \(I_{count}\), the threshold $\sigma = 1 \times 10^{-12}$. Both the counting-based monotonicity sampling method (Algorithm~\ref{alg:discrete monotonicity}) and single-frequency monotonicity spectral sampling method (Algorithm~\ref{alg:spectrum monotonicity}) adopt the fixed frequency $\omega = 18$. For the multi-frequency monotonicity spectral sampling method (Algorithm~\ref{alg:multi-frequency spectrum monotonicity}), we employ four frequencies $\omega \in \{12, 14, 16, 18\}$ with uniform weights $a_\ell = 1/4$ in the indicator $I_{sum}^{\mathrm{MF}}(\mathbf{z}_{ij})$.
The three indicators $I_{count}$, $I_{sum}$, and $I_{sum}^{\mathrm{MF}}$ are computed and visualized as color-coded maps for each scatterer. Figures~\ref{fig:Dirichlet square}, \ref{fig:Dirichlet cassini}, and \ref{fig:Dirichlet kite} display reconstruction results for the three scatterers, with the first row showing noise-free cases ($\delta = 0$) and the second row showing results under $\delta = 10\%$ additive noise.

In the noise-free case, the counting-based monotonicity sampling method correctly identifies the location and overall geometry of each scatterer (Figures~\ref{fig:Dirichlet square}(a), \ref{fig:Dirichlet cassini}(a), and \ref{fig:Dirichlet kite}(a)); however, the reconstructed images exhibit visible artifacts, particularly for scatterers with concave boundaries such as the Cassini-shaped and kite-shaped domains (Figures~\ref{fig:Dirichlet cassini}(a) and \ref{fig:Dirichlet kite}(a)). This difficulty is attributed to high-curvature concave structures that generate stronger multiple scattering and diffraction effects, introducing perturbations into the far-field data. By contrast, both spectral sampling methods yield substantially cleaner reconstructions by exploiting multiscale spectral information to mitigate these scattering-induced perturbations. The single-frequency indicator $I_{sum}$ concentrates its largest values near the true boundary $\partial D$, although isolated bright spots may appear in the interior (Figures~\ref{fig:Dirichlet square}(b), \ref{fig:Dirichlet cassini}(b), and ef{fig:Dirichlet kite}(b)). The multi-frequency indicator $I_{sum}^{\mathrm{MF}}$ further improves reconstruction quality by suppressing interior artifacts and producing sharper boundary profiles, particularly for concave regions (Figures~\ref{fig:Dirichlet square}(c), \ref{fig:Dirichlet cassini}(c), and \ref{fig:Dirichlet kite}(c)). This improvement is attributed to multiscale aggregation of spectral information across frequencies, enabling the capture of both coarse geometric features at low frequencies and fine boundary details at high frequencies, while effectively resolving features that are challenging for the counting-based approach.

Under $\delta = 10\%$ additive noise, the counting-based method deteriorates dramatically, with noise-induced artifacts obscuring the scatterer location (Figures~\ref{fig:Dirichlet square}(d), \ref{fig:Dirichlet cassini}(d), and \ref{fig:Dirichlet kite}(d)). This degradation is consistent with the theoretical observation that binary counting is inherently discontinuous with respect to eigenvalue perturbations. The single-frequency spectral method is considerably more stable, preserving main structural features despite noticeable blurring (Figures~\ref{fig:Dirichlet square}(e), \ref{fig:Dirichlet cassini}(e), and \ref{fig:Dirichlet kite}(e)). The multi-frequency method exhibits the strongest robustness, maintaining clear boundary localization and accurate recovery of concave features even under substantial noise contamination (Figures~\ref{fig:Dirichlet square}(f), \ref{fig:Dirichlet cassini}(f), and \ref{fig:Dirichlet kite}(f)). These results demonstrate that the spectral methods' advantage over counting-based approaches is particularly pronounced for geometrically complex scatterers with concave boundaries, where the multiscale aggregation effectively resolves features that challenge the binary counting mechanism.

\end{example}

The numerical results in Example \ref{ex:Dirichlet spectrum} reveal a clear hierarchy in reconstruction quality and robustness among the three proposed algorithms. The counting-based monotonicity sampling method is directly derived from Theorem \ref{Characterofimpenetrable scatterer}, but exhibits limited stability under noisy data and limited resolution of fine geometric features, see Figures~\ref{fig:Dirichlet square}(d), \ref{fig:Dirichlet cassini}(d) and \ref{fig:Dirichlet kite}(d). The single-frequency monotonicity spectral sampling method provides a marked improvement in both numerical stability and imaging quality by exploiting the magnitudes of the negative eigenvalues of $\mathcal{A}_{ij}$, see Figures~\ref{fig:Dirichlet square}(e), \ref{fig:Dirichlet cassini}(e) and \ref{fig:Dirichlet kite}(e). The multi-frequency monotonicity spectral sampling method achieves the best overall performance across all tested scenarios, delivering the highest reconstruction fidelity, the strongest robustness to additive noise, and the most accurate
recovery of concave boundary structures, see Figures~\ref{fig:Dirichlet square}(f), \ref{fig:Dirichlet cassini}(f) and \ref{fig:Dirichlet kite}(f). These observations confirm the effectiveness of the proposed spectral indicators $I_{sum}$ and $I_{sum}^{\mathrm{MF}}$ as improvements over the classical counting-based criterion $I_{count}$.

\begin{figure}[htbp]
    \centering
    
    \begin{minipage}{0.32\textwidth}
        \centering
        \includegraphics[width=\linewidth]{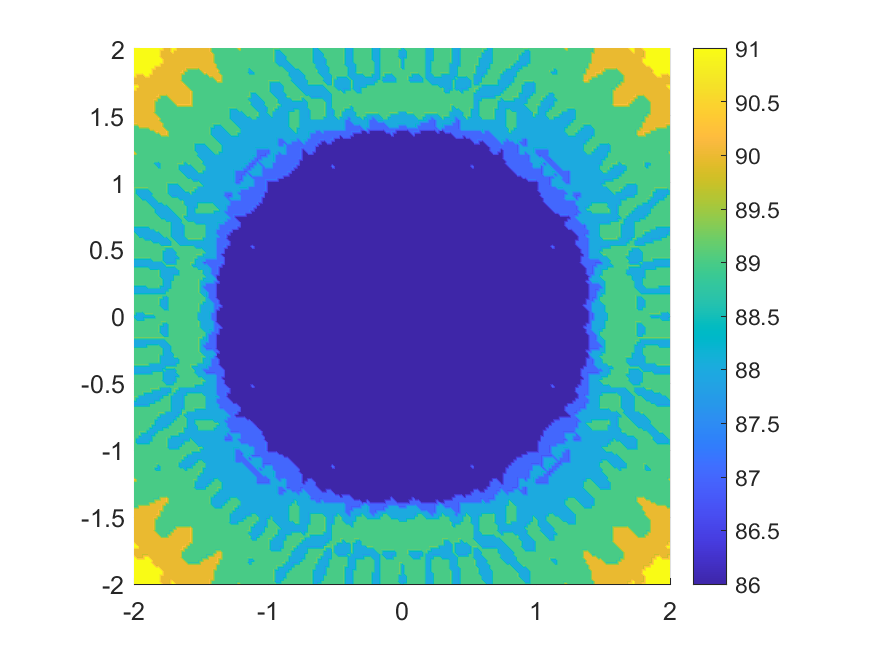}
        \vspace{-0.5em} 
        \centering (a)
    \end{minipage}
    \hfill
   \begin{minipage}{0.32\textwidth}
        \centering
        \includegraphics[width=\linewidth]{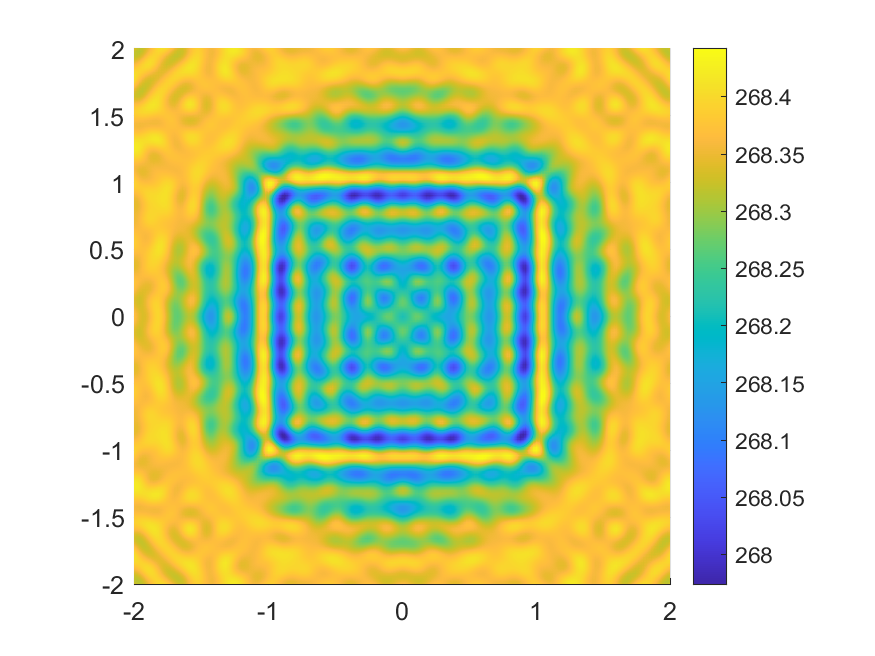}
        \vspace{-0.5em}
        \centering (b)
    \end{minipage}
    \hfill
    \begin{minipage}{0.32\textwidth}
        \centering
        \includegraphics[width=\linewidth]{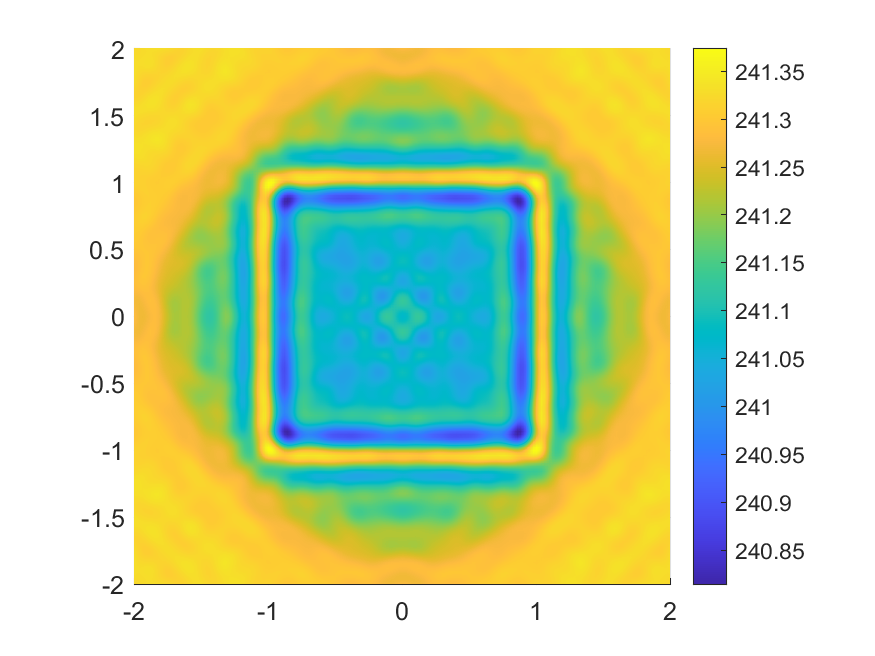}
        \vspace{-0.5em}
        \centering (c)
    \end{minipage}

    \vspace{1em}

    \begin{minipage}{0.32\textwidth}
        \centering
        \includegraphics[width=\linewidth]{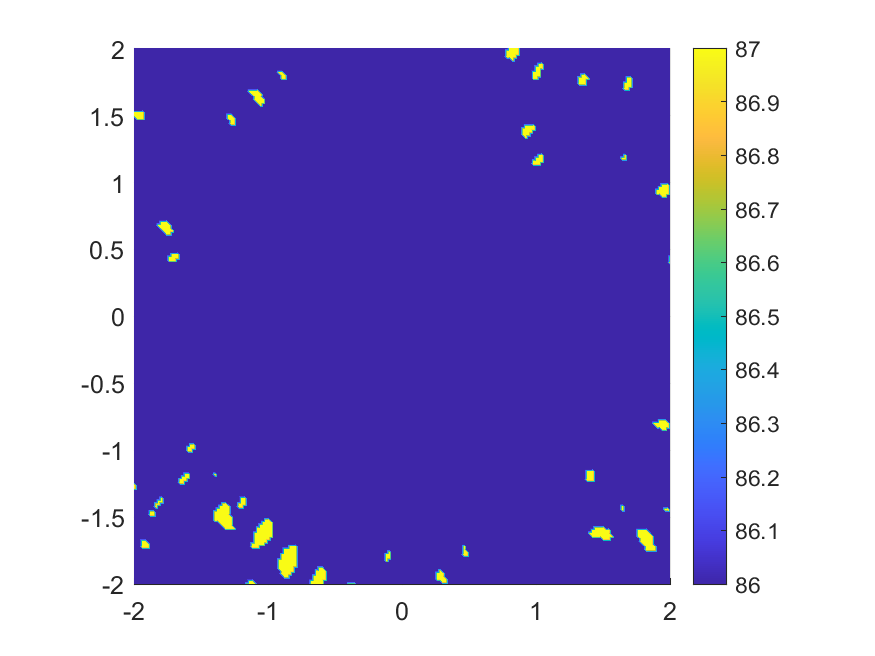}
        \vspace{-0.5em}
        \centering (d)
    \end{minipage}
    \hfill
    \begin{minipage}{0.32\textwidth}
        \centering
        \includegraphics[width=\linewidth]{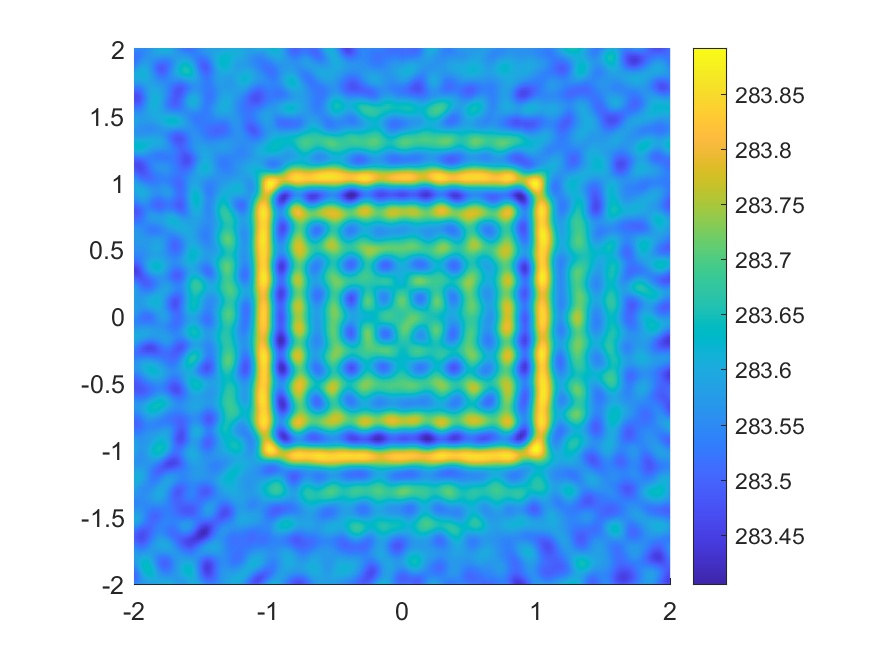}
        \vspace{-0.5em}
        \centering (e)
    \end{minipage}
    \hfill
    \begin{minipage}{0.32\textwidth}
        \centering
        \includegraphics[width=\linewidth]{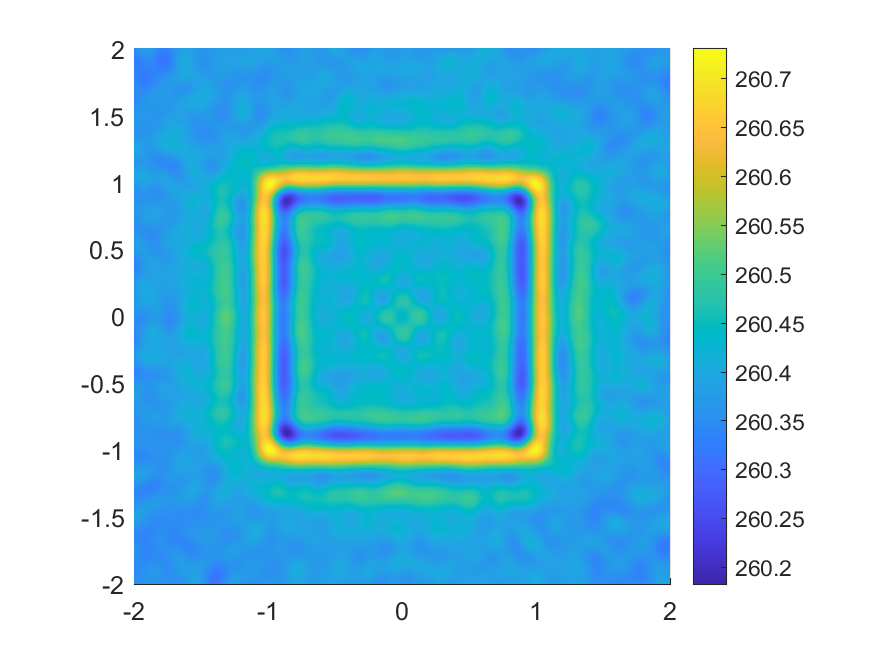}
        \vspace{-0.5em}
        \centering (f)
    \end{minipage}
    \caption{Reconstruction results for the square-shaped rigid impenetrable scatterer obtained by the three proposed sampling methods at $\omega = 18$, $N = 32$. The first row corresponds to the noise-free case ($\delta = 0$), and the second row corresponds to data contaminated by additive noise at relative noise level $\delta = 10\%$. (a), (d): counting-based monotonicity sampling method (Algorithm~\ref{alg:discrete monotonicity}). (b), (e): single-frequency monotonicity spectral sampling method (Algorithm~\ref{alg:spectrum monotonicity}). (c), (f): multi-frequency monotonicity spectral sampling method (Algorithm~\ref{alg:multi-frequency spectrum monotonicity}).}
    \label{fig:Dirichlet square}
\end{figure}
\begin{figure}[htbp]
    \centering
    \begin{minipage}{0.32\textwidth}
        \centering
        \includegraphics[width=\linewidth]{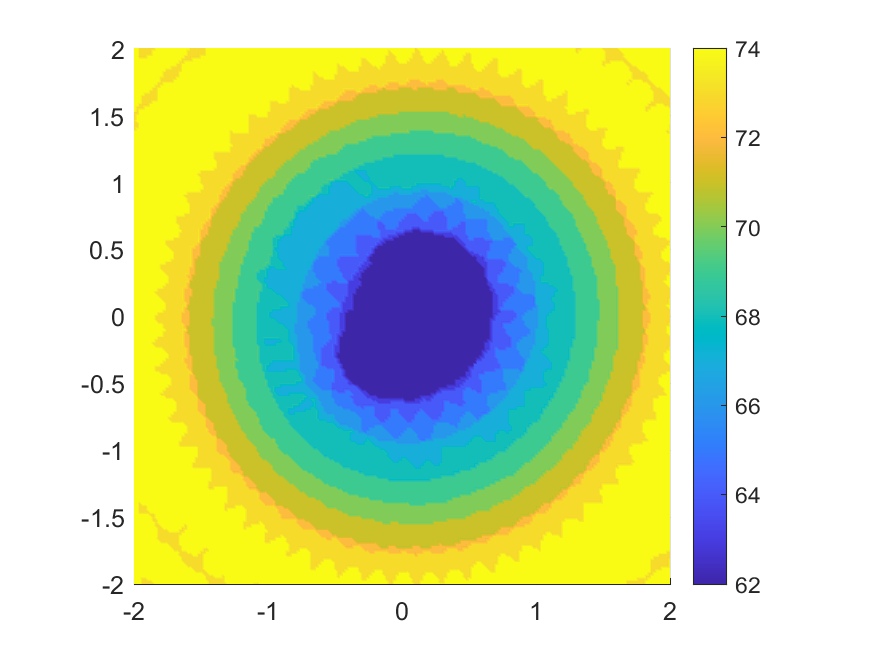}
        \vspace{-0.5em} 
        \centering (a)
    \end{minipage}
    \hfill
   \begin{minipage}{0.32\textwidth}
        \centering
        \includegraphics[width=\linewidth]{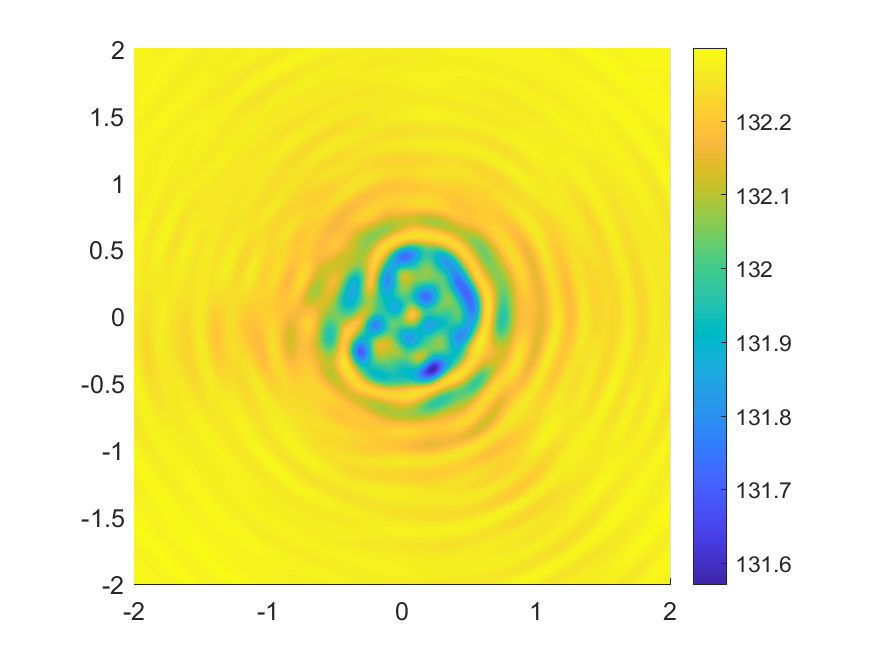}
        \vspace{-0.5em}
        \centering (b)
    \end{minipage}
    \hfill
    \begin{minipage}{0.32\textwidth}
        \centering
        \includegraphics[width=\linewidth] {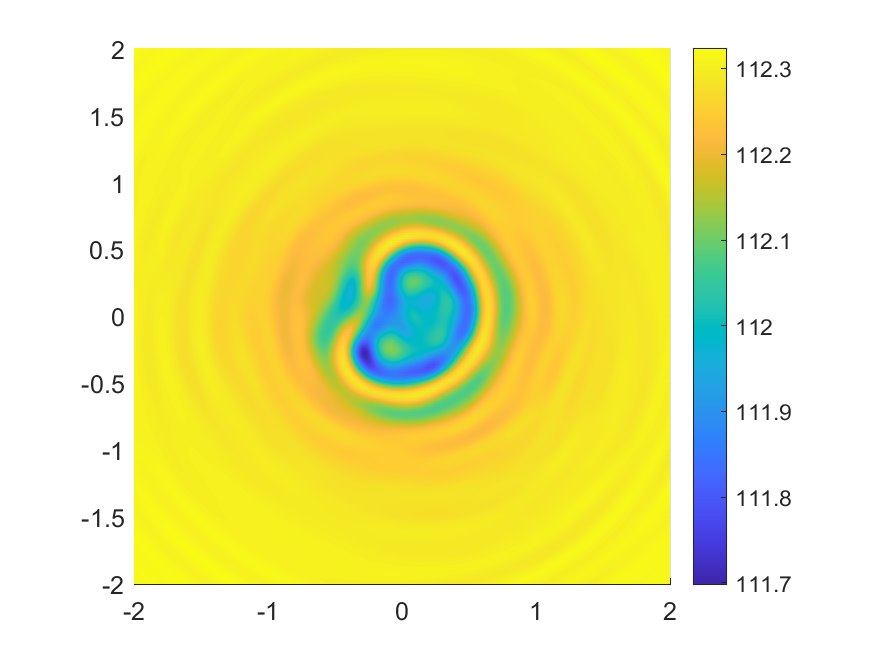}
        \vspace{-0.5em}
        \centering (c)
    \end{minipage}

    \vspace{1em}

    \begin{minipage}{0.32\textwidth}
        \centering
        \includegraphics[width=\linewidth]{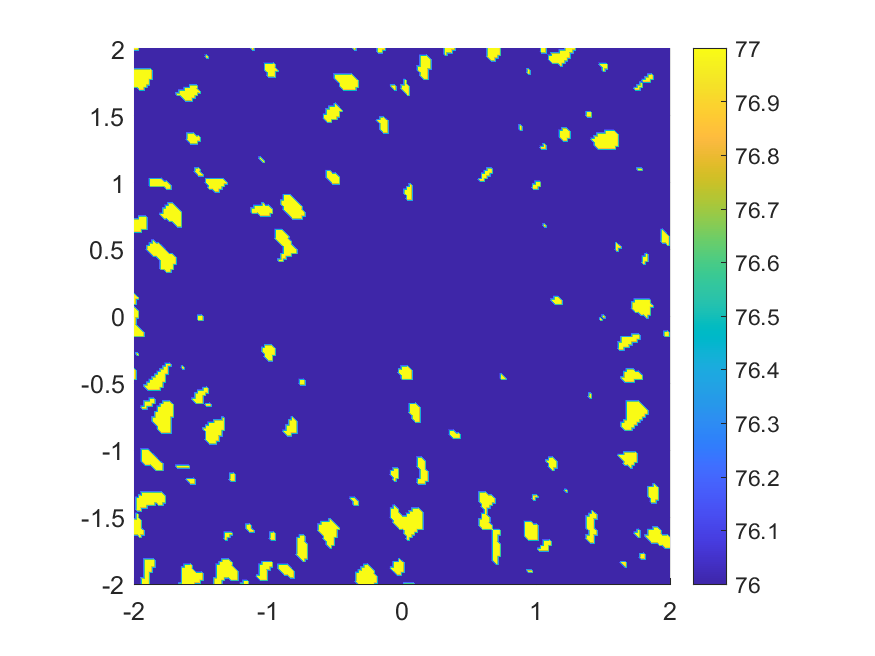}
        \vspace{-0.5em}
        \centering (d)
    \end{minipage}
    \hfill
    \begin{minipage}{0.32\textwidth}
        \centering
        \includegraphics[width=\linewidth]{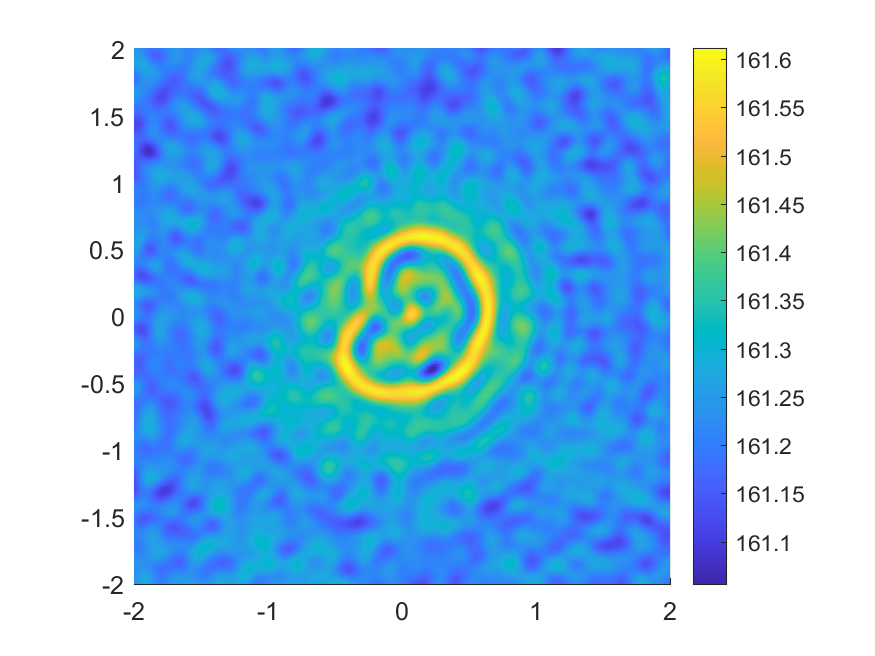}
        \vspace{-0.5em}
        \centering (e)
    \end{minipage}
    \hfill
    \begin{minipage}{0.32\textwidth}
        \centering
        \includegraphics[width=\linewidth] {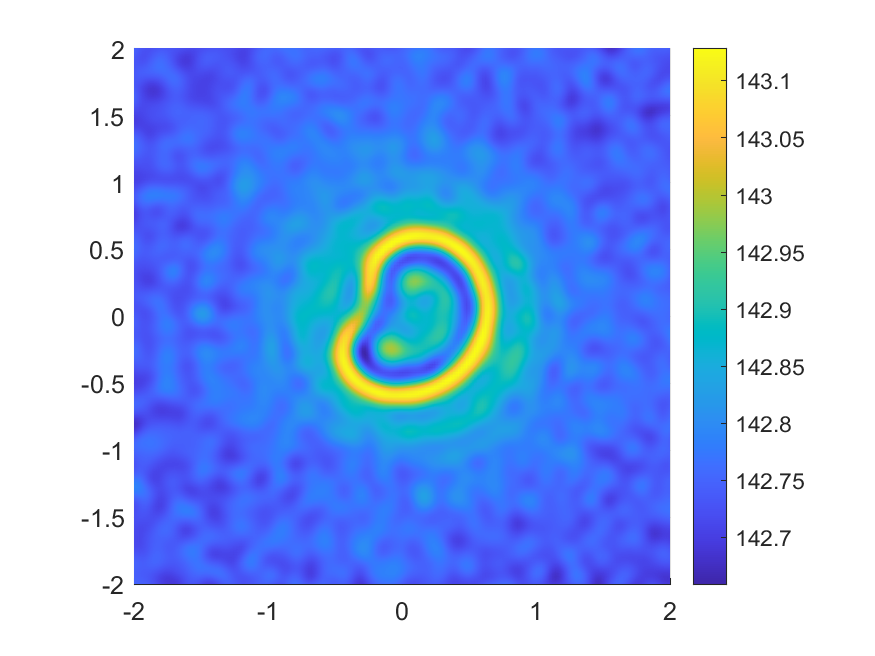}
        \vspace{-0.5em}
        \centering (f)
    \end{minipage}
    
    \caption{Reconstruction results for the Cassini-shaped rigid impenetrable scatterer obtained by the three proposed sampling methods at $\omega = 18$, $N = 32$. The first row corresponds to the noise-free case ($\delta = 0$), and the second row corresponds to data contaminated by additive noise at relative noise level $\delta = 10\%$. (a), (d): counting-based monotonicity sampling method (Algorithm~\ref{alg:discrete monotonicity}). (b), (e): single-frequency monotonicity spectral sampling method (Algorithm~\ref{alg:spectrum monotonicity}). (c), (f): multi-frequency monotonicity spectral sampling method (Algorithm~\ref{alg:multi-frequency spectrum monotonicity}). }
    \label{fig:Dirichlet cassini}
\end{figure}
\begin{figure}[htbp]
    \centering
   \begin{minipage}{0.32\textwidth}
        \centering
        \includegraphics[width=\linewidth]{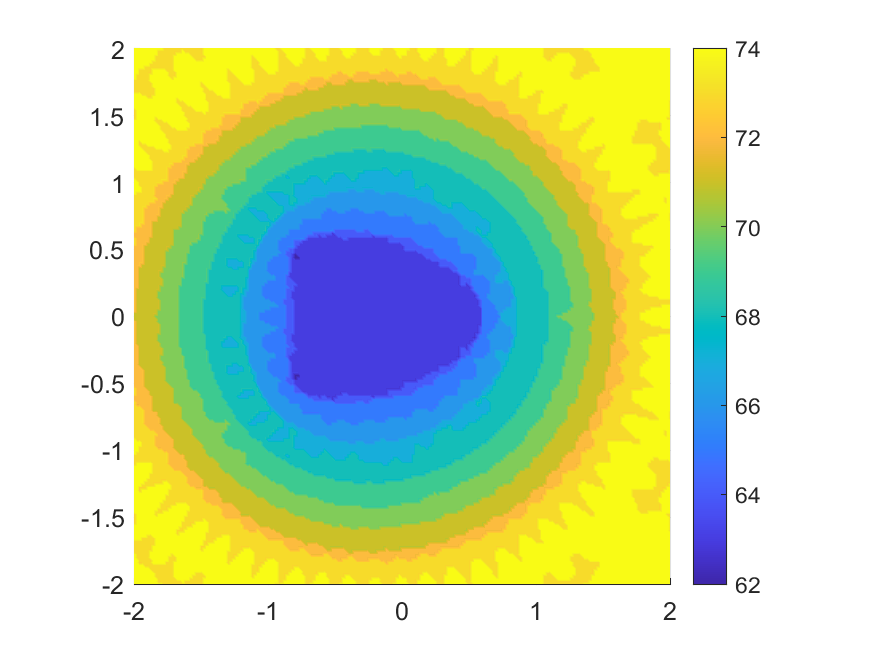}
        \vspace{-0.5em} 
        \centering (a)
    \end{minipage}
    \hfill
   \begin{minipage}{0.32\textwidth}
        \centering
        \includegraphics[width=\linewidth]{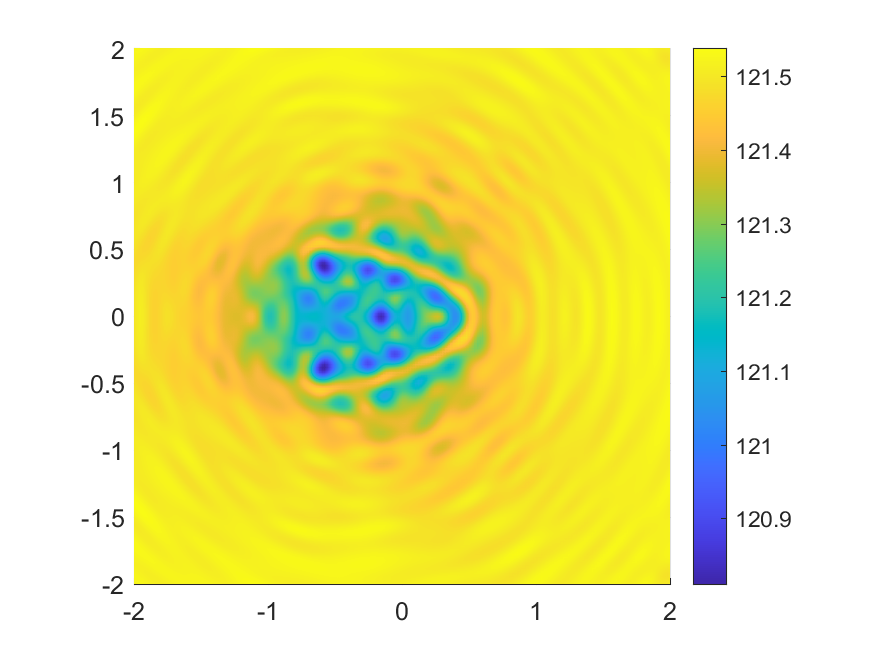}
        \vspace{-0.5em}
        \centering (b)
    \end{minipage}
    \hfill
    \begin{minipage}{0.32\textwidth}
        \centering
        \includegraphics[width=\linewidth] {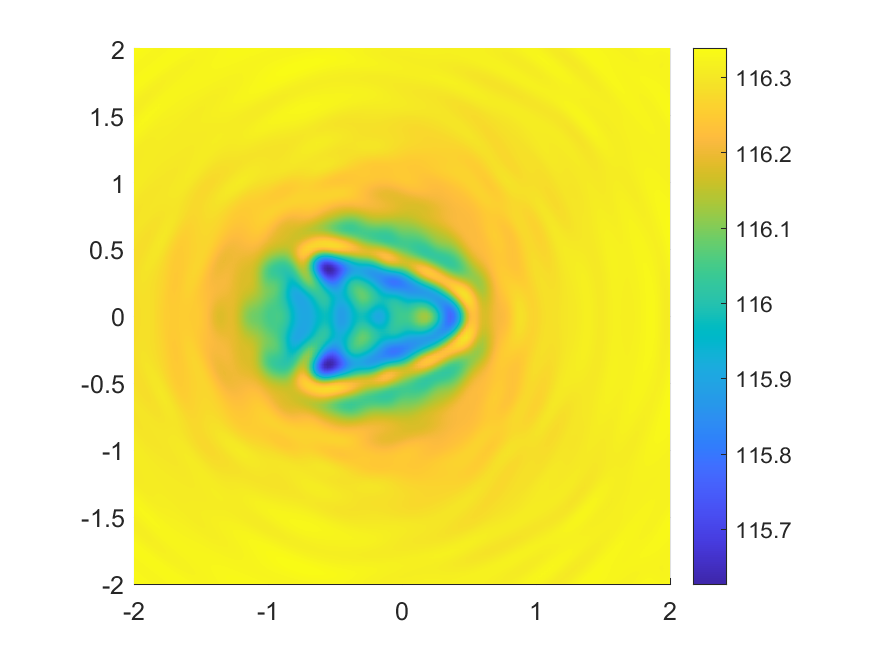}
        \vspace{-0.5em}
        \centering (c)
    \end{minipage}

    \vspace{1em}

    \begin{minipage}{0.32\textwidth}
        \centering
        \includegraphics[width=\linewidth]{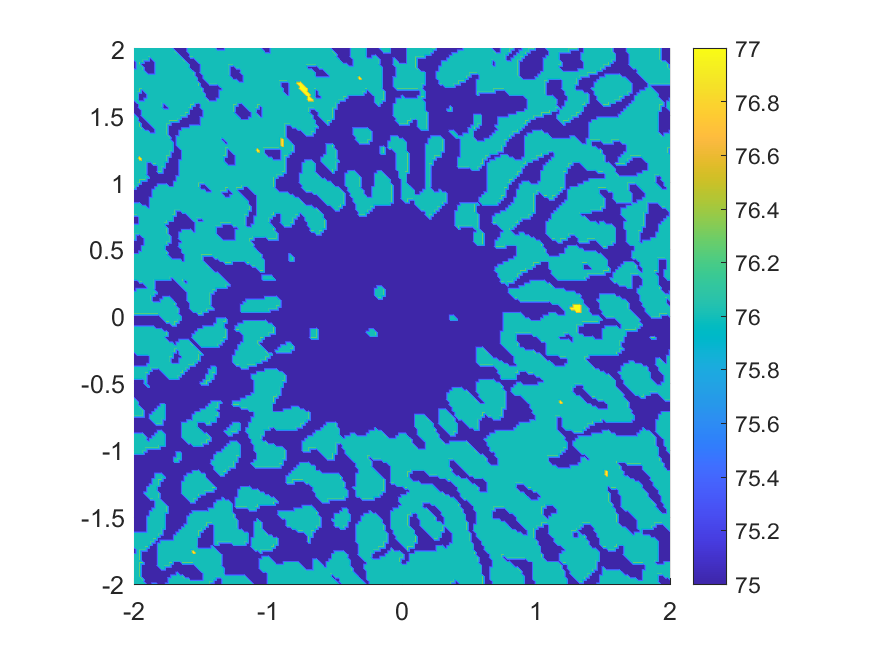}
        \vspace{-0.5em}
        \centering (d)
    \end{minipage}
    \hfill
    \begin{minipage}{0.32\textwidth}
        \centering
        \includegraphics[width=\linewidth]{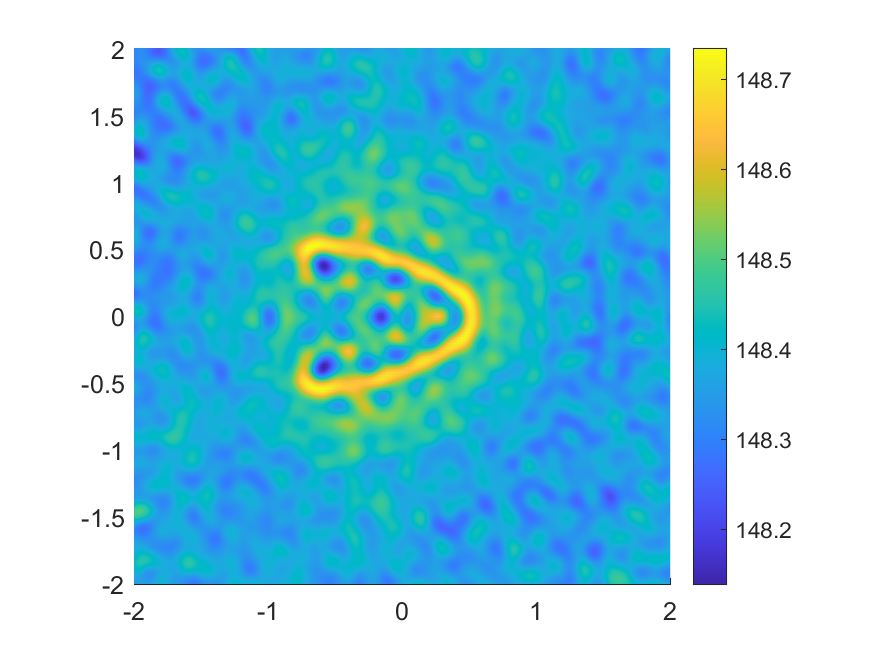}
        \vspace{-0.5em}
        \centering (e)
    \end{minipage}
    \hfill
    \begin{minipage}{0.32\textwidth}
        \centering
        \includegraphics[width=\linewidth] {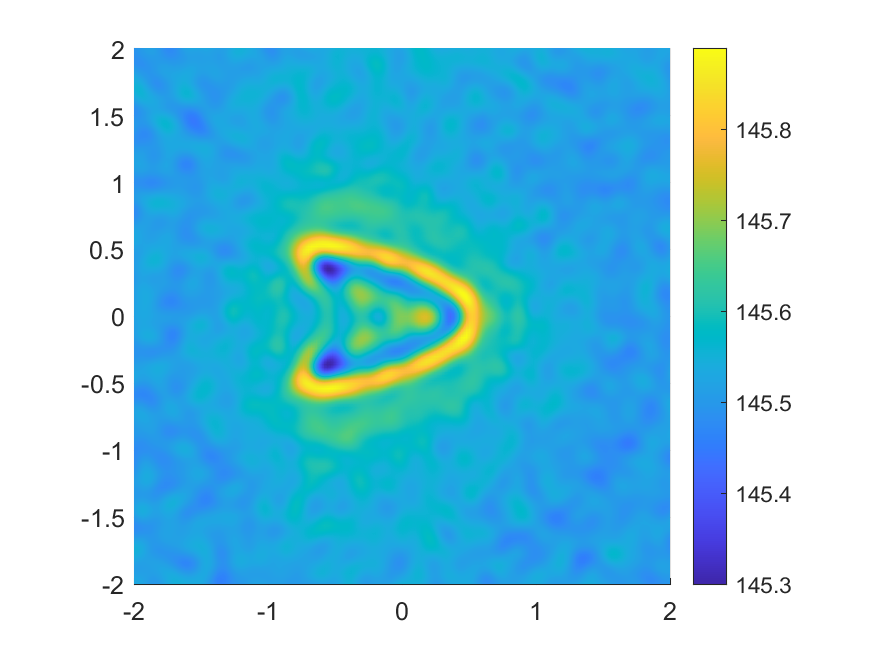}
        \vspace{-0.5em}
        \centering (f)
    \end{minipage}
    
    \caption{Reconstruction results for the kite-shaped rigid impenetrable scatterer obtained by the three proposed sampling methods at $\omega = 18$, $N = 32$. The first row corresponds to the noise-free case ($\delta = 0$), and the second row corresponds to data contaminated by additive noise at relative noise level $\delta = 10\%$. (a), (d): counting-based monotonicity sampling method (Algorithm~\ref{alg:discrete monotonicity}). (b), (e): single-frequency monotonicity spectral sampling method (Algorithm~\ref{alg:spectrum monotonicity}). (c), (f): multi-frequency monotonicity spectral sampling method (Algorithm~\ref{alg:multi-frequency spectrum monotonicity}). }
    \label{fig:Dirichlet kite}
\end{figure}

\subsection{Numerical examples for traction-free impenetrable scatterers}

In this subsection, we present numerical experiments for the case in which $D$ is a traction-free impenetrable scatterer governed by the Neumann boundary condition~\eqref{eq:NeumannBC}, in order to verify the effectiveness of the three proposed algorithms, namely, Algorithm~\ref{alg:discrete monotonicity}, Algorithm~\ref{alg:spectrum monotonicity}, and Algorithm~\ref{alg:multi-frequency spectrum monotonicity}, in reconstructing such scatterers.
In Example~\ref{ex:Neumann discrete}, we consider a kite-shaped scatterer and a three-leaf-shaped scatterer, and examine the reconstruction performance at different frequencies as well as the numerical stability of the counting-based monotonicity sampling method (Algorithm~\ref{alg:discrete monotonicity}). 
In Example~\ref{ex:Neumann spectrum}, we further compare the reconstruction performance and noise robustness of all three proposed algorithms on the Cassini-shaped and kite-shaped scatterers.

\begin{example}\label{ex:Neumann discrete}
We assume that $D$ is a traction-free impenetrable scatterer embedded in a homogeneous isotropic background medium with Lam\'{e} parameters $\lambda = 5$ and $\mu = 5$. A uniform rectangular sampling grid $\mathcal{N}_h$ with step size $h = 0.01$ over $[-2,2]^2$ yields $401\times 401$ grid points. The far-field matrix $\mathbb{F}^n$ is computed via the Nystr\"{o}m method with $N = 32$ incident and observation directions. The threshold parameter is set to $\sigma = 2 \times 10^{-9}$.

We first investigate the reconstruction quality at two frequencies, $\omega = 10$ and $\omega = 30$, for the three-leaf-shaped scatterer. As shown in Figure~\ref{fig:3-leaf}, the reconstruction quality improves substantially at higher frequencies, consistent with the behavior observed for rigid scatterers, since higher frequencies better capture fine geometric details.

Figure~\ref{fig: discrete kite neumann}(a) displays the counting-based indicator $I_{count}(\mathbf{z}_{ij})$ for the kite-shaped scatterer at $\omega = 30$ in the noise-free case. As predicted by Theorem~\ref{Characterofimpenetrable scatterer2}, the indicator attains smaller values inside or near the scatterer than in the exterior domain. The lowest level set provides a sharp approximation to the true boundary.
We next corrupt $\mathbb{F}^n$ with additive Gaussian noise at relative noise level $\delta = 0.1\%$. Figure~\ref{fig: discrete kite neumann}(b) shows that the counting-based indicator becomes compromised, with visible blurring of the boundary compared to the noise-free result. This performance degradation, consistent with the rigid scatterer case, reveals the inherent sensitivity of binary counting to small eigenvalue perturbations, motivating the use of summation-based indicators $I_{sum}$ and $I_{sum}^{\mathrm{MF}}$ examined in Example~\ref{ex:Neumann spectrum}.

\end{example}

\begin{figure}[htbp]
    \centering
    \begin{minipage}{0.45\textwidth}
        \centering
        \includegraphics[width=\textwidth]{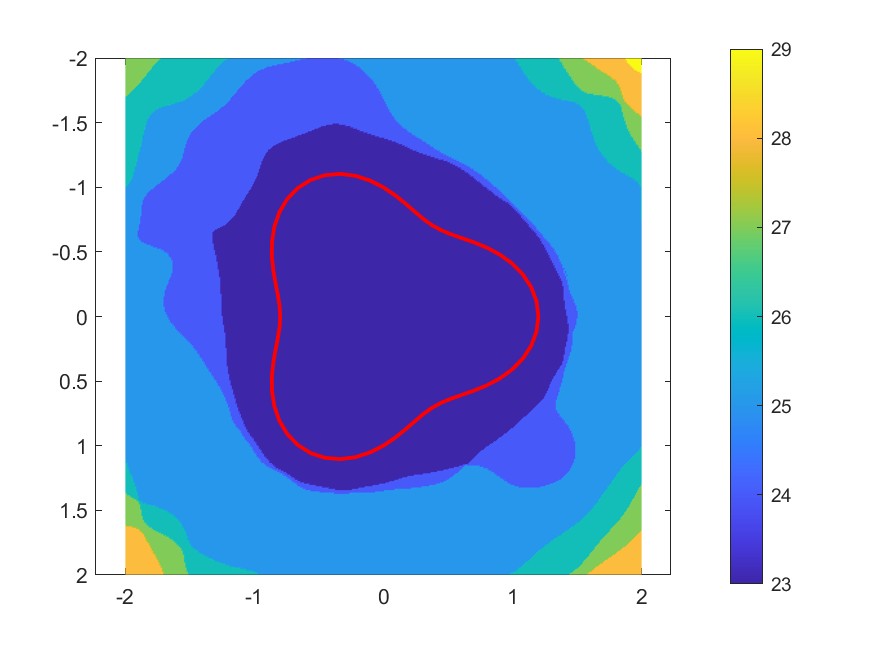}
        \vspace{-0.5em} 
        \centering (a)
    \end{minipage}
    \hfill
   \begin{minipage}{0.45\textwidth}
        \centering
        \includegraphics[width=\textwidth]{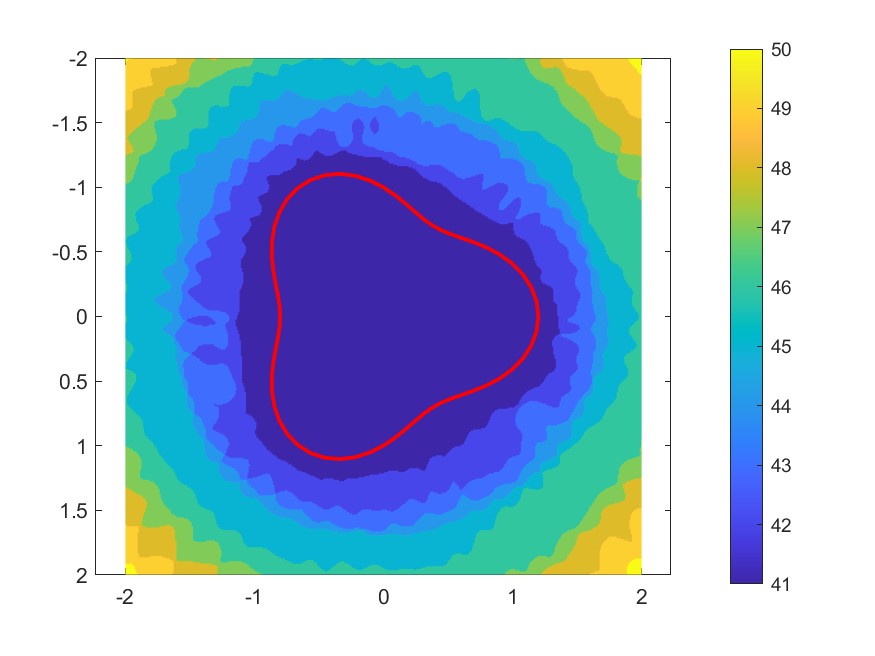}
        \vspace{-0.5em}
        \centering (b)
    \end{minipage}
\caption{Reconstruction results for the three-leaf-shaped traction-free impenetrable scatterer obtained by the counting-based monotonicity sampling method (Algorithm~\ref{alg:discrete monotonicity}), $N = 32$.	(a): $\omega = 10$.	(b): $\omega = 30$. The solid red curve denotes the true boundary  $\partial D$.}
    \label{fig:3-leaf}
\end{figure}

\begin{figure}[htbp]
    \centering
    \begin{minipage}{0.45\textwidth}
        \centering
        \includegraphics[width=\textwidth]{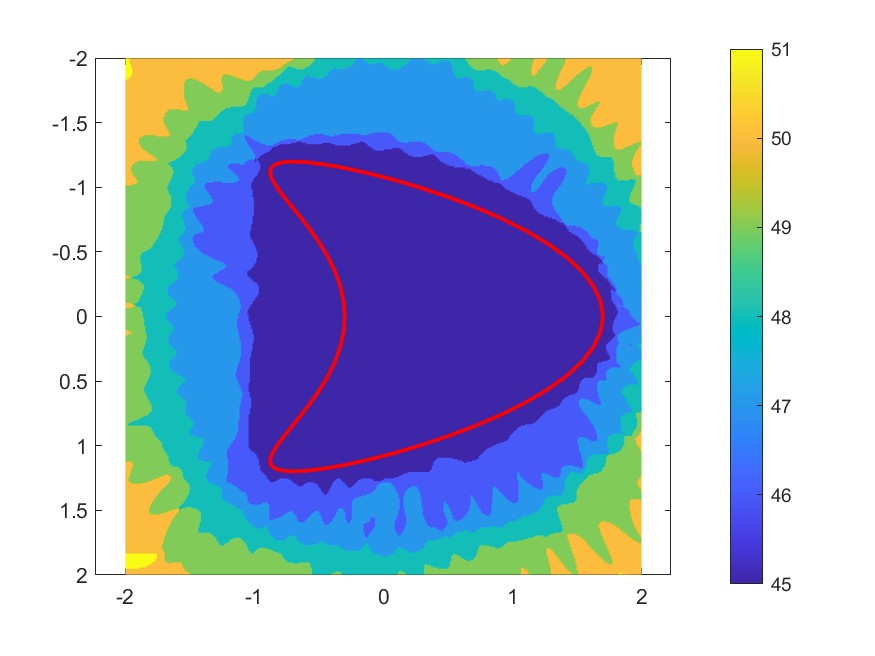}
        \vspace{-0.5em} 
        \centering (a)
    \end{minipage}
    \hfill
   \begin{minipage}{0.45\textwidth}
        \centering
        \includegraphics[width=\textwidth]{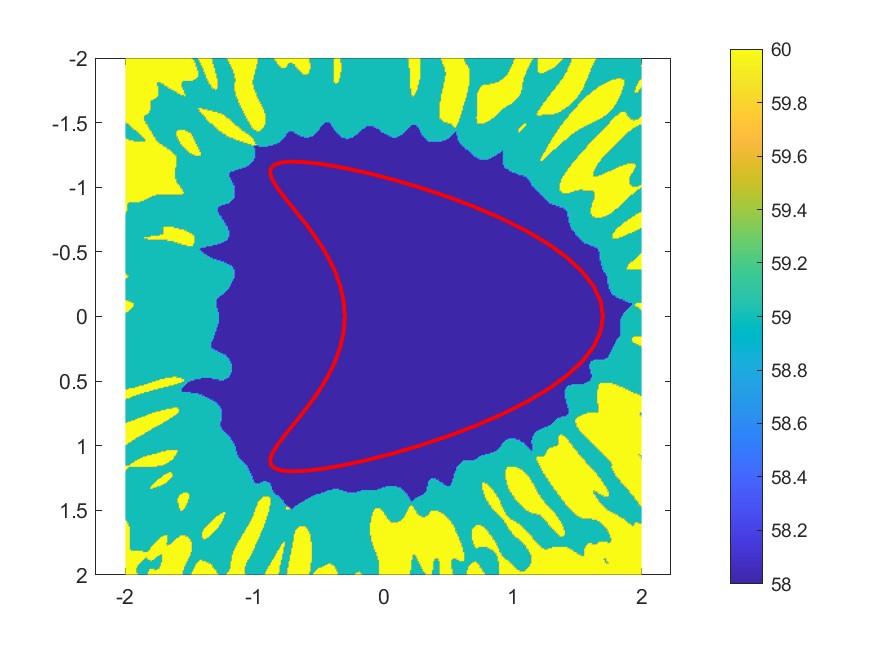}
        \vspace{-0.5em}
        \centering (b)
    \end{minipage}
\caption{Reconstruction results for the three-leaf-shaped traction-free impenetrable scatterer obtained by the counting-based monotonicity sampling method (Algorithm~\ref{alg:discrete monotonicity}), $N = 32$. (a): $\omega = 10$. (b): $\omega = 30$. The solid red curve denotes the true boundary $\partial D$. }
    \label{fig: discrete kite neumann}
\end{figure}

\begin{example}\label{ex:Neumann spectrum}
We test two traction-free impenetrable scatterers: a Cassini-shaped scatterer and a kite-shaped scatterer, which are embedded in a homogeneous isotropic background medium with Lam\'{e} parameters $\lambda = 0.2$ and $\mu = 1$. A uniform rectangular sampling grid $\mathcal{N}_h$ with step size $h = 0.02$ over $[-2,2]^2$ yields $201\times 201$ grid points. We employ $N = 32$ incident and observation directions throughout.

For the counting-based monotonicity sampling method (Algorithm~\ref{alg:discrete monotonicity}), the frequency is $\omega = 40$ with threshold $\sigma = 1\times 10^{-12}$. For the single-frequency monotonicity spectral sampling method (Algorithm~\ref{alg:spectrum monotonicity}), the frequency is also $\omega = 40$. For the multi-frequency monotonicity spectral sampling method (Algorithm~\ref{alg:multi-frequency spectrum monotonicity}), we employ four frequencies $\omega \in \{40, 42, 44, 46\}$ with uniform weights $a_\ell = 1/4$ in the indicator $I_{sum}^{\mathrm{MF}}(\mathbf{z}_{ij})$.
The three indicators $I_{count}$, $I_{sum}$, and $I_{sum}^{\mathrm{MF}}$ are computed and visualized as color-coded maps for each scatterer. Figures~\ref{fig:Neumann cassini} and \ref{fig:Neumann kite} display reconstruction results for the two scatterers, with the first row showing noise-free cases ($\delta = 0$) and the second row showing results under $\delta = 10\%$ additive noise.

In the noise-free case, all three methods successfully identify the location and overall geometry of each scatterer; however, their reconstruction qualities differ substantially. The counting-based monotonicity sampling method produces visible artifacts and yields only a rough approximation of the scatterer boundary (Figures~\ref{fig:Neumann cassini}(a) and \ref{fig:Neumann kite}(a)). This difficulty is attributed to high-curvature concave structures that generate stronger multiple scattering and diffraction effects, introducing perturbations into the far-field data. By contrast, both spectral sampling methods yield significantly cleaner reconstructions by exploiting multiscale spectral information to mitigate these scattering-induced perturbations. The single-frequency indicator $I_{sum}$ concentrates its values near the true boundary $\partial D$ (Figures~\ref{fig:Neumann cassini}(b) and \ref{fig:Neumann kite}(b)), whereas the multi-frequency indicator $I_{sum}^{\mathrm{MF}}$ produces the sharpest and most localized reconstruction by further suppressing artifacts and concentrating its largest values precisely at the boundary (Figures~\ref{fig:Neumann cassini}(c) and \ref{fig:Neumann kite}(c)). This improvement is attributed to multiscale aggregation of spectral information across frequencies, enabling effective resolution of features that challenge the counting-based approach, particularly for geometrically complex boundaries with varying curvature.

Under $\delta = 10\%$ additive noise, the differences in robustness among the three methods become pronounced. The counting-based method deteriorates severely, with noise-induced artifacts dominating the reconstruction (Figures~\ref{fig:Neumann cassini}(d) and \ref{fig:Neumann kite}(d)). The single-frequency spectral method is considerably more stable, preserving main geometric features despite noticeable blurring (Figures~\ref{fig:Neumann cassini}(e) and \ref{fig:Neumann kite}(e)). The multi-frequency method exhibits the strongest robustness, maintaining clear boundary localization and accurate shape recovery even under substantial noise contamination (Figures~\ref{fig:Neumann cassini}(f) and \ref{fig:Neumann kite}(f)). These results demonstrate that the spectral methods' advantage over counting-based approaches is particularly pronounced for geometrically complex scatterers with concave boundaries, where the multiscale aggregation effectively resolves features that challenge the binary counting mechanism.

Among the three algorithms, the multi-frequency monotonicity spectral sampling method (Algorithm~\ref{alg:multi-frequency spectrum monotonicity}) achieves the best overall performance in all tested scenarios. In particular, for the Cassini-shaped scatterer, the reconstructed boundary in Figure~\ref{fig:Neumann cassini}(c) nearly coincides with the true boundary $\partial D$. Even at a noise level of $\delta = 10\%$, the multi-frequency method reliably distinguishes the concave boundary features (see Figures~\ref{fig:Neumann cassini} (f) and~\ref{fig:Neumann kite}(f)), whereas the counting-based and single-frequency methods deteriorate significantly under the same conditions (see Figures~\ref{fig:Neumann cassini}(d), \ref{fig:Neumann cassini}(e), \ref{fig:Neumann kite}(d) and \ref{fig:Neumann kite}(e)). Overall, the numerical results confirm that the counting-based monotonicity sampling method is directly derived from Theorem \ref{Characterofimpenetrable scatterer2}. but highly sensitive to noise; the single-frequency monotonicity spectral sampling method provides a marked improvement in numerical stability and reconstruction accuracy; and the multi-frequency monotonicity spectral sampling method consistently delivers the most accurate and noise-robust reconstructions across all tested scatterer geometries and noise levels.
\end{example}

\begin{figure}[htbp]
    \centering
    
    \begin{minipage}{0.32\textwidth}
        \centering
        \includegraphics[width=\linewidth]{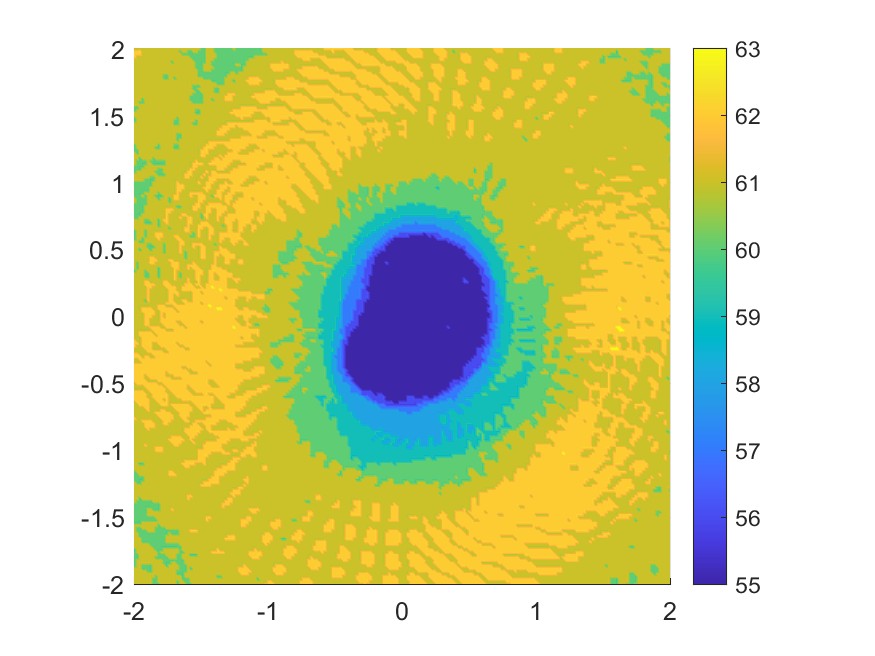}
        \vspace{-0.5em} 
        \centering (a)
    \end{minipage}
    \hfill
   \begin{minipage}{0.32\textwidth}
        \centering
        \includegraphics[width=\linewidth]{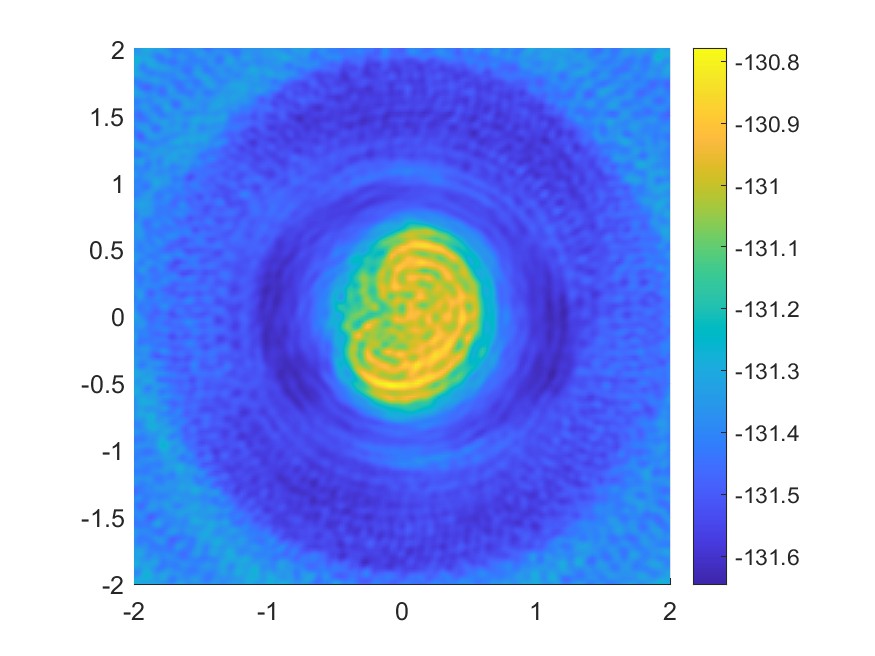}
        \vspace{-0.5em}
        \centering (b)
    \end{minipage}
    \hfill
    \begin{minipage}{0.32\textwidth}
        \centering
        \includegraphics[width=\linewidth]{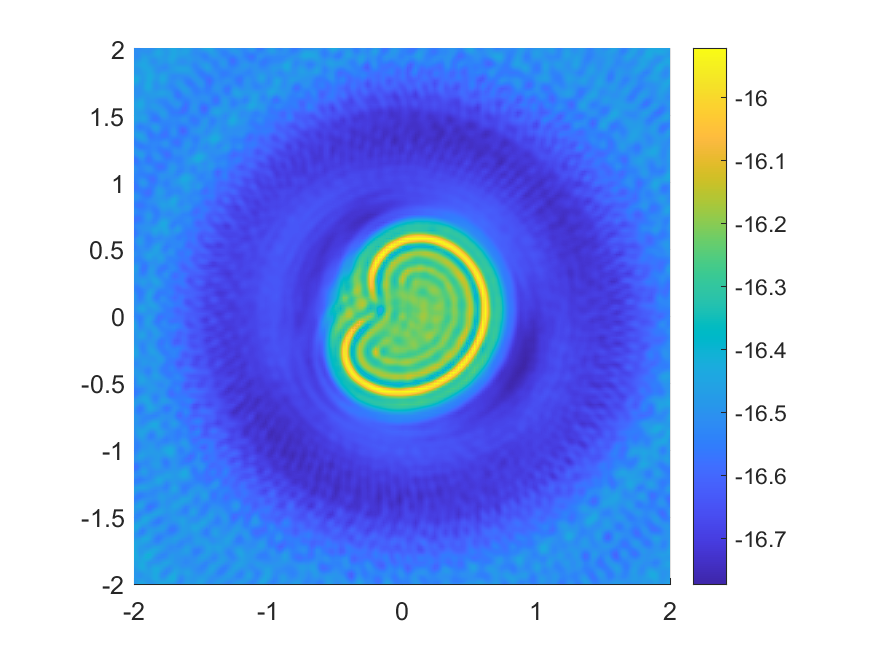}
        \vspace{-0.5em}
        \centering (c)
    \end{minipage}

    \vspace{1em}

    \begin{minipage}{0.32\textwidth}
        \centering
        \includegraphics[width=\linewidth]{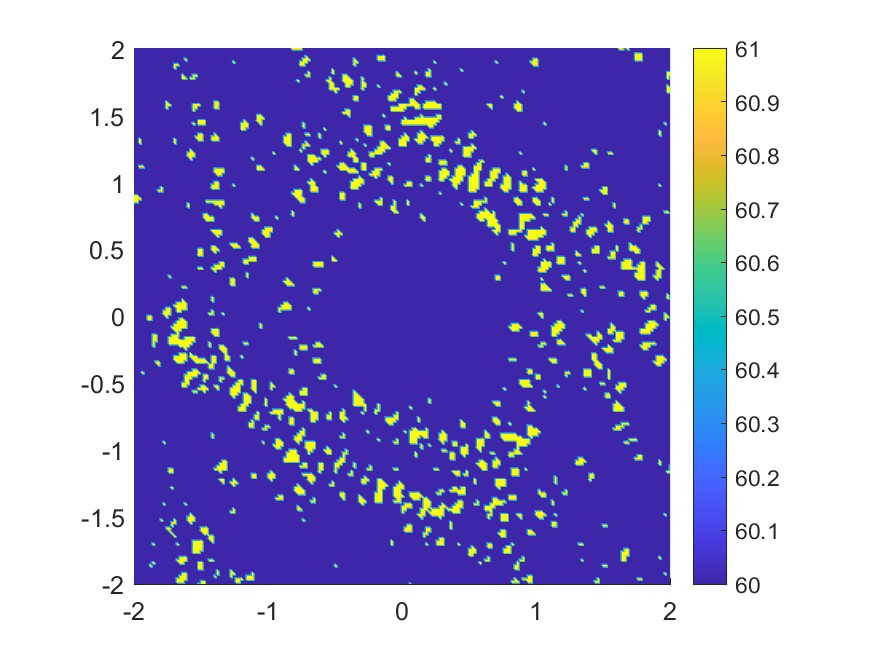}
        \vspace{-0.5em}
        \centering (d)
    \end{minipage}
    \hfill
    \begin{minipage}{0.32\textwidth}
        \centering
        \includegraphics[width=\linewidth]{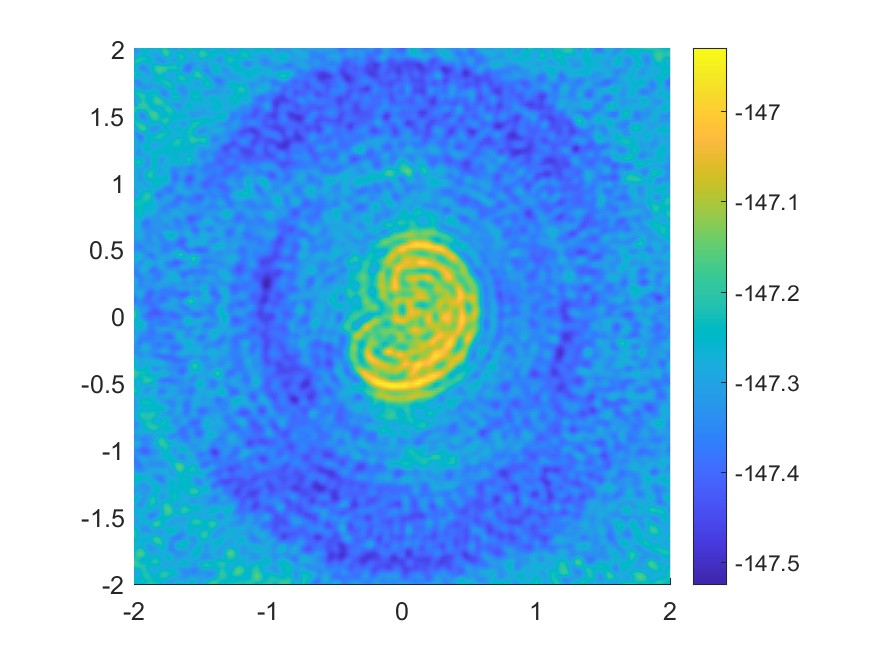}
        \vspace{-0.5em}
        \centering (e)
    \end{minipage}
    \hfill
    \begin{minipage}{0.32\textwidth}
        \centering
        \includegraphics[width=\linewidth]{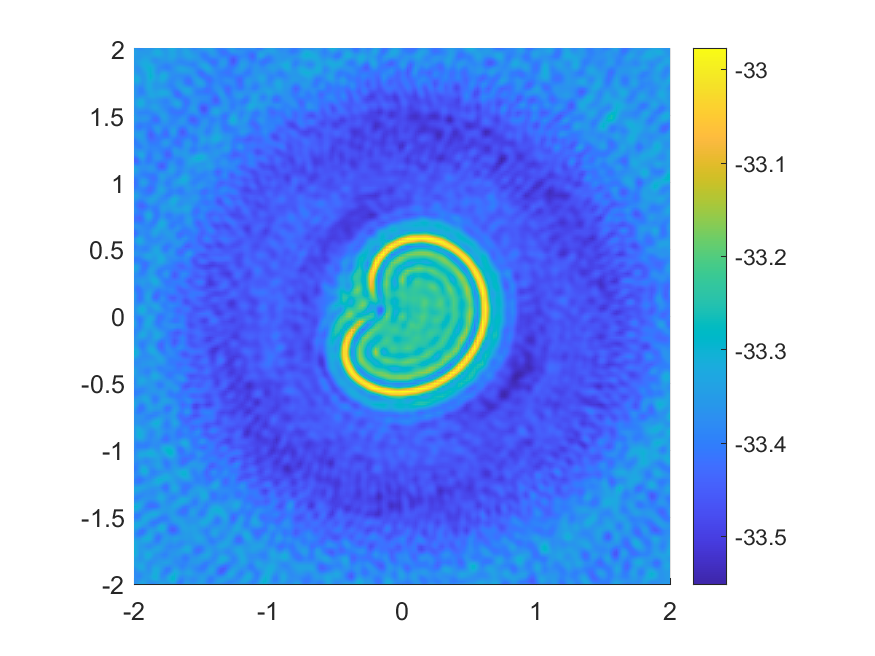}
        \vspace{-0.5em}
        \centering (f)
    \end{minipage}
    \caption{ Reconstruction results for the Cassini-shaped traction-free impenetrable scatterer obtained by the three proposed sampling methods at $\omega = 40$, $N = 32$. The first row corresponds to the noise-free case ($\delta = 0$), and the second row corresponds to data contaminated by additive noise at relative noise level $\delta = 10\%$. (a), (d): counting-based monotonicity sampling method (Algorithm~\ref{alg:discrete monotonicity}). (b), (e): single-frequency monotonicity spectral sampling method (Algorithm~\ref{alg:spectrum monotonicity}). (c), (f): multi-frequency monotonicity spectral sampling method (Algorithm~\ref{alg:multi-frequency spectrum monotonicity}). }
    \label{fig:Neumann cassini}
\end{figure}
\begin{figure}[htbp]
    \centering
    
    \begin{minipage}{0.32\textwidth}
        \centering
        \includegraphics[width=\linewidth]{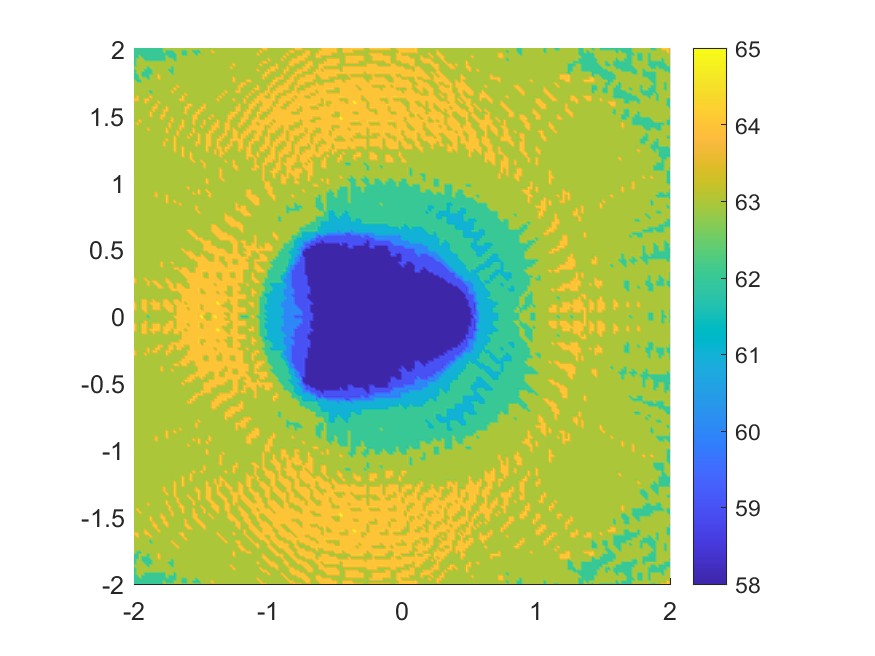}
        \vspace{-0.5em} 
        \centering (a)
    \end{minipage}
    \hfill
   \begin{minipage}{0.32\textwidth}
        \centering
        \includegraphics[width=\linewidth]{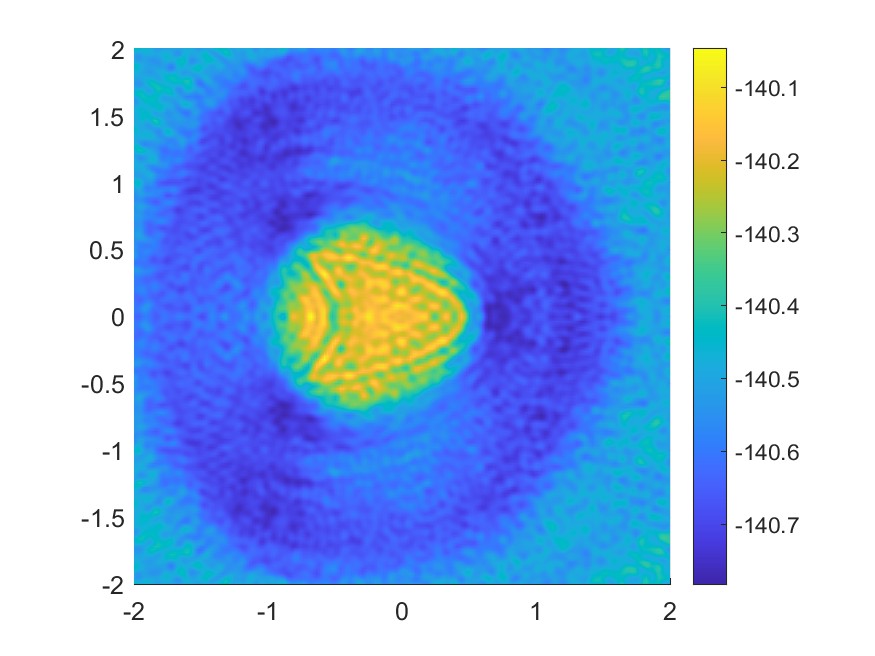}
        \vspace{-0.5em}
        \centering (b)
    \end{minipage}
    \hfill
    \begin{minipage}{0.32\textwidth}
        \centering
        \includegraphics[width=\linewidth]{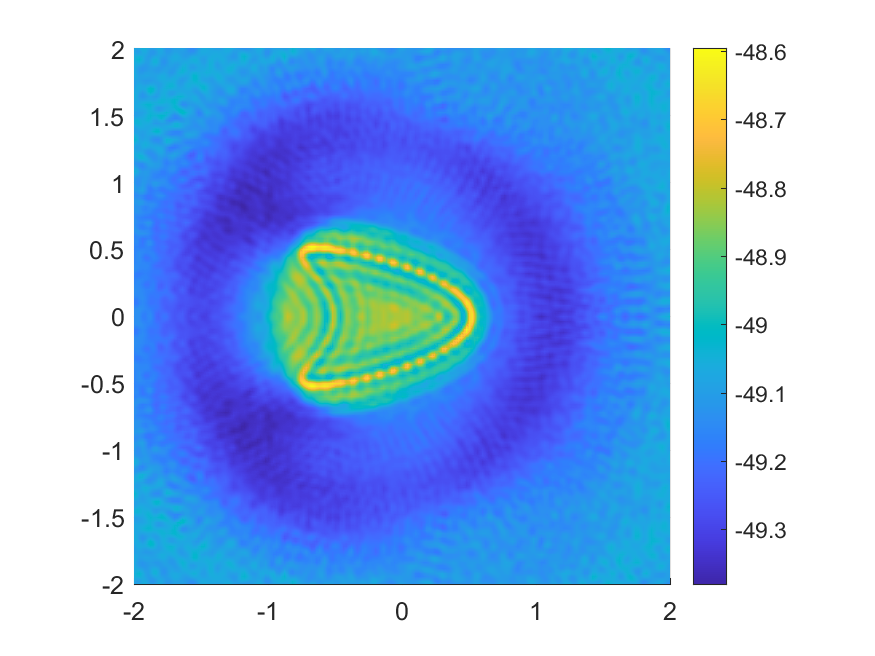}
        \vspace{-0.5em}
        \centering (c)
    \end{minipage}

    \vspace{1em}

    \begin{minipage}{0.32\textwidth}
        \centering
        \includegraphics[width=\linewidth]{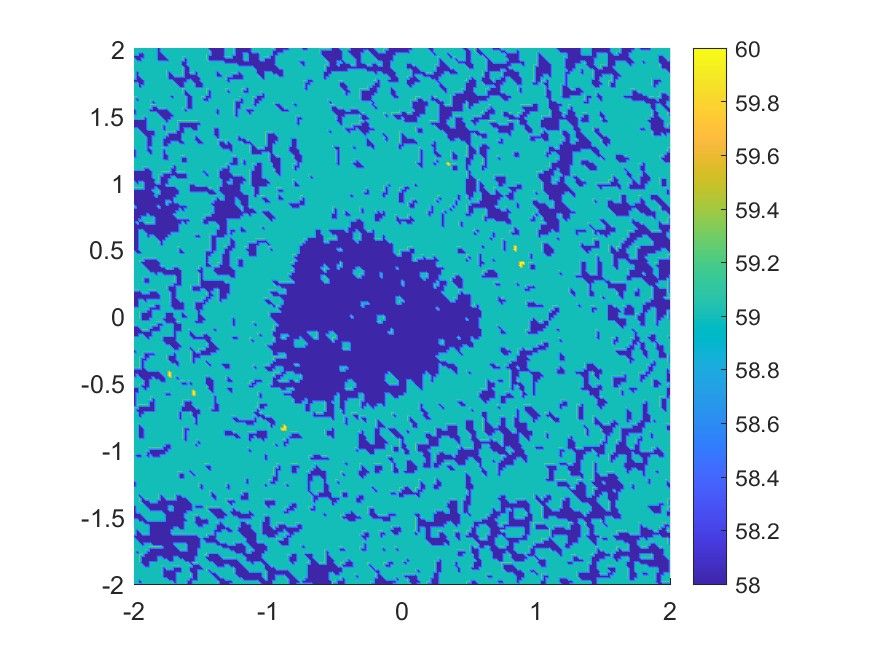}
        \vspace{-0.5em}
        \centering (d)
    \end{minipage}
    \hfill
    \begin{minipage}{0.32\textwidth}
        \centering
        \includegraphics[width=\linewidth]{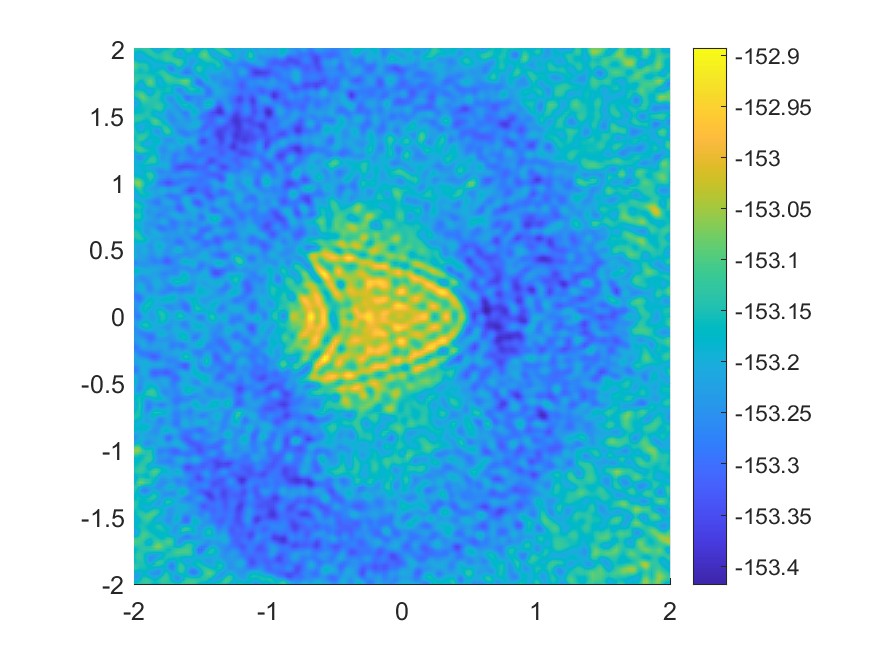}
        \vspace{-0.5em}
        \centering (e)
    \end{minipage}
    \hfill
    \begin{minipage}{0.32\textwidth}
        \centering
        \includegraphics[width=\linewidth]{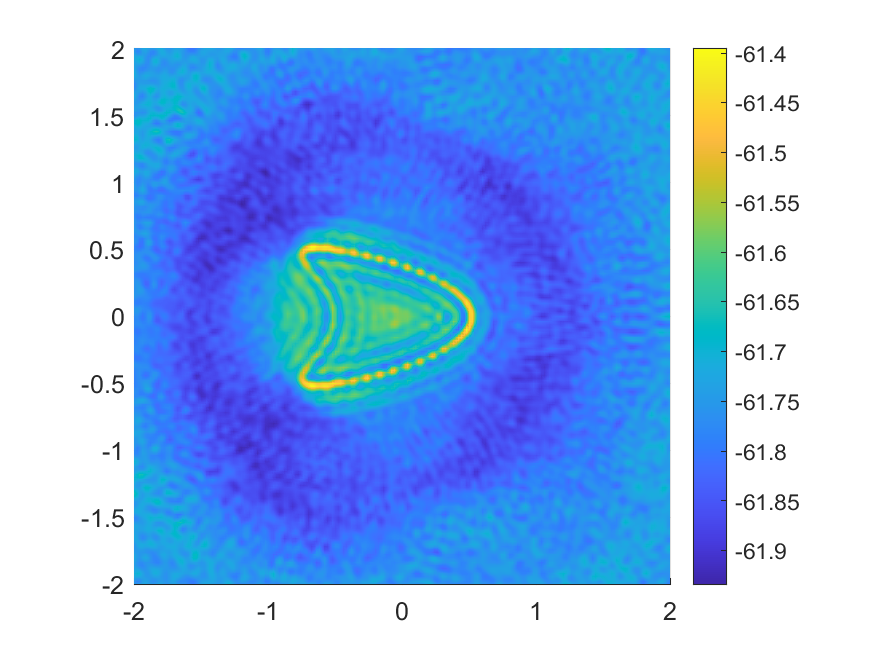}
        \vspace{-0.5em}
        \centering (f)
    \end{minipage}
    \caption{Reconstruction results for the kite-shaped traction-free impenetrable scatterer obtained by the three proposed sampling methods at $\omega = 40$, $N = 32$. The first row corresponds to the noise-free case ($\delta = 0$), and the second row corresponds to data contaminated by additive noise at relative noise level $\delta = 10\%$. (a), (d): counting-based monotonicity sampling method (Algorithm~\ref{alg:discrete monotonicity}). (b), (e): single-frequency monotonicity spectral sampling method (Algorithm~\ref{alg:spectrum monotonicity}). (c), (f): multi-frequency monotonicity spectral sampling method (Algorithm~\ref{alg:multi-frequency spectrum monotonicity}). }
    \label{fig:Neumann kite}
\end{figure}

In numerical examples presented above, the counting-based monotonicity sampling method (Algorithm~\ref{alg:discrete monotonicity}) produces an indicator $I_{count}$ that attains its smallest values inside and near the impenetrable scatterer. By contrast, the single-frequency and multi-frequency monotonicity spectral sampling methods (Algorithms~\ref{alg:spectrum monotonicity} and~\ref{alg:multi-frequency spectrum monotonicity}) yield indicators $I_{sum}$ and $I_{sum}^{\mathrm{MF}}$ that enable a more accurate identification of the shape of the scatterer. In particular, even in the presence of noise, the regions of the largest indicator values in the reconstructed images still delineate the true scatterer $D$ with satisfactory accuracy.

Furthermore, the newly proposed monotonicity spectral sampling methods (Algorithms~\ref{alg:spectrum monotonicity} and~\ref{alg:multi-frequency spectrum monotonicity}) demonstrate not only significantly greater numerical robustness with respect to noise than the counting-based monotonicity sampling method (Algorithm~\ref{alg:discrete monotonicity}), but also higher accuracy in reconstructing concave parts of the scatterer boundary. These advantages confirm that the spectral summation-based indicators $I_{sum}$ and $I_{sum}^{\mathrm{MF}}$ constitute effective and practically reliable improvements over the classical counting-based criterion $I_{count}$ within the monotonicity-based reconstruction framework.

\section*{Acknowledgment}
The work of H. Diao is supported by National Natural Science Foundation of China  (No. 12371422), 2025 National Foreign Experts Program (Grant No. D20250157) and the Fundamental Research Funds for the Central Universities, JLU.

\end{document}